\documentclass[reqno,10pt, centertags]{amsart}
\usepackage{amssymb,upref,esint,color}
\usepackage{hyperref}
\usepackage{etoolbox}
\newcommand*{\mailto}[1]{\href{mailto:#1}{\nolinkurl{#1}}}

\makeatletter
\usepackage{verbatim}
\usepackage{amscd}
\usepackage{framed}
\numberwithin{equation}{section}
\input epsf

\newtheorem{theorem}{Theorem}[section]
\newtheorem{definition}[theorem]{Definition}
\newtheorem{lemma}[theorem]{Lemma}
\newtheorem{corollary}[theorem]{Corollary}
\newtheorem{proposition}[theorem]{Proposition}
\numberwithin{equation}{section}
\theoremstyle{definition}
\newtheorem{example}[theorem]{Example}

\begin{document}

\title[Complex analysis of symmetric operators. II]
{Complex analysis of symmetric operators. II: entire operators with deficiency index 1 }

\author[Yicao Wang]{Yicao Wang}
\address
{School of Mathematics, Hohai University, Nanjing 210098, China}

%\thanks{This study is supported by the Natural Science Foundation of Jiangsu Province (BK20150797).}
\baselineskip= 20pt
\begin{abstract}This paper is a continuation of our previous work \cite{wang2024complex}. It mainly deals with entire operators $T$ with deficiency index 1 \emph{systematically} from the complex-geometric viewpoint proposed in \cite{wang2024complex}. We pay special attention to the characteristic line bundle $F$ of $T$. We investigate its curvature in detail and demonstrate how it is connected to the height function of $T$ and to the distribution of zeros of elements in the canonical model Hilbert space which consists of certain holomorphic sections of $F$. This study is applied to an indeterminate Hamburger moment problem to show the growth property of the associated Jacobi operator coincides with that defined in terms of entries of the Nevanlinna matrix. We also show how various functional models for $T$ can be derived from our canonical model by restricting $F$ to certain subsets of $\mathbb{C}$ and choosing suitable trivializations. This makes the interrelationships among these models much more transparent. By introducing the mean type of a generic non-self-adjoint extension and using the de Branges-Rovnyak model, we show the mean type is the only obstruction to completeness of such an extension. We also prove that the measure of incomplete extensions is zero. Some other new results and new proofs of old results are also included.
\end{abstract}
\maketitle
%\newpage

%\begin{quotation}
%\textbf{Ghosts are the easiest to draw}.--------\texttt{Anonymity}

%\end{quotation}
\tableofcontents

\section{Introduction}
In \cite{wang2024complex}, the author has set up a complex geometric formalism to deal with the extension theory of symmetric operators initiated by von Neumann in late 1920s. A central object there is the Weyl curve $W_T(\lambda)$ of a simple symmetric operator $T$, encoding all the unitarily invariant information of $T$. If $T$ is of deficiency indices $(n_+,n_-)$ with $0<n_\pm<\infty$, loosely speaking, then $W_T(\lambda)$ is a holomorphic map from the upper and lower half-planes
$\mathbb{C}_+\cup \mathbb{C}_-$ to \[Gr(n_+, n_++n_-)\cup Gr(n_-,n_++n_-),\]
where $Gr(n,m)$ is the Grassmannian of $n$-dimensional subspaces in $\mathbb{C}^m$ ($1\leq n<m$). In particular if $n_+=n_-=n$, then $W_T(\lambda)$ maps $\mathbb{C}_+\cup \mathbb{C}_-$ into $Gr(n,2n)$. Consequently the author was motivated in \cite{wang2024complex} to single out a special kind of symmetric operators with deficiency indices $(n,n)$, whose Weyl curves have an analytic continuation on the whole complex plane $\mathbb{C}$. Such symmetric operators were called \emph{entire} simply because their Weyl curves are entire curves in $Gr(n,2n)$. For simplicity, we just call these entire operators with deficiency index $n$. One of the merits of this geometric framework is that we can interpret generic boundary value problems in terms of the value distribution theory of entire curves. The reader should also be warned that the word "entire operator" is also used in the existing literature in a narrower sense and our entire operators are traditionally called \emph{regular}.

However, the value distribution theory of entire curves itself is far from finished. One reason for that is the target manifold $M$ could be of a much higher dimension so that there is too much space for the curve to stretch\footnote{In our case, the target manifold is $Gr(n,2n)$, with complex dimension $n^2$. }. The only most well-understood case is $M=\mathbb{CP}^1=Gr(1,2)$, i.e., the projective line,\footnote{The theory works well for general compact Riemann surfaces, but the content for $\mathbb{C}P^1$ is the most abundant and interesting.} and the curve is then essentially a meromorphic function on $\mathbb{C}$. Our focus in this paper is thus entire operators with deficiency index 1, for which the full power of value distribution theory and the mechanism of function theory of one complex variable may be used to push-forward the study in \cite{wang2024complex}. Therefore the paper can be viewed as an elaboration of ideas presented in \cite{wang2024complex} for entire operators with deficiency index 1. This is of course also the first step towards the general case.

Though not typical in the above sense and without a systematic study, entire operators with deficiency index 1 are the most well-studied case in the community of spectral theorists. They are so intriguing for their connections with several other interesting topics, e.g., spectral analysis of 1-dimensional Schrodinger operators and canonical systems, the Hamburger moment problems, the theory of de Branges spaces of entire functions, the theory of model subspaces, sampling theory in signal processing, etc. We just refer the interested reader to \cite{gorbachuk2012mg, silva2007applications, silva2010spectral, silva2013branges, makarov2005meromorphic, martin2010symmetric} and the bibliography therein. Maybe it's also worth mentioning that, some people trying to prove the Riemann Conjecture are also interested in such operators, which may provide the conjectural Hilbert-Polya operator \cite{lagarias2006hilbert}. People often study these relevant topics from different motivations and with different tools. Thus the other purpose of the paper is to \emph{unify} certain materials scattered in the literature from the \emph{single} general viewpoint of \cite{wang2024complex}. We believe that by placing entire operators with deficiency index 1 front and center, our treatment has made the connections more compact and conceptually clearer. However completeness is beyond our ambition and our choice of subject matter is heavily affected by our attempt to extend the viewpoint presented in \cite{wang2024complex} further. In particular, we insist that generic extensions should be handled altogether whether they are self-adjoint or not and via this even self-adjoint extensions can be understood better. Sometimes we give new proofs of old results, but our emphasis is often their reformulation and how it can be viewed from our new formalism. We think this is necessary if we attempt to extend some of the results to the case with higher deficiency index.

The paper is organized as follows.

In \S~\ref{P} we specialize the general picture of \cite{wang2024complex} to the case of simple symmetric operators with deficiency indices $(1,1)$. Compared with the general study in \cite{wang2024complex}, some additional details are included and \S\S~\ref{bvp} also contains a few new immediate results.

\S~\ref{line} investigates the geometry of the characteristic line bundle $F$ (of the second kind in the sense of \cite{wang2024complex}). The focus is certain natural local frames of $F$ and the curvature function $\omega(\lambda)$ expressed in terms of them. The curvature function was only slightly touched in \cite{wang2024complex}, but here we consider it in some detail and in fact try to explore its spectral theoretical meaning in the whole paper. In particular, we investigate the restriction of $\omega(\lambda)$ on the real line and prove that it's a complete unitary invariant of $T$ (Thm.~\ref{complete}). We also demonstrate how this invariant controls the growth property of $T$, one of the main topics of the paper (Prop.~\ref{Schw}, Prop.~\ref{compare}, etc).

In \S~\ref{char}, we collect elements on the structure of three characteristic functions. To some extent, all these functions are only function-theoretic representations of the Weyl curve $W_T(\lambda)$ and as a consequence function theory becomes an efficient tool in the study of entire operators. We pay special attention to implications of these functions in our formalism. Thus several quick conclusions are derived. In particular, we provide an example of Nevanlinna exceptional boundary condition whose Nevanlinna defect lies in $(0,1)$ (Example \ref{defect}). Such an example was missing in \cite{wang2024complex}.

 \S~\ref{S4} is devoted to some functional models in the literature. We show how these models can be recovered from our canonical model in \S\S~\ref{model}. Thus our model has provided a unified perspective towards functional models, which makes the interconnections of these models much more transparent.

As an application of investigations in the previous two sections, in \S~\ref{com} we solve the completeness problem in the context of generic non-self-adjoint extensions of entire operators with deficiency index 1. We define the notion of mean type of such an extension and prove that it is the only obstruction to completeness (Thm.~\ref{mean}). Taking the complexity of the general completeness problem into account, this result is surprisingly simple. We also show that almost all (in the sense of measure theory) non-self-adjoint extensions are complete. Furthermore we investigate when the normalized eigenvectors of a non-self-adjoint extension form a Riesz basis under the assumption that all eigenvalues are simple (Thm.~\ref{Riesz}).

\S~\ref{growth} is devoted to the growth aspect of $T$. Unlike \cite{wang2024complex}, our attention here is to use a line bundle version of First Main Theorem in value distribtion theory to show how the curvature of the characteristic line bundle $F$ controls the distribution of zeros of elements in the canonical model space of $T$ (Prop.~\ref{first} and Prop.~\ref{esti}). We show that the characteristic function $\mathrm{T}_F(r)$ of $F$ is essentially the height function defined in \cite{wang2024complex} in terms of value distribution theory of entire curves (Thm.~\ref{equal}). The rest of the section mainly deals with entire operators whose Weyl order is $\leq 1$. Such operators are frequently encountered in practice, e.g., the indeterminate Hamburger moment problems.

In the last section \S~\ref{Hmoment}, as an application of ideas presented in the previous sections we revisit the operator-theoretic approach to indeterminate Hamburger moment problems. The growth property of such a moment problem has been noted already in M. Riesz's work, but it was not until the paper \cite{berg1994order} that people realized that the entries of the associated  Nevanlinna matrix all have the same order and type. These numbers are then defined to be the type and order of the moment problem. We prove that the characteristic function $\mathrm{T}_F(r)$ associated with the underlying Jacobi operator coincides with the Nevanlinna characteristic functions of the entries of the Nevanlinna matrix up to the term $O(\ln r)$ (Thm.~\ref{moment1} and Thm.~\ref{moment2}). Thus our $\mathrm{T}_F(r)$ actually provides a canonical and unitarily invariant quantity to measure the growth of the moment problem. This result not only is of its own interest, but also suggests how the growth of an indeterminate \emph{matrix} Hamburger moment problem may be measured. We shall turn to this question somewhere else. In the same section, we prove that the Jacobi operator is completely determined by its Weyl class (Thm.~\ref{unique}). There are also some other observations, shedding new light on the indeterminate Hamburger moment problems.
\begin{center}
\textbf{Conventions and notations}
\end{center}

The imaginary unit will be denoted by $\mathrm{i}$.

All complex Hilbert spaces $(H, (\cdot, \cdot))$ are separable. We adopt the convention that the inner product $(\cdot, \cdot)$ is linear in the first variable and conjugate-linear in the second. The induced norm will be written as $\|\cdot\|$. If necessary, the inner product is also denoted by $(\cdot, \cdot)_H$ to emphasize the underlying Hilbert space and to distinguish it from an ordered pair. We use $\oplus$ (resp.~$\oplus_\bot$) to denote a topological (resp.~orthogonal) direct sum. For a subspace $V$, its orthogonal complement is $V^\bot$. The zero vector space will be just denoted by $0$. For a collection of subspaces $V_\alpha\subset H$, $\bigvee_\alpha V_\alpha$ means the space of finite linear combinations of elements from these $V_\alpha$'s.

The identity operator on $H$ will be just denoted by $Id$ if the underlying space is clear from the background. For a constant $c$, $c\times Id$ is often written simply as $c$ if no confusion arises. $D(A)$ is the domain of an operator $A$, and $\textup{ker}A$ (resp. $\textup{Ran}A$) is the kernel (resp. range) of $A$. The spectrum of a (bounded or unbounded) operator $A$ is denoted by $\sigma(A)$ and the resolvent set of $A$ by $\rho(A)$.

On a Hermitian vector bundle $E$ over a manifold $M$, the Hermitian structure at $p\in M$ will be denoted by $(\cdot, \cdot)_p$. The length of $v\in E_p$ is denoted by $|v|_p$ and for brevity the subscript $p$ may sometimes be omitted.

If $\lambda\in \mathbb{C}$, then $\Re \lambda$ and $\Im \lambda$ are the real and imaginary parts of $\lambda$ respectively. If $f(x)$ is a function, $O(f)$ represents a function $g(x)$ such that $|g/f|$ is bounded w.r.t. the underlying limiting process of the independent variable $x$. We shall use the words "holomorphic" and "analytic" interchangeably.

\section{Preliminaries}\label{P}
In this section, we collect some material on simple symmetric operators with deficiency indices $(1,1)$ and introduce the special class of entire operators with deficiency index 1. The presentation follows closely our previous work \cite[\S~3, \S~4, \S~9]{wang2024complex}, which contains more details.
\subsection{Symmetric operator with deficiency indices (1,1) and its Weyl curve}
In a separable infinite-dimensional complex Hilbert space $H$, a symmetric operator $T$ in $H$ is a linear operator defined on a subspace $D(T)\subset H$ such that $(Tx,y)=(x,Ty)$ for all $x,y\in D(T)$. We allow $D(T)$ to be not dense in $H$ but always assume that $T$ is closed, i.e., the graph $A_T$ of $T$ is a closed subspace of $\mathbb{H}:=H\oplus_\bot H$. A symmetric operator can always be decomposed as the orthogonal direct sum of its self-adjoint part and its completely non-self-adjoint part. If the nontrivial self-adjoint part is absent, $T$ is called simple. Throughout the paper, we only consider simple symmetric operators.

$\mathbb{H}$ is a strong symplectic Hilbert space with a canonical strong symplectic structure, i.e., for $x=(x_1,x_2)\in \mathbb{H}$, $y=(y_1,y_2)\in \mathbb{H}$,
\[[x,y]_c=(x_2, y_1)_H-(x_1, y_2)_H.\]
For a simple symmetric operator $T$, $A_T$ is isotropic in $\mathbb{H}$, i.e., the restriction of $[\cdot, \cdot]_c$ on $A_T$ vanishes. The symplectic orthogonal complement $A_T^{\bot_s}$
is
\[A_T^{\bot_s}:=\{x\in \mathbb{H}|[x,y]_c=0,\ \forall y\in A_T\}.\]
Clearly, $A_T^{\bot_s}$ is closed and $A_T\subset A_T^{\bot_s}$. Then the quotient space $\mathcal{B}_T:=A_T^{\bot_s}/A_T$ acquires a canonical strong symplectic Hilbert space structure whose strong symplectic structure is defined via
\[[[x],[y]]_q=[x,y]_c,\ \forall\ [x],[y]\in \mathcal{B}_T.\]
$-\mathrm{i}[\cdot,\cdot]_q$ is actually an indefinite and non-degenerate inner product on $\mathcal{B}_T$. Let $(n_+,n_-)$ be its signature. Then $(n_+,n_-)$ are precisely the deficiency indices of $T$. From now on and for the rest of the paper, we assume $n_+=n_-=1$. Then $\mathcal{B}_T$ is necessarily of complex dimension $2$.

An extension of $T$ is an intermediate subspace between $A_T$ and $A_T^{\bot_s}$. Thus all extensions of $T$ are parameterized by the Grassmannian $\mathcal{M}_0$ of subspaces in $\mathcal{B}_T$. $\mathcal{M}_0$ consists of three topological components, i.e., two isolated points represent the two trivial extensions $A_T$ and $A_T^{\bot_s}$ respectively and the third component $\mathcal{M}$ parameterizes all extensions $\tilde{A}$ such that $\textup{dim}_\mathbb{C}(\tilde{A}/A_T)=1$. Thus $\mathcal{M}$ can be holomorphically identified with $\mathbb{CP}^1$. We call extensions parameterized by points in $\mathcal{M}$ \emph{generic}. In particular, extensions parameterized by Lagrangian subspaces are called self-adjoint extensions.

A generic extension $\tilde{A}$ of $T$ parameterized by $m\in \mathcal{M}$ is often the graph of an operator $\tilde{T}$, and we can identify $\tilde{A}$ with $\tilde{T}$. Actually, in this sense it is known that if $T$ is densely defined, then all generic extensions are operators; when $T$ is not densely defined, $\overline{D(T)}$ must have codimension 1 in $H$ and there is one and only one $m\in \mathcal{M}$ whose corresponding extension is not an operator but a self-adjoint relation. If we want to emphasize the underlying operator associated to $m\in \mathcal{M}$, we denote the extension by $T_m$.

For each $\lambda\in \mathbb{C}$, let $\mathcal{W}_\lambda:=\{(x,\lambda x)\in \mathbb{H}|x\in H\}$. Then for each $\lambda\in \mathbb{C}_\pm$, $W_\lambda:=\mathcal{W}_\lambda\cap A_T^\bot$ is of complex dimension 1\footnote{When $T$ is densely defined, $W_\lambda$ is essentially $\textup{ker}(T^*-\lambda)$.}. Let $W_T(\lambda)\subset \mathcal{B}_T$ be the image of $W_\lambda$ under the quotient map from $A_T^\bot$ to $\mathcal{B}_T$. $W_T(\lambda)$ is still of complex dimension 1 and thus we obtain a map $W_T(\lambda): \mathbb{C}_+\cup \mathbb{C}_-\rightarrow \mathcal{M}$,
called the (two-branched) \emph{Weyl curve} of $T$. $W_T(\lambda)$ is actually holomorphic. Even more, for $\lambda\in \mathbb{C}_+$ (resp. $\lambda\in \mathbb{C}_-$), the restriction of $-\mathrm{i}[\cdot, \cdot]_q$ (resp. $\mathrm{i}[\cdot, \cdot]_q)$ on $W_T(\lambda)$ is positive-definite (resp. negative-definite) and $W_T(\overline{\lambda})=(W_T(\lambda))^{\bot_s}$. If furthermore $W_T(\lambda)$ can be analytically extended to be a holomorphic map from $\mathbb{C}$ to $\mathcal{M}$, $T$ is then called \emph{entire}. For later convenience, the set of entire operators with deficiency index 1 in $H$ will be denoted by $\mathfrak{E}_1(H)$.

The Weyl curve $W_T(\lambda)$ can be described in more familiar terms by using proper "coordinate charts" on $\mathcal{M}$. We now explain the details.

Two strong sympletic Hilbert spaces with the same signature are always isomorphic in the sense that there is a topological linear isomorphism between them, preserving the strong symplectic structures. We can choose a certain "standard" strong symplectic Hilbert space $\mathbb{G}$ and an isomorphism $\Phi$ between $\mathcal{B}_T$ and $\mathbb{G}$. Then all objects defined in terms of the strong symplectic structure $[\cdot,\cdot]_q$ on $\mathcal{B}_T$ transfer to the level of $\mathbb{G}$.

We can view $\mathbb{C}$ as a 1-dimensional Hilbert space with the standard Euclidean inner product. Then there are two natural ways to equip $\mathbb{G}:=\mathbb{C}\oplus_\bot \mathbb{C}=\mathbb{C}^2$ with a strong symplectic structure. If $x=(x_1,x_2)\in \mathbb{G}$, $y=(y_1,y_2)\in \mathbb{G}$, then the first strong symplectic structure is defined by
\[[x,y]_1=\mathrm{i}x_1\overline{y_1}-\mathrm{i}x_2\overline{y_2}\]
while the second is defined by
\[[x,y]_2=x_2\overline{y_1}-x_1\overline{y_2}.\]
For $(x_1,x_2)\in \mathbb{G}$, the map $\beta((x_1,x_2))=(\frac{x_1-x_2}{\mathrm{i}\sqrt{2}}, \frac{x_1+x_2}{\sqrt{2}})\in \mathbb{G}$ is an isomorphism from $(\mathbb{G}, [\cdot, \cdot]_1)$ to $(\mathbb{G}, [\cdot, \cdot]_2)$. $\beta$ is called the Cayley transform on $\mathbb{G}$.

All positive-definite subspaces in $(\mathbb{G}, [\cdot, \cdot]_1)$ are easy to specify. They are all of the form
$$L_z:=\{(x,x z)\in\mathbb{G} |x\in \mathbb{C}\}$$
 where $z$ is a complex number with $|z|<1$. As a consequence
$L_z^{\bot_s}=\{(x\bar{z},x)\in\mathbb{G}|x\in \mathbb{C}\}$
 gives rise to all negative-definite subspaces. Besides, those $L_z$ with $|z|=1$ are Lagrangian, i.e., $\textup{dim}L_z=1$ and the restriction of $[\cdot, \cdot]_1$ on $L_z$ vanishes. These possibilities actually exhaust all 1-dimensional subspaces of $\mathbb{G}=\mathbb{C}^2$, i.e., in other words $\mathbb{CP}^1$ is partitioned into two copies of the open unit disc and a circle. Similarly, the positive-definite (resp. negative-definite) subspaces in $(\mathbb{G}, [\cdot, \cdot]_2)$ are of the form $\{(x,x z)\in\mathbb{G} |x\in \mathbb{C}\}$ where $z\in \mathbb{C}_+$ (resp. $z\in \mathbb{C}_-$). This can be seen directly or by using the Cayley transform.

 Now let $\Phi$ be a fixed isomorphism between $\mathcal{B}_T$ and $(\mathbb{G}, [\cdot, \cdot]_1)$. $\Phi$ can be extended to $A_T^{\bot_s}$ by setting its value at $x\in A_T^{\bot_s}$ to be $\Phi([x])$. This extension will be still denoted by $\Phi$. According to the decomposition $\mathbb{G}=\mathbb{C}\oplus_\bot \mathbb{C}$, now we have two maps $\Gamma_\pm: A_T^{\bot_s}\rightarrow \mathbb{C}$ such that $\Phi(x)=(\Gamma_+ x, \Gamma_-x)\in \mathbb{G}$ for any $x\in A_T^{\bot_s}$ and consequently we have the abstract Green's formula
 \begin{equation}[x,y]_c=\mathrm{i}(\Gamma_+x, \Gamma_+y)_\mathbb{C}-\mathrm{i}(\Gamma_-x,\Gamma_-y)_\mathbb{C},\ \forall\ x,y\in A_T^{\bot_s}.\label{green}\end{equation}
 Similarly, $\beta\circ \Phi$ is an isomorphism between $\mathcal{B}_T$ and $(\mathbb{G}, [\cdot, \cdot]_2)$ and we get two maps $\Gamma_{0/1}: A_T^{\bot_s}\rightarrow \mathbb{C}$ such that
 \begin{equation}[x,y]_c=(\Gamma_1x, \Gamma_0y)_\mathbb{C}-(\Gamma_0x,\Gamma_1y)_\mathbb{C},\ \forall\ x,y\in A_T^{\bot_s}.\label{green1}\end{equation}
 The triple $(\mathbb{C}, \Gamma_+, \Gamma_-)$ or $(\mathbb{C}, \Gamma_0, \Gamma_1)$ will be called a boundary triplet and the context will clarify which we really mean.

 Now it is fairly clear that when $\lambda\in \mathbb{C}_+$, \[\Phi(W_T(\lambda))=\{(x, xB(\lambda))\in \mathbb{G}|x\in \mathbb{C}\}\] for a holomorphic function $B(\lambda)$ on $\mathbb{C}_+$ such that $|B(\lambda)|<1$ and
 \[\beta\circ\Phi(W_T(\lambda))=\{(x, xM(\lambda))\in \mathbb{G}|x\in \mathbb{C}\}\] for a holomorphic function $M(\lambda)$ on $\mathbb{C}_+$ such that $\Im M(\lambda)>0$. $B(\lambda)$ and $M(\lambda)$ are related via
 $B(\lambda)=(M(\lambda)-\mathrm{i})(M(\lambda)+\mathrm{i})^{-1}$. As a consequence, for $\lambda\in \mathbb{C}_-$,
 \[\Phi(W_T(\lambda))=\{(x\overline{B(\bar{\lambda})}, x)\in \mathbb{G}|x\in \mathbb{C}\},\quad \beta\circ\Phi(W_T(\lambda))=\{(x, x\overline{M(\bar{\lambda})})\in \mathbb{G}|x\in \mathbb{C}\}.\]
 Following \cite{wang2024complex} we shall call $B(\lambda)$ the \emph{contractive Weyl function} of $T$ w.r.t. $(\mathbb{C}, \Gamma_+, \Gamma_-)$, while $M(\lambda)$ is traditionally called the Weyl function of $T$ w.r.t. $(\mathbb{C}, \Gamma_0, \Gamma_1)$.

 Since there are many choices of $\Phi$, $B(\lambda)$ and $M(\lambda)$ are not uniquely determined by $T$ itself. However, different $B(\lambda)$'s (resp.~$M(\lambda)$'s) are related by the action of the pseudo-unitary group $\mathbb{U}(1,1)$. In effect, they are connected by the automorphism group of the unit disc $\mathbb{D}$ (resp. the upper half-plane $\mathbb{C}_+$). For instance, if $B(\lambda)$ is a contractive Weyl function for $T\in \mathfrak{E}_1(H)$, then any other contractive Weyl function for $T$ should be of the form
 \begin{equation}\tilde{B}(\lambda)=\gamma\times\frac{B(\lambda)-\tau}{1-\bar{\tau}B(\lambda)}\label{mo}\end{equation}
 where $\gamma, \tau\in \mathbb{C}$ are constants such that $|\gamma|=1$ and $|\tau|<1$. $B(\lambda)$ and $\tilde{B}(\lambda)$ are said to be congruent. It is the congruence class of $B(\lambda)$ (so-called Weyl class of $T$) that really determines the unitary equivalence class of $T$.

It is of basic interest to learn how the unitarily invariant properties of $T$ are related to properties of its Weyl curve $W_T(\lambda)$. Here is an example. We note that $\varpi(w):=W_T(\mathrm{i}\frac{1+w}{1-w})$ maps the unit disc $\mathbb{D}$ in the $w$-plane to $\mathcal{M}$.
\begin{definition}We say $\varpi(w)$ has an \emph{angular tangent line} at $w=1$ in the sense of Caratheodory, if the non-tangential limits of $\varpi(w)$ and $\frac{d}{dw}\varpi(w)$ exist at $w=1$ and the non-tangential limit of $\varpi(w)$ at $w=1$ is a Lagrangian subspace of $\mathcal{B}_T$.
\end{definition}
\emph{Remark}. The definition is a geometric abstraction of the function-theoretic notion of angular derivative in the sense of Caratheodory. Note that here $\frac{d}{dw}\varpi(w)$ should be interpreted as the tangent map of $\varpi(\omega)$, which is a complex curve lying in the holomorphic tangent bundle of $\mathcal{M}$.
\begin{theorem}If $T$ is a simple symmetric operator in the Hilbert space $H$, then $T$ is densely defined if and only if $\varpi(w)$ doesn't have an angular tangent line at $w=1$ in the sense of Caratheodory.
\end{theorem}
\begin{proof}This is actually a geometric reformulation of a result in  \cite{martin2011representation}, which characterizes the denseness property in terms of the Lifschitz characteristic function $\omega_T(\lambda)$ (a special contractive Weyl function). See Thm.~3.1.2 there. It was shown that $T$ is densely defined if and only if $\omega_T(\mathrm{i}\frac{1+w}{1-w})$ doesn't have an angular derivative at $w=1$ in the sense of Caratheodory. It can be seen easily that this property is invariant under the coordinate transformation (\ref{mo}) and thus our claim follows.
\end{proof}
\emph{Remark}. The denseness property can also be characterized in terms of the Weyl function $M(\lambda)$. See \cite[Coro.~3.7.3]{behrndt2020boundary}.

A point $\lambda\in \mathbb{C}$ is called a point of regular type for $T$ if there is a number $c_\lambda>0$ such that
 \[\|(T-\lambda)x\|\geq c_\lambda \|x\|, \ \forall x\in D(T).\]
 The set of points of regular type for $T$ necessarily contains $\mathbb{C}_+\cup \mathbb{C}_-$ and if it's actually the whole $\mathbb{C}$, $T$ is called \emph{regular}. Due to Prop.~9.3 in \cite{wang2024complex}, operators $T\in \mathfrak{E}_1(H)$ are precisely the regular ones. For $T\in \mathfrak{E}_1(H)$, given the boundary triplet, both $B(\lambda)$ and $M(\lambda)$ can be extended to be meromorphic functions on $\mathbb{C}$. $B(\lambda)$ is actually a meromorphic inner function, i.e., $|B(\lambda)|=1$ for all $\lambda\in \mathbb{R}$. Additionally $B(\lambda)\cdot\overline{B(\bar{\lambda})}\equiv1$. Zeros (resp. poles) of $B(\lambda)$ should lie on $\mathbb{C}_+$ (resp. $\mathbb{C}_-$). Zeros and poles of $M(\lambda)$ all lie on the real line $\mathbb{R}$ and are simple and interlaced.
 \subsection{A canonical functional model}\label{model}
A functional model for a simple symmetric operator $T$ is usually a concrete function space $\mathfrak{H}$ together with a model operator $\mathfrak{T}$ acting in it such that $\mathfrak{T}$ is unitarily equivalent to the original $T$. $\mathfrak{H}$ is often equipped with extra structures facilitating the study of $T$. In \cite[\S5.2]{wang2024complex} a canonical model associated to a general simple symmetric operator $T$ has been suggested. Here "canonical" means the model is constructed in terms of $T$ itself, without referring to any other artificial choices. In this subsection, we recall this model theory for $T\in \mathfrak{E}_1(H)$.

Associated to $T$ is a holomorphic line bundle $E$ over $\mathbb{C}$: For each $\lambda\in \mathbb{C}$, the fiber $E_\lambda$ at $\lambda$ is just $\pi_1(W_\lambda)$ where $\pi_1$ is the projection from $\mathbb{H}=H\oplus_\bot H$ onto its first component. $E_\lambda\subset H$ is 1-dimensional and equipped with a natural Hermitian structure by restricting the inner product of $H$ on $E_\lambda$. Exchange the fiber at $\lambda$ with that at $\bar{\lambda}$ (when $\lambda\in \mathbb{R}$, $E_\lambda$ is unchanged). This produces an anti-holomorphic line bundle $F^\dag$ over $\mathbb{C}$. Let $F$ be the conjugate-linear dual of $F^\dag$. $F$ is again a holomorphic line bundle with its natural induced Hermitian structure. The pairing between $F$ and $F^\dag$ will be denoted by $\langle\cdot,\cdot\rangle$, i.e., $\varphi(\omega)=\langle\varphi,\omega\rangle$ for $\varphi\in F_\lambda$, $\omega\in F^\dag_\lambda$.

To each $x\in H$ is associated a holomorphic section $\hat{x}$ of $F$: For any $\omega\in F_\lambda^\dag$, $\langle\hat{x}(\lambda), \omega\rangle:=(x, \omega)_H$,
where $\hat{x}(\lambda)\in F_\lambda$ is the value of $\hat{x}$ at $\lambda$. Since $T$ is simple, it is necessary that $\hat{x}\equiv 0$ if and only if $x=0$. Then we get a linear space $\mathfrak{H}:=\hat{H}$ of holomorphic sections of $F$. We can transfer the Hilbert space structure of $H$ onto $\mathfrak{H}$. With this structure, $\mathfrak{H}$ is actually a reproducing kernel Hilbert space in the sense of \cite{bertram1998reproducing}. The reproducing kernel on $F$ can be described easily. Let $\iota_\lambda$ be the canonical embedding of $F^\dag_\lambda$ into $H$. Then we can define $\iota_\lambda^\dag: H\rightarrow F_\lambda$ via
\[\langle \iota_\lambda^\dag x, \omega\rangle=(x, \iota_\lambda\omega)_H,\  \forall x\in H, \ \omega\in F^\dag_\lambda.\]
Thus $\hat{x}(\cdot)=\iota_{\cdot}x$. Then the reproducing kernel is simply $K(\lambda, \mu)=\iota_\lambda^\dag\iota_\mu: F^\dag_\mu\rightarrow F_\lambda$. $\mathfrak{H}$ is just our canonical model space.

Multiplication by the independent variable $\lambda$ certainly maps a holomorphic section $s$ of $F$ to another $\mathfrak{X}s$. Let $D(\mathfrak{T}):=\{s\in \mathfrak{H}|\mathfrak{X}s\in \mathfrak{H}\}$. Then the operator $\mathfrak{T}:D(\mathfrak{T})\rightarrow \mathfrak{H}, s\in D(\mathfrak{T})\mapsto \mathfrak{X}s$ is our model operator of $T$. More precisely, $x\in D(T)$ if and only if $\hat{x}\in D(\mathfrak{T})$ and in that case $\widehat{Tx}(\lambda)=\lambda \hat{x}(\lambda)$ for any $\lambda\in \mathbb{C}$.

Due to the famous Oka-Grauert principle, our line bundle $F$ over $\mathbb{C}$ is holomorphically trivial, i.e., there is a (not unique) global holomorphic section $s$ of $F$ not vanishing at any $\lambda\in \mathbb{C}$. In terms of this $s$, any $\hat{x}$ now determines an entire function $f_x$ via $\hat{x}=f_x\cdot s$ and $\mathfrak{H}$ can thus be identified as a Hilbert space of such entire functions. Different choices of $s$ may produce Hilbert spaces seemingly rather different. More on this respect will be presented in \S~\ref{S4}.

The canonical model can also be recovered from the Weyl curve $W_T(\lambda)$. More generally, let $\mathcal{B}$ be a 2-dimensional strong symplectic Hilbert space whose signature is $(1,1)$ and let $\mathcal{M}_\mathcal{B}$ be the Grassmanian of 1-dimensional subspaces. A holomorphic map $N(\lambda)$ from $\mathbb{C}$ to $\mathcal{M}_\mathcal{B}$ is called an entire Nevanlinna curve if $N(\lambda)$ for each $\lambda\in \mathbb{C}_+$ represents a positive-definite subspace in $\mathcal{B}$ and $N(\lambda)=N(\bar{\lambda})^{\bot_s}$ for each $\lambda\in \mathbb{C}_-$. For each entire Nevanlinna curve $N(\lambda)$ one can associate a canonical Hilbert space $\mathfrak{H}$ of holomorphic sections of a holomorphic line bundle over $\mathbb{C}$ and multiplication in this space by the independent variable is an entire operator with deficiency index 1, for which $N(\lambda)$ can be viewed as the corresponding Weyl curve. In this sense, all entire Nevanlinna curves stem from entire operators with deficiency index 1. For details of this argument, see \cite[\S~5.2]{wang2024complex}.
\subsection{Boundary value problems and value distribution theory}\label{bvp}
Given $T\in \mathfrak{E}_1(H)$, its generic extensions are parameterized by points in the projective line $\mathcal{M}$. Recall that for $m\in \mathcal{M}$, we use $T_m$ to denote the corresponding extension. Then $\sigma(T_m)$ consists of only eigenvalues (if any). $\sigma(T_m)$ can be interpreted in a geometric way as follows.

Let $l$ be the tautological line bundle over $\mathcal{M}$, i.e., at $m\in \mathcal{M}$ the fiber $l_m$ is precisely the line in $\mathcal{B}_T$ parameterized by $m$. Let $l^*$ be the linear dual line bundle. We can fix a nonzero element $\varrho \in \wedge^2 \mathcal{B}_T$. Then a point $q\in \mathcal{M}$ (i.e., an abstract boundary condition) provides a holomorphic section $s_q$ of $l^*$ up to a constant factor: Choose $0\neq v\in l_q$ and set $s_q(w_m)\varrho =v\wedge w_m$ for each $w_m\in l_m$. $q$ is precisely the simple zero of $s_q$.

$\lambda\in \mathbb{C}$ is an eigenvalue of $T_m$ if and only if $W_\lambda\subset T_m$. Since $\textup{dim}_{\mathbb{C}}(T_m/A_T)=1$, this happens if and only if $W_T(\lambda)=m$. Phrased in another way, this is equivalent to that $s_m(W_T(\lambda))=0$ or that $\lambda$ is an $m$-point of the Weyl curve in the sense of value distribution theory. It is possible that a certain $m\in \mathcal{M}$ can only be obtained a finite number of times by $W_T(\lambda)$. In this case, we say $m$ is a Picard exceptional boundary condition.

Note that $\mathcal{B}_T$ has a canonical Hermitian inner product which induces a Hermitian metric on $l^*$. Let $c_1$ be the first Chern form of $l^*$ w.r.t. the unique Chern connection. Let $\mathbb{D}_t$ for $t>0$ be the open disc on $\mathbb{C}$ with center 0 and radius $t$. Then the height $h_T(r)$ of $T$ is defined to be the function
\[h_T(r)=\int_0^r\frac{dt}{t}\int_{\mathbb{D}_t}W_T^*(c_1),\]
where $W_T^*(c_1)$ is the pull-back of $c_1$ through $W_T: \mathbb{C}\rightarrow \mathcal{M}$. $h_T(r)$ is actually the characteristic function of the entire curve $W_T(\lambda)$, measuring the growth of $T$. If
$$\liminf_{r\rightarrow +\infty}\frac{h_T(r)}{\ln r}=+\infty,$$
we say $T$ is transcendental. Otherwise, $T$ is called rational. Actually, rational entire operators only occur when $\dim H<+\infty$ and the corresponding Weyl curves are essentially rational functions in $\lambda$. We only consider transcendental entire operators in the paper. The Weyl order $\rho_T$ is defined to be
\[\rho_T=\limsup_{r\rightarrow+\infty}\frac{\ln h_T(r)}{\ln r}.\]
If $\rho_T$ is finite, we say $T$ is of finite Weyl order, and in that case the number
\[\tau_T=\limsup_{r\rightarrow+\infty}\frac{h_T(r)}{r^{\rho_T}}\]
is called the Weyl type of $T$. Note that $\tau_T\in [0, +\infty]$. $T$ is called of at most minimal exponential type if $\rho_T\leq 1$ and $\tau_T=0$ in case $\rho_T=1$.

Let $n_T(r, m)$ be the number of eigenvalues (counted with analytic multiplicity) of $T_m$ in $\mathbb{D}_r$. The counting function of $T_m$ is defined to be
\[N_T(r, m)=\int_0^r [n_T(t, m)-n_T(0,m)]\frac{dt}{t}+n_T(0,m)\ln r,\]
where $n_T(0,m)$ is the multiplicity of 0 as an eigenvalue of $T_m$. The proximity function $\mathfrak{m}_T(r, m)$ of $T_m$ is
\[\mathfrak{m}_T(r, m):=-\int_0^{2\pi}\ln|s_m(W_T(re^{\mathrm{i}\phi}))|\frac{d\phi}{2\pi}.\] The First Main Theorem for entire curves now says that
\[h_T(r)=N_T(r,m)+\mathfrak{m}_T(r, m)+O(1),\]
where $O(1)$ means a bounded term in $r$. The Nevanlinna defect $\delta_T(m)$ of $T_m$ is defined by
\[\delta_T(m)=\liminf_{r\rightarrow +\infty}\frac{\mathfrak{m}_T(r,m)}{h_T(r)}.\]
Obviously $\delta_T(m)\in[0,1]$. If $\delta_T(m)>0$, $m$ is called a Nevanlinna exceptional boundary condition. In particular, for a Picard exceptional boundary condition $m$, $\delta_T(m)=1$. Nevanlinna's famous Defect Relation now reads
\begin{equation}\Sigma_{m\in \mathcal{M}}\delta_T(m)\leq 2.\end{equation}
Consequently, there can be at most countably many Nevanlinna exceptional boundary conditions.

Self-adjoint extensions are special in that they can never be Nevanlinna exceptional. Actually, for a self-adjoint boundary condition $m\in \mathcal{M}$, one can conclude from \cite[Thm.~10.19 and Thm.~10.43]{wang2024complex} that
\begin{equation}h_T(r)=N_T(r,m)+O(\ln r).\label{se}\end{equation}

If a boundary triplet $(\mathbb{C}, \Gamma_+, \Gamma_-)$ is chosen, all generic extensions can be written down concretely. For each $c\in \mathbb{C}\cup \{\infty\}\cong \mathbb{CP}^1$, there is the extension $T_c$ characterized by
\[T_c:=\{a\in A_T^{\bot_s}|\Gamma_-a=c \Gamma_+ a\}.\]
Note that for $c=\infty$, the above means $T_\infty=\{a\in A_T^{\bot_s}|\Gamma_+ a=0\}$. $T_c$ with $|c|<1$ is called \emph{strictly dissipative} while $T_c$ with $|c|>1$ is called \emph{strictly accumulative}. It is known that for $|c|<1$, we have $\sigma(T_c)\subset \mathbb{C}_+$, and for $|c|>1$ ($c$ could be $\infty$), $\sigma(T_c)\subset \mathbb{C}_-$ \cite[\S~9]{wang2024complex}. Self-adjoint extensions correspond to those $c$'s such that $|c|=1$. Note that $T_0$ is the conjugate of $T_\infty$. As a consequence, $\sigma(T_0)=\overline{\sigma(T_\infty)}$. It should also be pointed out that the above classification of generic boundary conditions is in fact independent of the choices of boundary triplets.

If $B(\lambda)$ is the contractive Weyl function w.r.t. $(\mathbb{C}, \Gamma_+, \Gamma_-)$, $B(\lambda)$ is necessarily a meromorphic function on $\mathbb{C}$, whose poles are precisely points in $\sigma(T_\infty)$ and whose zeros are precisely points in $\sigma(T_0)$. In general, $\lambda\in \sigma(T_c)$ if and only if $B(\lambda)-c$=0 and in that case the algebraic multiplicity of $\lambda\in \sigma(T_m)$ coincides with the analytic multiplicity of $\lambda$ as a root of this equation \cite[Thm.~10.36]{wang2024complex}. In particular, roots of the equations $B(\lambda)\mp 1=0$ are precisely the eigenvalues of the self-adjoint extensions $T_1$ and $T_{-1}$, and each eigenvalue is simple. The above argument has the following immediate conclusion.
\begin{proposition}\label{par}For $T\in \mathfrak{E}_1(H)$, let $m_1, m_2\in \mathcal{M}$ be two distinct boundary conditions. Then $\sigma(T_{m_1})\cap \sigma(T_{m_2})=\emptyset$. In particular, the collection $\{\sigma(T_m)\}_{m\in \mathcal{M}}$ is a partition of $\mathbb{C}$.
\end{proposition}
\begin{proof}Obviously, we only have to prove the first statement for $m_1,m_2$ that are strictly dissipative. We choose a boundary triplet $(\mathbb{C}, \Gamma_+, \Gamma_-)$. Then the two boundary conditions are determined by two distinctive constants $c_1, c_2$ in the open unit disc and the spectra arise from roots of $B(\lambda)-c_1=0$ and $B(\lambda)-c_2=0$ respectively. This surely implies $T_{m_1}$ and $T_{m_2}$ cannot have a common eigenvalue. The second statement is clear.
 \end{proof}
 \emph{Remark}. This proposition is a generalization of a well-known result on self-adjoint extensions. See for instance Thm.5 (ii) in \cite{silva2013branges}.

The height function can also be written simply in terms of $B(\lambda)$ ( \cite[Thm.~10.19]{wang2024complex}):
\begin{equation}h_T(r)=\frac{1}{2\pi \mathrm{i}}\int_0^r\frac{dt}{t}\int_{-t}^td\ln B(u)+O(1),\label{height}\end{equation}
where $B(u)$ is $B(\lambda)$ restricted on the real line $\mathbb{R}$.

From the above argument, we immediately have
\begin{proposition}For $T\in \mathfrak{E}_1(H)$ we have
\[\sum_{m\in s. diss.}\delta_T(m)\leq 1,\]
where the summation is over all strictly dissipative extensions of $T$. In particular, the number of Picard exceptional boundary conditions can only be 0 or 2.
\end{proposition}
\begin{proof} In terms of a fixed boundary triplet $(\mathbb{C}, \Gamma_+, \Gamma_-)$, for a strictly dissipative extension $T_c$, its conjugate is $T_{1/\bar{c}}$. Thus $\sigma(T_{1/\bar{c}})=\overline{\sigma(T_c)}$. In particular, the (analytic) multiplicity of $\lambda_i\in \sigma(T_c)$ as a root of $B(\lambda)-c=0$ is the same as that of $\overline{\lambda_i}$ as a root of $B(\lambda)-\frac{1}{\bar{c}}=0$. Thus Nevanlinna exceptional boundary conditions must occur in pairs. The conclusion follows from Nevanlinna's Defect Relation.
\end{proof}
\section{Geometry of the characteristic line bundle}\label{line}
Since the characteristic line bundle $F$ plays a basic role in our canonical model, it is interesting to learn the relationship of the geometry of $F$ with the model Hilbert space $\mathfrak{H}$. In particular, though the curvature of $F$ is essentially computed in \cite{wang2024complex} and shown to be a complete unitary invariant, its meaning for spectral theory was unrevealed there. In this section we investigate the geometric aspect of $F$. We continue to use the notations in \S\S~\ref{bvp}.

 Let a boundary triplet $(\mathbb{C}, \Gamma_+, \Gamma_-)$ for $T\in \mathfrak{E}_1(H)$ be fixed. We note that $T_\infty$ provides a holomorphic frame for $E$ over $\mathbb{C}\backslash \sigma(T_\infty)$. In fact, for each $\lambda\in \mathbb{C}\backslash \sigma(T_\infty)$, there is a unique $0\neq \gamma_+(\lambda)\in E_\lambda \subset H$ such that \[\hat{\gamma}_+(\lambda):=(\gamma_+(\lambda), \lambda\gamma_+(\lambda))\in W_\lambda\subset \mathbb{H}\] and $\Gamma_+\hat{\gamma}_+(\lambda)=1$. Similarly, for each $\lambda\in \mathbb{C}\backslash \sigma(T_0)$, there is a unique $0\neq \gamma_-(\lambda)\in E_\lambda\subset H$ such that
\[\hat{\gamma}_-(\lambda):=(\gamma_-(\lambda), \lambda\gamma_-(\lambda))\in W_\lambda\]
and $\Gamma_-\hat{\gamma}_-(\lambda)=1$. $\gamma_+(\lambda)$ (resp. $\gamma_-(\lambda)$) is thus a holomorphic frame of $E$ over $\mathbb{C}\backslash \sigma(T_\infty)$ (resp. $\mathbb{C}\backslash \sigma(T_0)$). Note that by definition $\Gamma_-\hat{\gamma}_+(\lambda)=B(\lambda)$ and $\Gamma_+\hat{\gamma}_-(\lambda)=\frac{1}{B(\lambda)}$.
\begin{lemma}\label{metric}$\gamma_+(\lambda)$ is a meromorphic section of $E$, whose poles are precisely points in $\sigma(T_\infty)$. The order of a pole $\lambda_i$ of $\gamma_+(\lambda)$ is the same as the analytic multiplicity of $\lambda_i$ as an eigenvalue of $T_\infty$. In particular, for $\lambda, \mu\in \mathbb{C}\backslash \sigma(T_\infty)$
\begin{equation}(\gamma_+(\lambda),\gamma_+(\mu))_H=\mathrm{i}\frac{1-\overline{B(\mu)}B(\lambda)}{\lambda-\bar{\mu}}.\label{fr}\end{equation}
\end{lemma}
\begin{proof}For each $\lambda_i\in \sigma(T_\infty)$, by definition $\gamma_-(\lambda)$ is non-vanishing around $\lambda_i$. Away from $\lambda_i$,
\[(\Gamma_+,\Gamma_-)(\hat{\gamma}_+(\lambda))=(1, B(\lambda))=B(\lambda)(\frac{1}{B(\lambda)}, 1)=B(\lambda)(\Gamma_+,\Gamma_-)(\hat{\gamma}_-(\lambda)).\]
This shows $\gamma_+(\lambda)=B(\lambda)\gamma_-(\lambda)$ and the claim follows. The latter formula comes from applying the abstract Green's formula (\ref{green}) to $\hat{\gamma}_+(\lambda)$ and $\hat{\gamma}_+(\mu)$.
\end{proof}
\emph{Remark}. In Eq.~(\ref{fr}), since $B(\lambda)\cdot \overline{B(\bar{\lambda})}\equiv 1$, we have
\[\frac{1-\overline{B(\mu)}B(\lambda)}{\lambda-\bar{\mu}}=\overline{B(\mu)}\times\frac{B(\bar{\mu})-B(\lambda)}{\lambda-\bar{\mu}}.\]
Thus when $\lambda=\bar{\mu}$, the RHS of Eq.~(\ref{fr}) should be interpreted as $-i\overline{B(\mu)}B'(\bar{\mu})$.

Note that for $\lambda\in \mathbb{C}\backslash \sigma(T_0)$, $\gamma_+(\bar{\lambda})\in F^\dag_\lambda$. By duality this defines $\psi(\lambda)\in F_\lambda$ via
\begin{equation}\langle \psi(\lambda),\gamma_+(\bar{\lambda})\rangle=1.\label{dual}\end{equation}
Then $\psi(\lambda)$ is a holomorphic frame of $F$ over $\mathbb{C}\backslash \sigma(T_0)$. More precisely,
\begin{lemma}$\psi(\lambda)$ is a holomorphic section of $F$ over $\mathbb{C}$. The zeros of $\psi(\lambda)$ are precisely points in $\sigma(T_0)$. The order of a zero $\lambda_i$ is the same as the analytic multiplicity of $\lambda_i$ as an eigenvalue of $T_0$. In particular,
\[(\psi(\lambda), \psi(\lambda))_\lambda=-\frac{2\Im \lambda}{1-|B(\bar{\lambda})|^2},\]
where $(\cdot,\cdot)_\lambda$ is the induced Hermitian inner product on $F_\lambda$.
\end{lemma}
\begin{proof}The first two statements are clear from the previous lemma and Eq.~(\ref{dual}). For the third, let $g(\lambda)\in \mathbb{C}$ be the unique solution of the equation
\[1=\langle\psi(\lambda), \gamma_+(\bar{\lambda})\rangle=(g(\lambda)\gamma_+(\bar{\lambda}), \gamma_+(\bar{\lambda}))_\lambda=-\frac{1-|B(\bar{\lambda})|^2}{2\Im \lambda}g(\lambda).\]
We see that
\[g(\lambda)=-\frac{2\Im \lambda}{1-|B(\bar{\lambda})|^2}.\]
Thus by definition
\[(\psi(\lambda), \psi(\lambda))_\lambda=[g(\lambda)]^2(\gamma_+(\bar{\lambda}),( \gamma_+(\bar{\lambda}))_\lambda=-\frac{2\Im \lambda}{1-|B(\bar{\lambda})|^2}.\]
\end{proof}
As a holomorphic Hermitian line bundle, $F$ is equipped with the canonical Chern connection whose curvature is of basic importance.
\begin{lemma}In terms of the above section $\psi(\lambda)$, the curvature of the Chern connection in $F$ is
\[R=[\frac{1}{4(\Im \lambda)^2}-\frac{|B'(\lambda)|^2}{(1-|B(\lambda)|^2)^2}]d\lambda\wedge d\bar{\lambda}.\]
\end{lemma}
\begin{proof}Let $h(\lambda):=(\psi(\lambda), \psi(\lambda))_\lambda$. Since $B(\lambda)\cdot \overline{B(\bar{\lambda})}\equiv 1$,
\[h(\lambda)=|B(\lambda)|^2\times \frac{2\Im \lambda}{1-|B(\lambda)|^2}.\]
It is well-known that $R=\overline{\partial}\partial\ln h$ and the lemma follows from a direct computation.
\end{proof}
\emph{Remark}. The 2-form $R$ is actually globally defined on $\mathbb{C}$ and independent of the choices of boundary triplets. $R$ is a complete unitary invariant of $T$, so is the real-analytic \emph{curvature function}
\[\omega(\lambda):=\frac{1}{4(\Im \lambda)^2}-\frac{|B'(\lambda)|^2}{(1-|B(\lambda)|^2)^2}\]
on $\mathbb{C}$. In terms of the Weyl function $M(\lambda)$ \cite[Thm.~5.9]{wang2024complex},
\begin{equation}\omega(\lambda)=\frac{1}{4(\Im \lambda)^2}(1-(\frac{|M'(\lambda)|}{\Im{M(\lambda)/\Im \lambda}})^2).\label{cur2}\end{equation}
It's easy to see $\omega(\overline{\lambda})=\omega(\lambda)$. It's also a basic result that $\omega(\lambda)>0$ for $\lambda\in \mathbb{C}\backslash \mathbb{R}$ \cite[Thm.~5.8]{wang2024complex}. This is actually a version of the famous Schwarz-Pick inequality, which implies
\begin{proposition}\label{lem1}Let $B(\lambda)$ be the contractive Weyl function of $T\in \mathfrak{E}_1(H)$ w.r.t. a given boundary triplet $(\mathbb{C}, \Gamma_+,\Gamma_-)$. Then for each fixed $u\in \mathbb{R}$, the function $\chi(\lambda):=\frac{1-|B(\lambda)|^2}{2\Im \lambda}$ is strictly decreasing on the half-line $l(u):=\{\lambda\in \mathbb{C}_+|\Re \lambda=u\}$.
\end{proposition}
\begin{proof}Let $\lambda=u+\mathrm{i}v$. We should prove that $\chi(u+\mathrm{i}v)$ is strictly decreasing for $v>0$. Note that
\begin{eqnarray*}\frac{\partial}{\partial v}(2\chi)&=&\frac{\partial}{\partial v}(\frac{1-|B|^2}{v})=\frac{2\Re(\mathrm{i}B\overline{B'})v-(1-|B|^2)}{v^2}\\
&\leq&\frac{2|B||B'|v-(1-|B|^2)}{v^2}\leq \frac{2|B'|v-(1-|B|^2)}{v^2}\\
&<&0,
\end{eqnarray*}
where the last line follows from the inequality $\omega(\lambda)>0$ for $\lambda\in \mathbb{C}_+$.
\end{proof}

 We are particularly interested in the restriction of $\chi(\lambda)$ and $\omega(\lambda)$ on the real line. Note that there is a strictly increasing real-analytic phase function $\theta(u)$ on $\mathbb{R}$ such that $B(u)=e^{\mathrm{i}\theta(u)}$ for $u\in \mathbb{R}$. $\theta(u)$ is determined by $B(\lambda)$ up to addition of an integral multiple of $2\pi$. We say $\theta(u)$ is the phase function associated to $B(\lambda)$ or the underlying boundary triplet.
 \begin{proposition}\label{ka}For $u\in \mathbb{R}$, $\chi(u)=\theta'(u)$.
 \end{proposition}
 \begin{proof}This is actually \cite[Prop.~8.9]{wang2024complex} in the case with deficiency indices $(1,1)$. See also the remark following this proposition. It can also be proved by directly using Taylor expansion.
 \end{proof}
 Since $\omega(\lambda)\leq \frac{1}{4(\Im \lambda)^2}$ for $\lambda\in \mathbb{C}\backslash \mathbb{R}$, the line bundle $F$ is almost flat far from the real line. On the real line we have
\begin{proposition}\label{cur}Let $\theta(u)$ be the phase function associated to the contractive Weyl function $B(\lambda)$. Then
\[\omega(u)=\frac{1}{4}[\frac{(\theta'(u))^2}{3}+\frac{2}{3}\frac{\theta'''(u)}{\theta'(u)}-(\frac{\theta''(u)}{\theta'(u)})^2].\]
\end{proposition}
\begin{proof}For a fixed $u\in \mathbb{R}$, w.l.g., we assume $B(u)=1$. Let $\lambda=u+\mathfrak{i}\varepsilon$ for a small $\varepsilon>0$. By using the Taylor expansion of $B(\lambda)$ around $u$, we have
\[B(u+\mathrm{i}\varepsilon)=1+\sum_{k=1}^{\infty}a_k(\mathrm{i}\varepsilon)^k.\]
Due to the identity $B(\lambda)\cdot \overline{B(\bar{\lambda})}\equiv 1$, we can find that
\[a_1=\mathrm{i}c_1, \quad a_2=-\frac{c_1^2}{2}+\mathrm{i}c_2,\quad a_3=-c_1c_2+\mathrm{i}c_3,\]
where $c_1,c_2,c_3$ are real numbers. A careful but routine computation shows
\[|B'(u+\mathrm{i}\varepsilon)|^2=c_1^2-2c_1^3\varepsilon +(c_1^4+4c_2^2-6c_1c_3)\varepsilon^2+o(\varepsilon^2),\]
and
\[1-|B(u+\mathrm{i}\varepsilon)|^2=2c_1\varepsilon -2c_1^2\varepsilon^2-(2c_3-c_1^3)\varepsilon^3+o(\varepsilon^3).\]
From these we can deduce that
\begin{equation}\omega(u)=\lim_{\varepsilon\rightarrow 0+}(\frac{1}{4\varepsilon^2}-\frac{|B'(u+\mathrm{i}\varepsilon)|^2}{(1-|B(u+\mathrm{i}\varepsilon)|^2)^2})=\frac{1}{4}[c_1^2+\frac{4c_3}{c_1}-(\frac{2c_2}{c_1})^2].\label{cu}\end{equation}
On the other side, $B(u)=e^{\mathfrak{i}\theta(u)}$ on the real line and we find
\[c_1=\theta'(u),\quad c_2=\frac{\theta''(u)}{2},\quad c_3=\frac{1}{6}(\theta'''(u)-(\theta'(u))^3).\]
Substituting these into (\ref{cu}) leads to the claimed expression.
\end{proof}

Though generally a real-analytic function on $\mathbb{C}$ is not determined by its restriction on the real line, the function $\omega(u)$ for $u\in \mathbb{R}$ is really a complete unitary invariant of $T$. We need a lemma to prove this fact.
\begin{lemma}For $T\in \mathfrak{E}_1(H)$, a suitable boundary triplet can be chosen such that its corresponding phase function $\theta(u)$ satisfies
\[\theta(0)=0,\quad \theta'(0)=1,\quad \theta''(0)=0.\]
\end{lemma}
\begin{proof}In principle this is because in the formula (\ref{mo}), there are three free \emph{real} parameters in $\gamma$ and $\tau$ that can be adjusted to justify the claim. A detailed proof goes as follows.

It is clear that one can choose $\gamma$ to make the new $B(\lambda)$ to satisfy $B(0)=1$ and thus we can choose $\theta$ such that $\theta(0)=0$. In the following we assume $\theta(0)=0$.

We can choose $\gamma$ and $\tau$ in (\ref{mo}) such that the phase function $\tilde{\theta}$ for $\tilde{B}(\lambda)$ satisfies $\tilde{\theta}(0)=0$ and $\tilde{\theta}'(0)=1$. Since
\[e^{\mathrm{i}\tilde{\theta}(u)}=\gamma\frac{e^{\mathrm{i}\theta(u)}-\tau}{1-\bar{\tau}e^{\mathrm{i}\theta(u)}}\]
and consequently
\[\tilde{\theta}'=[\frac{1}{e^{\mathrm{i}\theta}-\tau}+\frac{\bar{\tau}}{1-\bar{\tau}e^{\mathrm{i}\theta}}]e^{\mathrm{i}\theta}\theta',\]
this requires we choose $\gamma$ and $\tau$ such that $\gamma=\frac{1-\bar{\tau}}{1-\tau}$
and $c(1-|\tau|^2)=|1-\tau|^2$ where $c=\theta'(0)>0$. Let $\tau=a+\mathrm{i}b$ for $a,b\in \mathbb{R}$. Then the latter equation is
\[(a-\frac{1}{1+c})^2+b^2=\frac{c^2}{(1+c)^2}.\]
This means $\tau$ lies on a circle centered at $(\frac{1}{1+c},0)$. Obviously, this circle has non-empty intersection with the open unit disc and we can thus make the claimed choice.

Now we assume $\theta(0)=0$ and $\theta'(0)=1$. We can choose $\gamma$ and $\tau$ in (\ref{mo}) such that the phase function $\tilde{\theta}$ for $\tilde{B}(\lambda)$ satisfies $\tilde{\theta}(0)=0$, $\tilde{\theta}'(0)=1$ and $\tilde{\theta}''(0)=0$. As in the above argument, now $\tau=a+\mathrm{i}b$ should lie on the circle $C_0$
\[(a-\frac{1}{2})^2+b^2=\frac{1}{4},\]
$\gamma=\frac{1-\bar{\tau}}{1-\tau}$ and
\[\theta''(0)+2\Im\frac{\tau}{(1-\tau)^2}=0.\]
Note that except for the point $(1,0)$, $C_0$ lies totally in the open unit disc. Let $\tau=\frac{1}{2}+\frac{1}{2}e^{\mathrm{i}\alpha}\in C_0$ with $\alpha\in (0,2\pi)$. We find that $\Im\frac{\tau}{(1-\tau)^2}=\cot\frac{\alpha}{2}$. This shows that we have the freedom to choose $\alpha$ such that $2\cot\frac{\alpha}{2}=-\theta''(0)$. This completes the proof of the lemma.
\end{proof}
\begin{theorem}\label{complete}The curvature function $\omega(u)$ for $u\in \mathbb{R}$ is a complete unitary invariant of $T\in \mathfrak{E}_1(H)$.
\end{theorem}
\begin{proof}Let $T_1$, $T_2$ be two entire operators with deficiency index 1 such that they share the same $\omega(u)$. By the above lemma we can assume that suitable boundary triplets have been chosen such that the corresponding phase functions $\theta_1$ and $\theta_2$ both have the Cauchy initial values
\[\theta(0)=0,\quad \theta'(0)=1,\quad \theta''(0)=0.\]
On the other side, $\theta_1, \theta_2$ are both solutions of the ordinary differential equation of third order
\[\frac{2}{3}\frac{\theta'''(u)}{\theta'(u)}-(\frac{\theta''(u)}{\theta'(u)})^2+\frac{(\theta'(u))^2}{3}=4\omega(u).\]
The uniqueness of solutions of the Cauchy problem for this differential equation then implies $\theta_1(u)\equiv \theta_2(u)$ on $\mathbb{R}$, and consequently the corresponding contractive Weyl functions $B_1(\lambda)$ and $B_2(\lambda)$ share the same boundary values on the real line. By the Identity Theorem, $B_1(\lambda)\equiv B_2(\lambda)$ and the proposition follows.
\end{proof}
\begin{example}\label{ex}Let $B(\lambda)=e^{\mathrm{i}b\lambda}$ for $b>0$. Then $\theta(u)=bu$ and $\omega(u)=\frac{b^2}{12}$. By the above proposition we see that this is the only case where $\omega(u)$ is constant.
\end{example}

\emph{Remark}. Let $\sigma(u):=\ln \theta'(u)$. Then $\omega(u)$ can be written in a more compact form
\[\omega(u)=\frac{1}{6}[\frac{1}{2}e^{2\sigma}+\sigma''-\frac{1}{2}(\sigma')^2].\]
Due to Prop.~\ref{complete}, it's interesting to know the precise condition for a function $f(u)$ to be the curvature function $\omega(u)$ of a certain entire operator $T$. Up to now we only know the necessary condition that $f(u)$ should be real-analytic and non-negative on the real line $\mathbb{R}$.

There is another interpretation of the curvature function $\omega(u)$. Let $\mathcal{L}\subset \mathcal{M}$ be the submanifold of Lagrangian subspaces in $\mathcal{B}_T$. Then when restricted on the real line $\mathbb{R}$, $W_T(u)$ is a real-analytic curve in $\mathcal{L}$. $\mathcal{L}$ carries a natural transitive $\mathbb{U}(1,1)$-action \cite[Example 3.17]{wang2024complex} and what interests us is the $\mathbb{U}(1,1)$-invariant properties of $W_T(u)$. What we have proved is that $\omega(u)$ is such a complete differential invariant. If a boundary triplet $(\mathbb{C}, \Gamma_0,\Gamma_1)$ is chosen, then $\mathcal{L}$ is realized as the real projective line $\mathbb{RP}^1\cong \mathbb{R}\cup \{\infty\}$, the $\mathbb{U}(1,1)$-action amounts to ordinary projective transformations, and $W_T(u)$ can be represented locally by the Weyl function $M(u)$. It is well-known that the Schwarzian derivative
\[(\mathcal{S}M)(u):=(\frac{M''(u)}{M'(u)})'-\frac{1}{2}(\frac{M''(u)}{M'(u)})^2\]
is a projective invariant of $W_T(u)$.
\begin{proposition}\label{Schw}In terms of the Weyl function $M(u)$, $\omega(u)=\frac{1}{6}(\mathcal{S}M)(u)$.
\end{proposition}
\begin{proof}Since $M(u)=\mathrm{i}\frac{1+e^{\mathrm{i}\theta(u)}}{1-e^{\mathrm{i}\theta(u)}}=-\cot \frac{\theta(u)}{2}$, the result follows from a direct computation.
\end{proof}
Since $\omega(u)$ is a complete unitary invariant of $T$, it's natural to ask how $\omega(u)$ captures the growth property of $T$. The interpretation of Prop.~\ref{Schw} suggests we consider the following Sturm-Liouville equation on the real line $\mathbb{R}$:
\begin{equation}y''+3\omega(u)y=0.\label{SL}\end{equation}
Let $y_1,y_2$ be any two linearly independent real-valued solutions to this equation and set $g:=y_1/y_2$. Then it's well-known that $(\mathcal{S}g)(u)=6\omega(u)$. Since any solution $g$ to the equation $(\mathcal{S}g)(u)=6\omega(u)$ has to be of this form, there are solutions $y_1$ and $y_2$ to Eq.~(\ref{SL}) such that $M(u)=y_1(u)/y_2(u)$. By the classical Sturm-Liouville theory (see \cite[\S~27]{walter1998ordinary}), $y_1$ and $y_2$ have no common zeros and each zero of them is simple. Hence we see that the zeros of $M(u)$ (i.e., eigenvalues of $T_{-1}$) on the real line are precisely the zeros of $y_1$. This is how $\omega(u)$ controls the distribution of eigenvalues of any self-adjoint extension of $T$ and $\sqrt{3\omega(u)}$ can be interpreted as a time-varying angular frequency of an oscillation system. Loosely speaking, the larger $\omega(u)$ is, the more concentrated the distribution of eigenvalues will be.
\begin{proposition}\label{compare}Let $T, \hat{T}\in \mathfrak{E}_1(H)$ and $\omega(u)$, $\hat{\omega}(u)$ be their curvature functions respectively such that $\omega(u)\not\equiv \hat{\omega}(u)$. If $\lambda_1<\lambda_2$ are two distinct eigenvalues of a certain self-adjoint extension of $T$ and $\omega(u)\leq \hat{\omega}(u)$ on $(\lambda_1, \lambda_2)$, then in $(\lambda_1, \lambda_2)$ there must be at least one eigenvalue of any self-adjoint extension of $\hat{T}$. In particular, if $\omega(u)\leq \hat{\omega}(u)$ on the whole real line, then $\rho_{T}\leq \rho_{\hat{T}}$.
\end{proposition}
\begin{proof}Since $\omega(u)\not\equiv \hat{\omega}(u)$ on $\mathbb{R}$, by the Identity Theorem this holds as well on $(\lambda_1, \lambda_2)$. Then the first claim is an immediate consequence of the Theorem of Sturm-Picone \cite[\S~27]{walter1998ordinary}. The second claim follows from the first.
\end{proof}
\begin{corollary}Let $\omega(u)$ be the curvature function of $T\in \mathfrak{E}_1(H)$. If on the interval $(-r, r)$ we have $0<m_1(r)\leq \omega(u)\leq m_2(r)$, then the number $n_T(r,m)$ of eigenvalues in $(-r,r)$ of any self-adjoint extension of $T_m$ satisfies
\[\lfloor\frac{2r\sqrt{3m_1(r)}}{\pi}\rfloor-1\leq n_T(r,m)\leq \lfloor\frac{2r\sqrt{3m_2(r)}}{\pi}\rfloor+1,\]
where $\lfloor\cdot \rfloor$ is the floor function.
\end{corollary}
\begin{proof}We can compare the curvature function $\omega(u)$ with the curvature function $\frac{b^2}{12}$ in Example \ref{ex} where $b=\sqrt{12m_j(r)}$, $j=1,2$. Then the result follows from counting the number of zeros of $\cos \frac{bu}{2}$ in $(-r,r)$.
\end{proof}
\begin{proposition}Let $\omega(u)$ be the curvature function of $T\in \mathfrak{E}_1(H)$ and $n_T(r,m)$ the number of eigenvalues of any self-adjoint extension $T_m$ in the interval $(-r,r)$. Then
\[n_T(r,m)< \frac{\sqrt{6}}{2}(r\int_{-r}^r \omega(u)du)^{1/2}+1.\]
\end{proposition}
\begin{proof}This is a direct consequence of Coro.~5.2 in \cite{hartman2002ordinary}.
\end{proof}
\begin{corollary}Let $\omega(u)$ be the curvature function of $T\in \mathfrak{E}_1(H)$. Then the Weyl order $\rho_T$ satisfies
\[\rho_T\leq \frac{1}{2}+\frac{1}{2}\limsup_{n\rightarrow +\infty}\frac{\ln (\int_{-r}^r\omega(u)du)}{\ln r}.\]
\end{corollary}
\begin{proof}Note that in the above proposition $n_T(r,m)$, the corresponding counting function $N_T(r,m)$ and the height function $h_T(r)$ all have the same order. Then the estimation holds by definition.
\end{proof}
Of course, in practice the above estimations could be of use only when one can find the asymptotic behavior of $\omega(u)$ to some extent.

\section{The structure of characteristic functions}\label{char}
A \emph{concrete} entire operator $T$ may contain too much irrelevant information if one is only concerned with the spectral theory of $T$, e.g., $T$ may be a differential operator and defined by local data, but the spectral theory of $T$ is definitely non-local. The importance of the Weyl curve $W_T(\lambda)$ lies in that it almost eliminates all the irrelevant information and reduces the spectral theory of $T$ to the study of a single finite-dimensional (unfortunately non-linear) object.

In practice, it is often very hard to write down the (contractive) Weyl functions of a concrete entire operator explicitly. However, if we reverse the direction, then a careful study of the possibility of the Weyl curves would convey in principle how complicated the spectral theory of an entire operator can be and what property of a concrete entire operator should be our focus. Through some special choices, the Weyl curve can be represented in terms of certain holomorphic or meromorphic functions, e.g., the functions $B(\lambda)$ and $M(\lambda)$, etc. In the literature, such functions have rather different names, say, $Q$-function, Lisvic characteristic function, etc. We shall roughly call all of them characteristic functions of $T$ and emphasize their common geometric origin. The focus of this section is the structure of some of these functions and its effect on spectral theory. Function theory of one complex variable then enters in spectral theory of $T$ in this manner.
\subsection{Contractive Weyl function}
Given a boundary triplet $(\mathbb{C}, \Gamma_+, \Gamma_-)$ for $T\in \mathfrak{E}_1(H)$, its corresponding contractive Weyl function $B(\lambda)$ is a meromorphic inner function over $\mathbb{C}$. By the canonical Riesz-Smirnov factorization, such functions are all of the form \cite[Thm.~3.2]{martin2018function}
\begin{equation}B(\lambda)=\gamma e^{\mathrm{i}b\lambda}\prod_{k=1}^N\frac{\overline{\lambda_k}}{\lambda_k}\frac{\lambda-\lambda_k}{\lambda-\overline{\lambda_k}},\label{cf}\end{equation}
where $\gamma$ is a constant with $|\gamma|=1$ and $b$ a non-negative constant, and $\{\lambda_k\}_{k=1}^N$ ($N\in \mathbb{N}\cup \{0, \infty\}$) is the sequence of zeros (counted with multiplicity) of $B(\lambda)$ on $\mathbb{C}_+$. When $N=\infty$, it is necessary that $\{\lambda_k\}_{k=1}^\infty$ has no finite accumulation points and that the Blaschke condition $\sum_{k=1}^\infty \frac{\Im \lambda_k}{|\lambda_k|^2}< +\infty$ holds. Conversely, each function of the form (\ref{cf}) with the above properties is a meromorphic inner function and hence a contractive Weyl function for an entire operator $T$ with deficiency index 1. $T$ is transcendental if and only if $b>0$ or $N=\infty$. $T$ is densely defined if and only if $b>0$ or $\sum_k\Im \lambda_k=+\infty$ \cite[Coro.~3.1.3]{martin2011representation}.

The Blaschke condition is a restriction on the distribution of zeros of $B(\lambda)$. Actually it implies that the eigenvalues of $T_0$ are almost located near the real line $\mathbb{R}$ in the following sense.
\begin{proposition}\label{p2}Let $\delta>0$ be a small positive constant and $\{\lambda_{k_j}\}$  zeros of $B(\lambda)$ in (\ref{cf}) that lie in the region
\[\Omega_\delta:=\{\lambda\in \mathbb{C}_+| \arg \lambda\in (\delta, \pi-\delta)\}.\]
Then $\sum_{j=1}^\infty\frac{1}{|\lambda_{k_j}|}<+\infty$.
\end{proposition}
\begin{proof}Let $\lambda_{k_j}=a_j+\mathrm{i}b_j$ for $a_j, b_j\in \mathbb{R}$. Note that
\[\frac{\Im \lambda_{k_j}}{|\lambda_{k_j}|^2}=\frac{b_j}{a_i^2+b_j^2}.\]
Since $\lambda_{k_j}\in \Omega_\delta$, we have $|a_j|<\cot \delta\cdot b_j$. Thus
\[\frac{1}{|\lambda_{k_j}|}=\frac{1}{\sqrt{a_j^2+b_j^2}}=\frac{\sqrt{a_j^2+b_j^2}}{a_j^2+b_j^2}< \frac{b_j\csc \delta}{a_j^2+b_j^2}.\]
Thus the Blaschke condition does imply the claim.
\end{proof}
\begin{corollary}If for $T\in \mathfrak{E}_1(H)$ there is a small number $\delta>0$ and a strictly dissipative boundary condition $m$ such that $\sigma(T_m)$ lies totally in $\Omega_\delta$ in Prop.~\ref{p2}, then the Weyl order $\rho_T\leq 1$.
\end{corollary}
\begin{proof}Let $B(\lambda)$ be a contractive Weyl function of $T$ such that the zero set $\{\lambda_i\}$ of $B(\lambda)$ coincides with $\sigma(T_m)$. Then due to Prop.~\ref{p2} the exponent of  convergence of $\{\lambda_i\}$ is not greater than 1. This implies that the order of the Blaschke product part in the factorization (\ref{cf}) is not greater than 1 \cite[Lemma 3.6]{kaltenbackbranges}. The result then follows.
\end{proof}

The number $-b$ in (\ref{cf}) is called the mean type of $B(\lambda)$. It's a basic result that
\begin{equation}-b=\limsup_{y\rightarrow+\infty}\frac{\ln|B(\mathrm{i}y)|}{y}.\label{mt}\end{equation}
An entire operator with $e^{\mathrm{i}b\lambda}$ as one contractive Weyl function arises naturally.
\begin{example}\label{exp}Consider $T(b)=-\mathrm{i}\frac{d}{dx}$ with domain $H_0^1(I)\subset L^2(I)$ where $I=[0,b]$ for a constant $b>0$ and $H_0^1(I)$ is the usual Sobolev space
 \[\{f\in L^2(I)|f'\in L^2(I), f(0)=f(b)=0\}.\] It's a classical result that $T(b)$ is a symmetric operator in $L^2(I)$ and $$D((T(b))^*)=H^1(I):=\{f\in L^2(I)|f'\in L^2(I)\}.$$ $T(b)$ is an entire operator with deficiency index 1. A natural boundary triplet can be chosen by setting $\Gamma_+\varphi=\varphi(0)$ and $\Gamma_-\varphi=\varphi(b)$ for $\varphi\in D(((T(b))^*)$. Then the corresponding contractive Weyl function is $B(\lambda)=e^{\mathrm{i}b\lambda}$.
\end{example}
$T\in \mathfrak{E}_1(H)$ which has Weyl order 1 is very special in the following sense.
\begin{proposition}If $T\in \mathfrak{E}_1(H)$ has a Picard exceptional boundary condition, then $\rho_T=1$. In particular, if $T$ has a Picard exceptional boundary condition $m$ such that $\sigma(T_m)=\varnothing$, then $T$ is unitarily equivalent to $T(b)$ in Example \ref{exp} for some $b>0$.
\end{proposition}
\begin{proof}Let $m$ be a Picard exceptional boundary condition. W.l.g., we assume that $T_m$ is strictly dissipative. Since $\mathbb{U}(1,1)$ acts transitively on the set of strictly dissipative boundary conditions, a boundary triplet can be chosen such that the extension $T_m$ is given by $\{x\in A_T^{\bot_s}|\Gamma_-x=0\}$. Thus zeros of $B(\lambda)$ are precisely points in $\sigma(T_m)$ and
\[B(\lambda)=\gamma e^{\mathrm{i}b\lambda}\prod_{k=1}^K\frac{\overline{\lambda_k}}{\lambda_k}\frac{\lambda-\lambda_k}{\lambda-\overline{\lambda_k}},\]
 where $K$ is a non-negative integer. Since $T$ is transcendental, we have $b>0$. The latter claim is clear.
\end{proof}
\emph{Remark}. The latter part of the above proposition is actually Thm.~4 in \cite[\S~114]{akhiezer2013theory}, but our emphasis is that the existence of Picard exceptional boundary conditions is a very strict restriction on $T$ and such boundary conditions won't occur for general entire operators $T\in \mathfrak{E}_1(H)$.

Given a strictly dissipative boundary condition $m$, we can always choose a boundary triplet such that $\sigma(T_m)$ is the zero set of the corresponding contractive Weyl function $B(\lambda)$. Such a boundary triplet is not unique, but $B(\lambda)$ can be determined up to a constant factor $\gamma$ with $|\gamma|=1$. Consequently the mean type of $B(\lambda)$ is well-defined and determined by $m$.

\begin{definition}In the above sense we call $b$ the mean type of the strictly dissipative boundary condition $m$ or the corresponding extension $T_m$.
 \end{definition}
\emph{Remark}. Since $m$ and its symplectic orthogonal complement $m^{\bot_s}$ determine each other mutually, we shall also call $b$ the mean type of $T_{m^{\bot_s}}$.

 From the viewpoint of inverse spectral theory, it matters whether the mean type of a boundary condition is zero or not.

\begin{proposition}For $T\in \mathfrak{E}_1(H)$, let $m$ be a strictly dissipative boundary condition whose mean type is zero. Then the set of all eigenvalues (counted with multiplicity) of $T_m$ determines the unitary equivalence class of $T$.
\end{proposition}
\begin{proof}We can choose a boundary triplet such that zeros of $B(\lambda)$ are precisely points in $\sigma(T_m)$. Then $B(\lambda)$ with $b=0$ in Eq.~(\ref{cf}) determines the Weyl class of $T$ and thus the unitary equivalence class of $T$.
\end{proof}
\emph{Remark}. If the strictly dissipative boundary condition $m$ has a non-vanishing mean type $b$, then the spectrum (with multiplicity) of $T_m$ only specifies the Blaschke product part of the corresponding contractive Weyl function $B(\lambda)$.
\begin{proposition}\label{order}If $T\in \mathfrak{E}_1(H)$ has a strictly dissipative boundary condition $m$ whose mean type is $b>0$, then the Weyl order $\rho_T\geq 1$.
\end{proposition}
\begin{proof}Choose a boundary triplet such that the zeros of $B(\lambda)$ produce $\sigma(T_m)$. Write $B(\lambda)$ in the form of (\ref{cf}). Note that for $u\in \mathbb{R}$, by differentiating we have
\[d\ln B(u)=\mathrm{i}bdu+2\mathrm{i}\Sigma_{k=1}^N\frac{\Im \lambda_k}{|u-\lambda_k|^2}du.\]
Thus
\[\frac{1}{2\pi \mathrm{i}}\int_{-t}^td\ln B(u)\geq \frac{bt}{\pi},\]
and consequently $h_T(r)\geq \frac{br}{\pi}+O(1)$. This shows that $\rho_T\geq 1$.
\end{proof}

The notion of Nevanlinna exceptional boundary conditions was defined in \cite{wang2024complex} but no example with $\delta_T(m)\in (0,1)$ was provided there. Here we give a simple example.
\begin{example}\label{defect}Let $b_1, b_2>0$ be two constants. Then both $B_1(\lambda)=e^{\mathrm{i}b_1\lambda}$ and $B_2(\lambda)=\frac{e^{\mathrm{i}b_2\lambda}-c}{1-\bar{c}e^{\mathrm{i}b_2\lambda}}$ ($0<|c|<1$) are meromorphic inner functions and the latter is congruent to $e^{\mathrm{i}b_2\lambda}$. They are contractive Weyl functions for $T(b_1)$ and $T(b_2)$ in Example \ref{exp} respectively. We can consider the meromorphic inner function $B(\lambda)=B_1(\lambda)\cdot B_2(\lambda)$, which can be viewed as a contractive Weyl function for a certain entire operator $T$ with deficiency index 1. If $\theta_1$, $\theta_2$ and $\theta$ are the phase functions for these three functions, then clearly
\[\theta'(u)=\theta_1'(u)+\theta_2'(u).\]
Thus the formulae (\ref{height}) and (\ref{se}) imply that \[h_T(r)=\frac{(b_1+b_2)r}{\pi}+O(\ln r).\]
Let $m$ be the boundary condition for $T$ such that $\sigma(T_m)$ is just the zero set of $B(\lambda)$ and $N_T(r, m)$ the corresponding counting function. However, the zero set of $B(\lambda)$ coincides with that of $B_2(\lambda)$. We know $\lim_{r\rightarrow +\infty} \frac{\pi N_T(r,m)}{b_2 r}=1$ because $0$ and $\infty$ are the only two Nevanlinna exceptional values of $e^{\mathrm{i}b_2\lambda}$. Therefore
\[\lim_{r\rightarrow +\infty}\frac{N_T(r,m)}{h_T(r)}=\frac{b_2}{b_1+b_2},\]
which implies that $\delta_T(m)=\frac{b_1}{b_1+b_2}$. In particular, we see that $\delta_T(m)$ can assume any value in $(0,1)$.
\end{example}
\subsection{Weyl function}
Given a boundary triplet $(\mathbb{C}, \Gamma_0, \Gamma_1)$ for $T\in \mathfrak{E}_1(H)$, its corresponding Weyl function $M(\lambda)$ is a meromorphic Herglotz function, which has the following canonical representation
\begin{equation}M(\lambda)=c+d\lambda+\sum_n (\frac{1}{t_n-\lambda}-\frac{t_n}{1+t_n^2})w_n,\label{Horg}
\end{equation}
where $c\in \mathbb{R}$, $d\geq 0$, (1) $\{w_n\}$ is a sequence of strictly positive constants and (2) $\{t_n\}$ is a purely discrete, strictly increasing real sequence with no finite accumulation point. It is necessary that (3) $\{w_n\}$ and $\{t_n\}$ satisfy the compatibility condition $\sum_{n}\frac{w_n}{1+t_n^2}< +\infty$. Note that here $n\in \mathbb{N}$, $-\mathbb{N}$ or $\mathbb{Z}$ according to whether $\{t_n\}$ is semi-bounded or not. Note that $\{t_n\}$ is precisely $\sigma(T_1)$ (in terms of the associated boundary triplet $(\mathbb{C}, \Gamma_+,\Gamma_-)$). It is known that $T$ is densely defined if and only if $d=0$ and $\sum_n w_n=+\infty$ \cite[Coro.~3.7.3]{behrndt2020boundary} \cite[Lemma 2.13]{martin2018function}.

Conversely, for $c\in \mathbb{R}$, $d\geq 0$ and sequences $\{t_n\}$, $\{w_n\}$ satisfying the above conditions (1)(2)(3),\footnote{In \cite{martin2018function}, such a pair is called a bandlimit pair.} the function defined via (\ref{Horg}) is a meromorphic Herglotz function and can be viewed as the Weyl function of an entire operator $T$ w.r.t. a certain boundary triplet. Then the sequence $\{t_n\}$ is nothing else but the spectrum of $T_1$. Combined with the formula (\ref{height}), this provides an opportunity to construct entire operators of different Weyl orders.
\begin{example}Let $\alpha>0$ be a fixed constant. For $n\in \mathbb{N}$, set $t_n=n^\alpha$ and $w_n=\frac{1}{n}$. Then clearly
\[\sum_{n=1}^{\infty}\frac{w_n}{1+t_n^2}=\sum_{n=1}^{\infty}\frac{1}{n+n^{2\alpha+1}}<+\infty.\]
Thus
\[M(\lambda)=\sum_{n=1}^\infty (\frac{1}{t_n-\lambda}-\frac{t_n}{1+t_n^2})w_n\]
is a meromorphic Herglotz function and ought to be the Weyl function of a certain densely defined entire operator $T$ w.r.t. a boundary triplet. Then $\{n^\alpha\}$ should be $\sigma(T_1)$ and in this case
\[r^{1/\alpha}-1\leq n_T(r, 1)< r^{1/\alpha}.\]
Then by Eq.~(\ref{se}), the Weyl order $\rho_T=\frac{1}{\alpha}$. Thus we see that the Weyl order of $T\in \mathfrak{E}_1(H)$ can be an arbitrary positive number.
\end{example}
\begin{example}For each $n\in \mathbb{N}$, set $t_n=\ln n$ and $w_n=\frac{1}{n^2}$. Then clearly
\[M(\lambda)=\sum_{n=1}^\infty (\frac{1}{t_n-\lambda}-\frac{t_n}{1+t_n^2})w_n\]
is a meromorphic Herglotz function, which can be viewed as the Weyl function of a not densely defined entire operator $T$. Then $\{\ln n\}$ should be $\sigma(T_1)$ and in this case
\[e^r-1\leq n_T(r,1)<e^r.\]
Consequently we have $\rho_T=+\infty$.
\end{example}
\subsection{de Branges function}
Recall that an Hermite-Biehler function $e: \mathbb{C}\rightarrow \mathbb{C}$ is an entire function such that (1) $|e(\lambda)|> |e^\sharp(\lambda)|$ for each $\lambda\in \mathbb{C}_+$, and (2) $e$ has no zeros on the real line. Here $e^\sharp(\lambda)=\overline{e(\bar{\lambda})}$.

Let $B(\lambda)$ be a contractive Weyl function of $T\in \mathfrak{E}_1(H)$, written in the form of (\ref{cf}). It is always possible to find an Hermite-Bielher function $e(\lambda)$ such that $B(\lambda)=\frac{e^\sharp(\lambda)}{e(\lambda)}$. We say $e(\lambda)$ is a de Branges function for $B(\lambda)$. This $e(\lambda)$ is not unique: If $c(\lambda)$ is a real function (i.e. $c^\sharp\equiv c$) vanishing nowhere on $\mathbb{C}$, then $c\times e$ is also a de Branges function for $B(\lambda)$. A canonical choice of $e(\lambda)$ is as follows: Let $\{\lambda_j\}$ be poles of $B(\lambda)$ on $\mathbb{C}_-$, arranged such that $|\lambda_j|\leq |\lambda_{j+1}|$. We can choose
\[e(\lambda)=e^{\frac{-\mathrm{i}b\lambda}{2}}\prod_{j=1}^N(1-\frac{\lambda}{\lambda_j})\Re(p_j(\lambda))\]
where $p_j(\lambda)=\sum_{k=1}^j\frac{\lambda^k}{k\lambda_j^k}$. See for example \cite[Lemma 2.1]{havin2003admissible}. These factors $\Re(p_j(\lambda))$ are present to make sure the convergence of the infinite product. If additionally we know in advance $\rho_T<1$, we should have $b=0$ and can omit these extra factors. See \cite[Lemma 3.12]{kaltenbackbranges}.

For a given de Branges function $e(\lambda)$, let
\[a(\lambda)=\frac{e(\lambda)+e^\sharp(\lambda)}{2},\quad b(\lambda)=\mathrm{i}\frac{e(\lambda)-e^\sharp(\lambda)}{2}.\]
Then $B(\lambda)-1=\frac{2\mathrm{i}b(\lambda)}{a(\lambda)-\mathrm{i}b(\lambda)}$ and we see that $\sigma(T_1)$ is precisely the zero set of the entire function $b(\lambda)$. Similarly, $\sigma(T_{-1})$ is just the zero set of the entire function $a(\lambda)$.

The pair $[e(\lambda), e^\sharp(\lambda)]$ can be viewed as a representation of the Weyl curve $W_T(\lambda)$ in terms of projective coordinates of $\mathbb{CP}^1$. In the same way, a generic boundary condition can be represented by a point $[a,b]\in \mathbb{CP}^1$. Then $\lambda$ is a corresponding eigenvalue if and only if
\[\det \left(
         \begin{array}{cc}
           e(\lambda) & e^\sharp(\lambda) \\
           a & b \\
         \end{array}
       \right)=0.
\]
In particular, $\lambda$ is a multiple eigenvalue if and only if
\[G(\lambda):=\det \left(
         \begin{array}{cc}
           e(\lambda) & e^\sharp(\lambda) \\
           e'(\lambda) & (e^\sharp)'(\lambda) \\
         \end{array}
       \right)=0.
\]
\begin{proposition}\label{mul}For $T\in \mathfrak{E}_1(H)$, there are at most countably many generic boundary conditions possessing a multiple eigenvalue. In particular, if the Weyl order $\rho_T<1$, then such a generic boundary condition always exists.
\end{proposition}
\begin{proof}It's easy to see the entire function $G(\lambda)$ is never trivial. Since $\lambda_0$ is a multiple eigenvalue for a certain generic extension of $T$ if and only if $\lambda_0$ is a zero of $G(\lambda)$. If it is the case, the boundary condition is provided by the point $[e(\lambda_0),e^\sharp(\lambda_0)]\in \mathbb{CP}^1$. The claim then follows because the zeros of $G(\lambda)$ are at most countably many. If furthermore $\rho_T<1$, then $e(\lambda)$ can be chosen to be of order $< 1$ and consequently the order of $G(\lambda)$ is less than 1 as well. The second claim follows since $G(\lambda)$ always has infinite zeros \cite[Thm.~1.1, Chap.~4]{gol'dberg2008value}.
\end{proof}

\section{Some other existing models compared with the canonical model}\label{S4}
Besides our model in \S~\ref{model}, there are still several other functional models for an entire operator $T\in \mathfrak{E}_1(H)$, in which the model spaces are all function spaces over a certain subset of $\mathbb{C}$ and the model operators are just multiplication by the independent variable. See for instance \cite{ross2013theorem}. These models were proposed independently and sometimes one appears to be rather different from another. The basic goal of this section is to show that our model is\emph{ universal} in the sense that the others can often be derived from it by two natural geometric operations--restricting $F$ onto a certain subset $N$ of $\mathbb{C}$ and choosing a suitable trivialization of $F$ over $N$. This certainly means there are even more functional models not mentioned in the literature at all!

In the following, we assume that $M$ is a complex manifold, although it will finally be specified to be $\mathbb{C}$ or an open subset of it.
 \begin{definition}For a holomorphic vector bundle $F$ over a complex manifold $M$,  a subset $N\subset M$ is called a determining subset of $F$ if any holomorphic section vanishing on $N$ also vanishes on the whole of $M$.
 \end{definition}

 It should be mentioned that in the above definition since $N$ may even be a discrete subset, $E|_N$ may not be a holomorphic vector bundle at all.

 \begin{lemma}\label{l1}Let $K$ be a reproducing kernel on a holomorphic vector bundle $F$ over $M$ and $\mathfrak{H}(K)$ the corresponding reproducing kernel Hilbert space, consisting of certain holomorphic sections of $F$. If $N$ is a determining subset for $F$, then the restriction $K_N$ of $K$ on $N$ is a reproducing kernel on $F|_N$ and the restriction map
 \[r_N^*: \mathfrak{H}(K)\rightarrow \mathfrak{H}(K_N),\  s\mapsto s|_N\]
 is a unitary map.
 \end{lemma}
 \begin{proof}This is essentially \cite[Prop.~1.6]{bertram1998reproducing} except that the denseness of $N$ is replaced with the determining property.
 \end{proof}
 Now let $M\subseteq \mathbb{C}$ be an open subset and $F$ a holomorphic vector bundle over $M$ with a reproducing kernel $K$ such that $\mathfrak{H}(K)$ consists of certain holomorphic sections. Since for each holomorphic section $s$ of $F$, the section $\mathfrak{X}s$ defined by $(\mathfrak{X}s)(\lambda):=\lambda s(\lambda)$ for $\lambda\in M$ is again holomorphic, we can define
 \[D(\mathfrak{T}_K)=\{s\in \mathfrak{H}(K)|\mathfrak{X}s\in\mathfrak{H}(K)\}\]
 and obtain an operator $\mathfrak{T}_K$ in $\mathfrak{H}(K)$ via $\mathfrak{T}_K s=\mathfrak{X}s$ for each $s\in D(\mathfrak{T}_K)$. Note that if $N\subseteq M$, we can define $\mathfrak{T}_{K_N}$ in $\mathfrak{H}(K_N)$ in a similar manner.
 \begin{lemma}\label{p1}Let $F$ be a holomorphic vector bundle over $M\subseteq\mathbb{C}$ and $K$ a reproducing kernel on $F$ such that $\mathfrak{H}(K)$ consists of certain holomorphic sections. If $N\subset M$ is a determining subset for $F$, then $\mathfrak{T}_K$ is unitarily equivalent to $\mathfrak{T}_{K_N}$. More precisely, $r_N^*(D(\mathfrak{T}_K))=D(\mathfrak{T}_{K_N})$ and for each $s\in D(\mathfrak{T}_K)$
 \begin{equation}r_N^*(\mathfrak{T}_K s)=\mathfrak{T}_{K_N}r_N^*(s).\label{e3}\end{equation}
 \end{lemma}
 \begin{proof}Obviously $r_N^*(D(\mathfrak{T}_K))\subset D(\mathfrak{T}_{K_N})$. Conversely, if $s\in D(\mathfrak{T}_{K_N})$, then in $\mathfrak{H}(K)$ there is a unique extension $\tilde{s}$ (resp. $\widetilde{\mathfrak{T}_{K_N}s}$) of $s$ (resp.~$\mathfrak{T}_{K_N}s$). Since $\mathfrak{X}\tilde{s}-\widetilde{\mathfrak{T}_{K_N}s}$ vanishes on $N$ and $N$ is determining, it is necessary that $\mathfrak{X}\tilde{s}=\widetilde{\mathfrak{T}_{K_N}s}$. Thus $\tilde{s}\in D(\mathfrak{T}_K)$ and consequently $r_N^*(D(\mathfrak{T}_K))=D(\mathfrak{T}_{K_N})$. Eq.~(\ref{e3}) is clear.
 \end{proof}
\subsection{de Branges model}\label{deBranges}
In this and the following subsections, we shall continue to use the notations in \S~\ref{line}. For $T\in \mathfrak{E}_1(H)$, we choose a fixed boundary triplet $(\mathbb{C}, \Gamma_+, \Gamma_-)$. Let $B(\lambda)$ be the associated contractive Weyl function and $e(\lambda)$ an Hermite-Biehler function for it, i.e., $B(\lambda)=e^\sharp(\lambda)/e(\lambda)$. Since zeros of $e(\lambda)$ are precisely poles of $B(\lambda)$, the section $\zeta(\lambda):=e(\lambda)\gamma_+ (\lambda)$ is a holomorphic section of $E$, vanishing nowhere on $\mathbb{C}$. As a consequence, $\xi(\lambda):=\psi(\lambda)/e^\sharp(\lambda)$ as a holomorphic section of $F$ vanishes nowhere on $\mathbb{C}$ as well. Furthermore
\[(\xi(\lambda),\xi(\lambda))_\lambda=-\frac{1}{|e^\sharp(\lambda)|^2}\times\frac{2\Im \lambda}{1-|B(\bar{\lambda})|^2}=\frac{2\Im \lambda}{|e(\lambda)|^2-|e(\bar{\lambda})|^2}.\]
In terms of the global frames $\zeta(\bar{\lambda})$ and $\xi(\lambda)$, the model space $\mathfrak{H}$ in \S~\ref{model} is now interpreted as a reproducing kernel Hilbert space $\mathcal{H}_e$ of entire functions.
\begin{proposition} $\mathcal{H}_e$ is just the de Branges space associated to the Hermite-Biehler function $e$.
\end{proposition}
\begin{proof}The reproducing kernel of $\mathcal{H}_e$ is
\begin{eqnarray*}\mathrm{K}_1(\lambda, \mu)&=&\langle K(\lambda,\mu)\zeta(\bar{\mu}),\zeta(\bar{\lambda})\rangle=\langle\iota^\dag_\lambda\iota_\mu\zeta(\bar{\mu}), \zeta(\bar{\lambda})\rangle=(\zeta(\bar{\mu}), \zeta(\bar{\lambda}))_H\\
&=&e(\bar{\mu})e^\sharp(\lambda)(\gamma_+ (\bar{\mu}),\gamma_+ (\bar{\lambda}))_H=e(\bar{\mu})e^\sharp(\lambda)\times \mathrm{i}\frac{1-\overline{B(\bar{\lambda})}B(\bar{\mu})}{\bar{\mu}-\lambda}\\
&=&\mathrm{i}\frac{e(\lambda)e^\sharp(\bar{\mu})-e(\bar{\mu})e^\sharp(\lambda)}{\lambda-\bar{\mu}}.
\end{eqnarray*}
This is precisely the reproducing kernel (up to a constant factor) of the de Branges space associated to $e$. See for instance \cite{kaltenbackbranges}. Due to the uniqueness of the reproducing kernel, $\mathcal{H}_e$ is just the de Branges space associated to $e$.
\end{proof}

It is not trivial at all that an entire function $f(\lambda)\in \mathcal{H}_e$ if and only if both $f(\lambda)/e(\lambda)$ and $f^\sharp(\lambda)/e(\lambda)$ are in the Hardy space $H^2(\mathbb{C}_+)$ and for $f,g\in \mathcal{H}_e$.
 \[(f(\cdot), g(\cdot))_{\mathcal{H}_e}=\int_\mathbb{R}\frac{f(x)\overline{g(x)}}{|e(x)|^2}\frac{dx}{2\pi}.\]
Recall that an analytic function $f$ is of bounded type on $\mathbb{C}_+$ if it is the quotient of two bounded analytic functions on $\mathbb{C}_+$\footnote{That $f$ is of bounded type is also equivalent to that $\ln^+|f|$ has a positive harmonic majorant on $\mathbb{C}_+$. Here $\ln^+ x=\max\{0, \ln x\}$.}. It is also necessary and sufficient that for $f$ to belong to $\mathcal{H}_e$, $f$ should satisfy (i) $\int_\mathbb{R}|\frac{f(u)}{e(u)}|^2du<+\infty$ and (ii) both $f(\lambda)/e(\lambda)$ and $f^\sharp(\lambda)/e(\lambda)$ have to be of bounded type on $\mathbb{C}_+$ and of non-positive mean type (defined via the same Eq.~(\ref{mt})). See \cite[Prop.~2.1]{remling2002schrodinger}. As a result, we have the following
\begin{proposition} If $s\in \mathfrak{H}$ is a nontrivial section of the characteristic line bundle $F$ of $T\in \mathfrak{E}_1(H)$ and $\{\lambda_i\}_{i\in \mathbb{N}}$ are zeros of $s$ (counted with multiplicity), then
\begin{equation}\sum_{i=1}^\infty |\Im(\frac{1}{\lambda_i})|<+\infty.\label{e4}\end{equation}
\end{proposition}
\begin{proof}This is just a geometric reformulation of part of Lemma 2.3 in \cite{kaltenbackbranges}.
\end{proof}
\emph{Remark}. As the Blaschke condition, (\ref{e4}) also implies that the zeros of $s$ are located near the real line. See \cite[Lemma 2.2]{kaltenbackbranges}.

If $B_1(\lambda)$ and $B_2(\lambda)$ are two contractive Weyl functions of $T\in \mathfrak{E}_1(H)$ and $e_1(\lambda)$, $e_2(\lambda)$ the corresponding de Branges functions, then it's possible that $\mathcal{H}_{e_1}=\mathcal{H}_{e_2}$. Actually we have
\begin{proposition} $\mathcal{H}_{e_1}$ is equal to $\mathcal{H}_{e_2}$ as reproducing kernel Hilbert spaces if and only if
\[e_2(\lambda)=ae_1(\lambda)+be_1^\sharp(\lambda),\quad \forall \lambda\in \mathbb{C},\]
where $a,b$ are complex constants such that $|a|^2-|b|^2=1$.
\end{proposition}
\begin{proof}This is a direct consequence of the uniqueness of reproducing kernel.
\end{proof}
\emph{Remark}. Thus $(e_1, e_1^\sharp)$ and $(e_2, e_2^\sharp)$ are related via
\[\left(
    \begin{array}{c}
      e_2 \\
      e_2^\sharp \\
    \end{array}
  \right)=
\left(
    \begin{array}{cc}
      a & b \\
      \bar{b} & \bar{a} \\
    \end{array}
  \right)\left(
           \begin{array}{c}
             e_1 \\
             e_1^\sharp \\
           \end{array}
         \right),
\]where $\left(
    \begin{array}{cc}
      a & b \\
      \bar{b} & \bar{a} \\
    \end{array}
  \right)\in \mathbb{SU}(1,1)$. The appearance of $\mathbb{SU}(1,1)$ here is not strange since it preserves the (strong) symplectic structure $[\cdot, \cdot]_1$ on $\mathbb{C}^2$.

Since de Branges' early work in 1960s, the theory of de Branges spaces has played a role in several different areas. Our emphasis here is that this theory is naturally included in spectral theory of entire operators.
\subsection{de Branges-Rovnyak model}\label{br}
 By the Identity Theorem, $\mathbb{C}_+$ is of course a determining subset of the characteristic line bundle $F$. Due to Lemma~\ref{p1}, $r^*_{\mathbb{C}_+}(\mathfrak{H})$ together with multiplication by the independent variable provides another functional model for $T$.

Given the boundary triplet $(\mathbb{C}, \Gamma_+, \Gamma_-)$, we can use $\gamma_-(\bar{\lambda})$ to trivialize $F^\dag$ over $\mathbb{C}_+$. Note that $\gamma_+(\bar{\lambda})=B(\bar{\lambda})\gamma_-(\bar{\lambda})$. The model space $\mathfrak{H}$ is now interpreted as a reproducing kernel Hilbert space $\mathcal{H}_B$ of analytic functions on $\mathbb{C}_+$.

 \begin{proposition} $\mathcal{H}_B$ is precisely the de Branges-Rovnyak space associated to $B(\lambda)$.
 \end{proposition}
 \begin{proof}
 The reproducing kernel of $\mathcal{H}_B$ is
\begin{eqnarray*}\mathrm{K}_2(\lambda,\mu)&=&\langle K(\lambda,\mu)\gamma_-(\bar{\mu}),\gamma_-(\bar{\lambda})\rangle=\frac{1}{B(\bar{\mu})\cdot \overline{B(\bar{\lambda})}}\times \mathrm{i}\frac{1-\overline{B(\bar{\lambda})}B(\bar{\mu})}{\bar{\mu}-\lambda}\\
&=&\mathrm{i}\frac{1-B(\lambda)\overline{B(\mu)}}{\lambda-\bar{\mu}}.
\end{eqnarray*}
This is precisely the reproducing kernel of the de Branges-Rovnyak space associated to $B(\lambda)$ up to a constant factor. See \cite{ross2013theorem} for more technical details.
\end{proof}

As in Eq.~(\ref{dual}), we can define the conjugate dual frame $\psi_-(\lambda)$ for $\gamma_-(\bar{\lambda})$. Since the two frames $\xi(\lambda)$ and $\psi_-(\lambda)$ are related via
\[\xi(\lambda)=\frac{\psi(\lambda)}{e^\sharp(\lambda)}=\frac{\psi_-(\lambda)}{\overline{B(\bar{\lambda})}\times e^\sharp(\lambda)}=\frac{B(\lambda)\times \psi_-(\lambda)}{e^\sharp(\lambda)}=\frac{\psi_-(\lambda)}{e(\lambda)}\] over $\mathbb{C}_+$, we thus naturally have $(\mathcal{H}_e)|_{\mathbb{C}_+}=e(\cdot)\times \mathcal{H}_B$, which is a well-known result. Another well-known result is that $\mathcal{H}_B$ is a closed subspace of the Hardy space $H^2(\mathbb{C}_+)$, i.e. $\mathcal{H}_B=(BH^2(\mathbb{C}_+))^\bot$.

It should also be pointed out that the de Branges-Rovnyak model actually can be defined on $\mathbb{C}\backslash \sigma(T_{\infty})$ and each element in $\mathcal{H}_B$ is analytic on $\mathbb{C}\backslash \sigma(T_{\infty})$ rather than only on $\mathbb{C}_+$. This fact will be used in the next section.
\subsection{Herglotz model}
This time we take $N=\mathbb{C}\backslash \mathbb{R}$. We use a boundary triplet $(\mathbb{C}, \Gamma_0, \Gamma_1)$ and let $M(\lambda)$ be the corresponding Weyl function. As in \S\S~\ref{bvp}, for each $\lambda\in \mathbb{C}\backslash \mathbb{R}$ there is a unique $\gamma(\lambda)\in E_\lambda$ such that
\[\hat{\gamma}(\lambda):=(\gamma(\lambda), \lambda \gamma(\lambda))\in W_\lambda\subset \mathbb{H}\]
and $\Gamma_0\hat{\gamma}(\lambda)=1$. Note that $\Gamma_1\gamma(\lambda)=M(\lambda)$ and thus $\gamma(\lambda)$ is a holomorphic frame of $E$ over $\mathbb{C}\backslash \mathbb{R}$. In particular, due to the abstract Green's formula
\begin{equation}(\gamma(\lambda),\gamma(\mu))_H=\frac{M(\lambda)-M(\bar{\mu})}{\lambda-\bar{\mu}}.\label{herg}\end{equation}
We use $\gamma(\bar{\lambda})$ to trivialize $F^\dag$ over $\mathbb{C}\backslash \mathbb{R}$ and then $r^*_{\mathbb{C}\backslash \mathbb{R}}(\mathfrak{H})$ is identified with a reproducing kernel Hilbert space $\mathcal{H}_M$ consisting of certain analytic functions on $\mathbb{C}\backslash \mathbb{R}$.
\begin{proposition}$\mathcal{H}_M$ is precisely the Herglotz space associated to $M(\lambda)$.
\end{proposition}
\begin{proof}The reproducing kernel of $\mathcal{H}_M$ is
\begin{eqnarray*}\mathrm{K}_3(\lambda,\mu)=\langle K(\lambda,\mu)\gamma(\bar{\mu}),\gamma(\bar{\lambda})\rangle=(\gamma(\bar{\mu}),\gamma(\bar{\lambda}))_H=\frac{M(\lambda)-M(\bar{\mu})}{\lambda-\bar{\mu}}.\end{eqnarray*}
This is just the reproducing kernel of the Herglotz space associated to $M(\lambda)$. See \cite{ross2013theorem} for more details.
\end{proof}
Let $(\mathbb{C}, \Gamma_+, \Gamma_-)$ be the boundary triplet associated to $(\mathbb{C}, \Gamma_0, \Gamma_1)$ via Cayley transform and $B(\lambda)$ the corresponding contractive Weyl function. If $e(\lambda)$ is a de Branges function for $B(\lambda)$ and $\zeta(\lambda)$ the corresponding frame in \S\S~\ref{deBranges}, then
\[(\Gamma_0, \Gamma_1)(\zeta(\lambda),\lambda\zeta(\lambda))=(\frac{e(\lambda)-e^\sharp(\lambda)}{\mathrm{i}\sqrt{2}}, \frac{e(\lambda)+e^\sharp(\lambda)}{\sqrt{2}}).\]
From this we know that on $\mathbb{C}\backslash \mathbb{R}$,
\[\zeta(\lambda)=\frac{e(\lambda)-e^\sharp(\lambda)}{\mathrm{i}\sqrt{2}}\gamma(\lambda).\]
This implies easily that the two spaces $\mathcal{H}_e$ and $\mathcal{H}_M$ are related via
\[(\mathcal{H}_e)|_{\mathbb{C}\backslash \mathbb{R}}=\frac{e(\lambda)-e^\sharp(\lambda)}{\mathrm{i}\sqrt{2}}\times \mathcal{H}_M,\]
and consequently by the argument in \S\S~\ref{br},
\[\mathcal{H}_B=\frac{1-B(\lambda)}{i\sqrt{2}}\times (\mathcal{H}_M)|_{\mathbb{C}_+}.\]
\subsection{Kempf-Martin model}
Now let $N$ be the real line $\mathbb{R}$. If $B(\lambda)$ is the contractive Weyl function of $T$ w.r.t. a boundary triplet $(\mathbb{C}, \Gamma_+, \Gamma_-)$ and $e(\lambda)$ is a de Branges function for $B(\lambda)$, as in \S\S~\ref{deBranges} we have the global frames $\zeta(\lambda)$ and $\xi(\lambda)$ for $E$ and $F$ respectively. Set $\tilde{\zeta}(\lambda)=\zeta(\lambda)/|\zeta(\lambda)|_\lambda$. On the real line, $\tilde{\zeta}(u)$ for $u\in \mathbb{R}$ is a frame for $E|_{\mathbb{R}}$, in terms of which $\mathfrak{H}|_{\mathbb{R}}$ is identified with a reproducing kernel Hilbert space $\mathcal{K}$ consisting of certain functions on the real line.
\begin{proposition}The reproducing kernel of $\mathcal{K}$ is
\[\mathrm{K}_4(u,v)=\frac{\mathrm{K}_1(u,v)}{N(u)\cdot N(v)},\]
where $N(u)=\sqrt{\mathrm{K}_1(u,u)}$. In particular, $(\mathcal{H}_e)|_{\mathbb{R}}= N(u)\times \mathcal{K}$.
\end{proposition}
\begin{proof}Since $|\zeta(u)|_u^2=(\zeta(u),\zeta(u))_H=\mathrm{K}_1(u,u)$, the conclusion is clear from the argument in \S\S~\ref{deBranges}.
\end{proof}

 $\mathcal{K}$ was first introduced in \cite{martin2018function} to define a notion of time-varying bandlimits in sampling theory. It was constructed from its reproducing kernel, which itself stems from a prescription of a bandlimit pair. This model was shown to be equivalent to the de Branges-Rovnyak model \cite{martin2018function} and later in \cite{marco2022spaces} an alternative definition of $\mathcal{K}$ in terms of the de Branges function was given and the name Kempf-Martin space was coined. In \cite{marco2022spaces} the reproducing kernel was the formula (6.20) .

 Since $\gamma_\pm(\lambda)$ are both well-defined on the real line, they can also be normalized to replace $\tilde{\zeta}(u)$ in the above argument and produce new models on the real line. Clearly, all these models differ from each other only by a unitary gauge transformation on $F|_\mathbb{R}$. We won't spell the details out any more.
 \subsection{A discrete model}
 Let $B(\lambda)$ be a contractive Weyl function for $T\in \mathfrak{E}_1(H)$ w.r.t. a fixed boundary triplet $(\mathbb{C}, \Gamma_+, \Gamma_-)$. Then $T_1$ is a self-adjoint extension of $T$, each eigenvalue of which is simple. W.l.g., we assume that $T_1$ is an operator. We choose $N=\sigma(T_1)$. Since there are many nontrivial holomorphic sections of $F$ vanishing on $\sigma(T_1)$, $\sigma(T_1)$ is in no way a determining subset of $F$. Let $\gamma_-(\bar{\lambda})$ be the frame of $F^\dag$ and $\psi_-(\lambda)$ the conjugate dual frame in $F$. Then $\{\widehat{\gamma_-(\lambda_j)}\}_{\lambda_j\in \sigma(T_1)}$ is a complete orthogonal subset in $\mathfrak{H}$. In particular, on one side, as in \S\S~\ref{br} we have
 \begin{eqnarray*}\|\widehat{\gamma_-(\lambda_j)}\|_{\mathfrak{H}}^2&=&(\gamma_-(\lambda_j), \gamma_-(\lambda_j))_{H}=\lim_{\varepsilon\rightarrow 0+}(\gamma_-(\lambda_j+\mathrm{i}\varepsilon), \gamma_-(\lambda_j+\mathrm{i}\varepsilon))_{H}\\
 &=&\lim_{\varepsilon\rightarrow 0+}\frac{1-|B(\lambda_j+\mathrm{i}\varepsilon)|^2}{2\varepsilon}=\theta'(\lambda_j),
 \end{eqnarray*}
 where $\theta(u)$ is the phase function for $B(\lambda)$ and Prop.~\ref{ka} is used. On the other side, for any $s\in \mathfrak{H}$
 \[(s(\cdot), \widehat{\gamma_-(\lambda_j)}(\cdot))_{\mathfrak{H}}=(s(\cdot), K(\cdot, \lambda_j)\gamma_-(\lambda_j))_{\mathfrak{H}}=\langle s(\lambda_j), \gamma_-(\lambda_j)\rangle.\]
 Due to the spectral resolution of $T_1$, for any $s\in \mathfrak{H}$, we ought to have
 \[s(\lambda)=\sum_{\lambda_j\in \sigma(T_1)}\langle s(\lambda_j), \gamma_-(\lambda_j)\rangle\frac{\widehat{\gamma_-(\lambda_j)}(\lambda)}{\theta'(\lambda_j)}.\]
 Of course this is essentially only a geometric reformulation of the well-known sampling formula in de Branges spaces. Thus if $s\in \mathfrak{H}$ is vanishing on $\sigma(T_1)$, it is necessary that $s(\lambda)\equiv 0$. From this variants of Lemma \ref{l1} and Lemma \ref{p1} hold and $\mathfrak{H}|_{\sigma(T_1)}$ is also a model space for $T$. The following proposition is clear.
 \begin{proposition}In terms of the frame $\gamma_-(\lambda)|_{\sigma(T_1)}$, $\mathfrak{H}|_{\sigma(T_1)}$ is identified with the reproducing kernel Hilbert space $L^2(\Gamma)$ where $\Gamma$ is the discrete measure
 \[\Gamma=\sum_{\lambda_j\in \sigma(T_1)}\frac{\delta_{\lambda_j}}{\theta'(\lambda_j)}\]
 and $\delta_{\lambda_j}$ is the Dirac measure at $\lambda_j$. The reproducing kernel of $L^2(\Gamma)$ is
 \[\mathrm{K}_5(\lambda_i, \lambda_j)=\theta'(\lambda_i)\times \delta_{ij}.\]
 \end{proposition}
 \emph{Remark}. Another description of this model can be found in \cite[\S~4]{martin2018function}. If $e(\lambda)$ is a de Branges function for $B(\lambda)$, in the above argument one can also use $\zeta(\lambda)$ in \S\S~\ref{deBranges} in place of $\gamma_-(\lambda)$ to get another model. However, the two models are related via a gauge transformation on $F|_{\sigma(T_1)}$.
\section{On the completeness problem}\label{com}
As an application of our previous investigation, we solve the completeness problem for generic extensions of $T\in \mathfrak{E}_1(H)$. Let $T_m$ be a generic extension of $T$ and $\lambda_0\in \sigma(T_m)$ of analytic multiplicity $k$. From \cite[Thm.~10.36]{wang2024complex} we know the algebraic multiplicity of $\lambda_0$ is also $k$. Then there are vectors $\varphi_0,\cdots, \varphi_{k-1}\in D(T_m)$, forming a Jordan chain of length $k$ for $T_m$ at $\lambda_0\in \sigma(T_m)$, i.e.,
\[\varphi_0\neq 0,\quad (T_m-\lambda_0)\varphi_0=0,\quad (T_m-\lambda_0)\varphi_j=\varphi_{j-1},\quad j=1,\cdots, k-1.\]
The subspace $E_{\lambda_0}:=\textup{span}\{\varphi_0, \varphi_1,\cdots, \varphi_{k-1}\}$ is called the generalized eigensubspace associated to $\lambda_0\in \sigma(T_m)$. One can use the Riesz projector
to cut out $E_{\lambda_0}$. This is the projection $P_{\lambda_0}$ defined by
\[P_{\lambda_0}=\frac{\mathrm{i}}{2\pi}\int_{C(\lambda_0)}(T_m-\lambda)^{-1}d\lambda,\]
where $C(\lambda_0)$ is a small circle centered at $\lambda_0$ such that $\lambda_0$ is the only eigenvalue lying in $C(\lambda_0)$.
\begin{definition}The generic extension $T_m$ is called complete if $\bigvee_{\lambda\in \sigma(T_m)}E_{\lambda}$ is dense in $H$. If it is the case, we say $m$ is a complete boundary condition.
\end{definition}
\begin{example}Let $T(b)$ be the entire operator in Example \ref{exp}. Since the contractive Weyl function $e^{\mathrm{i}b\lambda}$ never vanishes, the boundary condition $\varphi(b)=0$ has to be incomplete. It is no wander because the inverse of the corresponding extension $T_m$ is actually a Volterra operator.
 \end{example}

If $T_m$ is a self-adjoint operator, the spectral resolution of $T_m$ conveys that $T_m$ is complete\footnote{If $T$ is not densely defined and $T_m$ is the unique self-adjoint extension that is not an operator, the corresponding eigenvectors cannot span $H$. See \cite[Thm.~8]{silva2013branges}}. Since strictly dissipative extensions and strictly accumulative extensions of $T$ appear in pairs, it is easy to find that if a strictly dissipative extension is complete, so is the corresponding strictly accumulative extension. Therefore, w.l.g., we can only consider strictly accumulative extensions.

If $T_m$ is a fixed strictly accumulative extension of $T$, we can choose a boundary triplet $(\mathbb{C}, \Gamma_+, \Gamma_-)$ such that $T_\infty$ is just $T_m$ and $T_0$ the corresponding strictly dissipative extension. Our first goal in this section is to prove the following theorem.
\begin{theorem}\label{mean}$T_\infty$ is complete if and only if its mean type is zero.
\end{theorem}
Let $\lambda_0\in \sigma(T_\infty)\subset \mathbb{C}_-$ and $B(\lambda)$ the contractive Weyl function associated to $(\mathbb{C}, \Gamma_+, \Gamma_-)$. Our first step is to characterize $E_{\lambda_0}$ in the de Branges-Rovnyak model. By abuse of notation, we still use $T_\infty$ to denote the model operator of $T_\infty$. In the following proposition, elements in $\mathcal{H}_B$ are interpreted as analytic functions on $\mathbb{C}\backslash \sigma(T_\infty)$.
\begin{proposition}If $\mu\in \rho(T_\infty)$, then
\[D(T_\infty)=\{f\in \mathcal{H}_B|\exists h\in \mathcal{H}_B \quad s.t.\quad  f(\lambda)=\frac{h(\lambda)-h(\mu)}{\lambda-\mu},\quad \forall \lambda\in \mathbb{C}\backslash \sigma(T_\infty)\}\]
and for each $h\in \mathcal{H}_B$
\[[(T_\infty-\mu)^{-1}h](\lambda)=\frac{h(\lambda)-h(\mu)}{\lambda-\mu},\quad \forall \lambda\in \mathbb{C}\backslash \sigma(T_\infty).\]
The generalized eigensubspace of $T_\infty$ corresponding to $\lambda_0\in \sigma(T_\infty)$ has a basis
\[\{\frac{1}{(\lambda-\lambda_0)^j}\}_{j=1}^k.\]
\end{proposition}
\begin{proof}According to Prop.~5.24 in \cite{wang2024complex}, the graph $G_{T_\infty}$ of $T_\infty$ is
\[G_{T_\infty}=\{(f,g)\in \mathcal{H}_B\times \mathcal{H}_B|\exists c_-\in \mathbb{C}\ s.t. \ \forall \lambda\in \mathbb{C}\backslash \sigma(T_\infty),\  g(\lambda)=\lambda f(\lambda)+c_-\}.\]
This is because $T_\infty$ is characterized by the boundary condition $\Gamma_+ a=0$ for $a\in A_T^{\bot_s}$. Note that $D(T_\infty)=\textup{Ran}((T_\infty-\mu)^{-1})$. For each $h\in \mathcal{H}_B$, consider the equation $(T_\infty-\mu)f=h$, i.e.,
\[\lambda f(\lambda)+c_--\mu f(\lambda)=h(\lambda),\quad \forall \lambda\in \mathbb{C}\backslash \sigma(T_\infty).\]
Setting $\lambda=\mu$ in the above formula, we find $c_-=h(\mu)$ and thus $f(\lambda)=\frac{h(\lambda)-h(\mu)}{\lambda-\mu}$.

Now for any $h\in \mathcal{H}_B$ and $\lambda\in \mathbb{C}_+$,
\[(P_{\lambda_0}h)(\lambda)=\frac{\mathrm{i}}{2\pi}\int_{C(\lambda_0)}\frac{h(\lambda)-h(\mu)}{\lambda-\mu}d\mu.\]
Note that here the properties of Bochner integral and the continuity of the evaluation map at $\lambda$ are used. Since $\frac{1}{\lambda-\mu}$ is analytic in $C(\lambda_0)$,
\[(P_{\lambda_0}h)(\lambda)=\frac{1}{2\pi \mathrm{i}}\int_{C(\lambda_0)}\frac{h(\mu)}{\lambda-\mu}d\mu.\]
By using a de Branges function $e(\lambda)$ for $B(\lambda)$, we know $h(\mu)=\frac{\tilde{h}(\mu)}{e(\mu)}$ for an entire function $\tilde{h}(\mu)$. We can also write $e(\mu)=e_0(\mu)\times (\mu-\lambda_0)^k$ where $e_0(\lambda_0)\neq 0$. Thus
\[(P_{\lambda_0}h)(\lambda)=\frac{1}{2\pi \mathrm{i}}\int_{C(\lambda_0)}\frac{j(\mu)}{(\mu-\lambda_0)^k}d\mu=\frac{j^{(k-1)}(\lambda_0)}{(k-1)!},\]
where $j(\mu)=\frac{\tilde{h}(\mu)}{(\lambda-\mu)e_0(\mu)}$. However, by using the Leibnitz formula for higher derivatives of a product, it is easy to see $j^{(k-1)}(\lambda_0)$ is of the following form
\[\sum_{j=1}^k \frac{c_j(\lambda_0)}{(\lambda-\lambda_0)^j},\]
where the coefficients $c_j(\lambda_0)$ are independent of $\lambda$. This shows the generalized eigensubspace $E_{\lambda_0}$ is contained in the subspace $V_{\lambda_0}$ generated by $\{\frac{1}{(\lambda-\lambda_0)^j}\}_{j=1}^k$. Due to dimensional reason, $E_{\lambda_0}=V_{\lambda_0}$.
\end{proof}
\emph{Remark}. Along the same line, one can compute the generalized eigensubspace of $T_0$ associated to the eigenvalue $\overline{\lambda_0}$. It is generated by the basis
\[\{\frac{B(\lambda)}{(\lambda-\overline{\lambda_0})^j}\}_{j=1}^k.\]

The above proposition has the following essentially known byproduct which will not be used later.
\begin{corollary}For $T\in \mathfrak{E}_1(H)$, the canonical model space $\mathfrak{H}$ is division invariant, i.e., if $\lambda_0$ is a zero of $s(\lambda)\in \mathfrak{H}$, then $\frac{s(\lambda)}{\lambda-\lambda_0}$ also lies in $\mathfrak{H}$.
\end{corollary}
\begin{proof}Due to Prop.~\ref{par}, we can always choose a strictly accumulative boundary condition $m$ such that $\lambda_0\in \rho(T_m)$. Choose a suitable boundary triplet $(\mathbb{C}, \Gamma_+, \Gamma_-)$ such that $T_m=T_\infty$. Then the result follows from the above proposition by setting $\mu=\lambda_0$.
\end{proof}

The following lemma seems elementary, but we haven't found a suitable reference covering it.
\begin{lemma}\label{zeros}A function $f\in \mathcal{H}_B$ is orthogonal to the generalized eigensubspace $E_{\lambda_0}$ if and only if $\overline{\lambda_0}$ (not $\lambda_0$ itself) is a zero of $f$ with multiplicity at least $k$, i.e., $f^{(j)}(\overline{\lambda_0})=0$ for $j=0, 1,\cdots, k-1$.
\end{lemma}
\begin{proof}Recall that $\mathcal{H}_B$ is a closed subspace of the Hardy space $H^2(\mathbb{C}_+)$. To prove the lemma, we have to recall some facts on $H^2(\mathbb{C}_+)$. See \cite[Chap.~VI]{koosis1998introduction} and \cite[Chap.~8]{hoffman2014banach} for further technical details.

Each function $f\in H^2(\mathbb{C}_+)$ has a non-tangential boundary value $f^*(t)$ on the real line. $f^*$ is in $L^2(\mathbb{R})$ and for $f,g\in H^2(\mathbb{C}_+)$
\[(f,g)_{H^2(\mathbb{C}_+)}=\int_\mathbb{R}f^*(t)\overline{g^*(t)}\frac{dt}{2\pi}.\]
$f\in H^2(\mathbb{C})$ can also be recovered from its boundary value $f^*$, i.e., we have the Cauchy formula
\[f(\lambda)=\frac{1}{2\pi \mathrm{i}}\int_\mathbb{R}\frac{f^*(t)}{t-\lambda}dt,\quad \forall \lambda\in \mathbb{C}_+.\]
By taking derivatives under the integration and a routine use of Lebesgue's Dominated Convergence Theorem, we can deduce from this formula that
\[\frac{f^{(j)}(\lambda)}{j!}=\frac{1}{2\pi \mathrm{i}}\int_\mathbb{R}\frac{f^*(t)}{(t-\lambda)^{j+1}}dt,\quad j=1,2,\cdots.\]

Now let $e_j(\lambda):=\frac{1}{(\lambda-\lambda_0)^j}$ and $f\in \mathcal{H}_B\subseteq H^2(\mathbb{C}_+)$. Then
\[(f, e_j)_{\mathcal{H}_B}=\int_\mathbb{R}f^*(t)\overline{e_j^*(t)}\frac{dt}{2\pi}=\int_\mathbb{R}\frac{f^*(t)}{(t-\overline{\lambda_0})^j}\frac{dt}{2\pi}=\mathrm{i}\times \frac{f^{(j-1)}(\overline{\lambda_0})}{(j-1)!}.\]
The claim of the lemma then follows immediately.
\end{proof}
We can use the method of Toeplitz kernel to translate the completeness problem into the problem of determining the injectivity of a certain Toeplitz operator. We refer the reader to \cite{makarov2005meromorphic} for a detailed investigation of this method. Recall that we can also define the Hardy space $H^2(\mathbb{C}_-)=\overline{H^2(\mathbb{C}_+)}$ whose elements also have $L^2(\mathbb{R})$ boundary values. If elements in $H^2(\mathbb{C}_\pm)$ are identified with their boundary values, we should have $H^2(\mathbb{C}_+)\oplus_\bot H^2(\mathbb{C}_-)=L^2(\mathbb{R})$. Let $P$ be the orthogonal projector onto $H^2(\mathbb{C}_+)$. Then for a bounded measurable function $f$ on the real line, the Toeplitz operator $\mathrm{T}_f$ with symbol $f$ is defined to be
\[\mathrm{T}_f: H^2(\mathbb{C}_+)\rightarrow H^2(\mathbb{C}_+),\quad g\mapsto P(fg).\]

W.l.g, now we write the contractive Weyl function $B(\lambda)$ as $e^{\mathrm{i}b\lambda}\times\Delta(\lambda)$ where $\Delta(\lambda)$ is the Blaschke product part in the factorization (\ref{cf}). Then we have the function $\overline{B^*(t)}\Delta^*(t)=e^{-\mathrm{i}bt}$.
\begin{lemma}$T_\infty$ is complete if and only if $\textup{ker}(\mathrm{T}_{\overline{B^*}\Delta^*})=0$.
\end{lemma}
\begin{proof}If $T_\infty$ is not complete, then there is an $f\in \mathcal{H}_B$ such that $f$ is orthogonal to $\bigvee_{\lambda\in \sigma(T_\infty)}E_{\lambda}$. Due to Lemma \ref{zeros}, $\overline{\lambda_j}$ is a zero of $f$ with multiplicity at least $k_j$, where $k_j$ is the analytic multiplicity of $\lambda_j\in \sigma(T_\infty)$. This shows
\[f(\lambda)=\Delta(\lambda)\times g(\lambda)\]
for a nonzero function $g\in H^2(\mathbb{C}_+)$. That $g$ lies in $H^2(\mathbb{C}_+)$ is a result of the inner-outer factorization of $f$. On the other side, since $\mathcal{H}_B=(BH^2(\mathbb{C}_+))^\bot$, $f\in \mathcal{H}_B$ if and only if there is an $h\in H^2(\mathbb{C}_+)$ such that on $\mathbb{R}$
\[f^*(t)=B^*(t)\overline{h^*(t)}.\]
Therefore,
\[\overline{h^*(t)}=\overline{B^*(t)}f^*(t)=\overline{B^*(t)}\Delta^*(t)g^*(t),\]
implying that $g$ lies in $\textup{ker}(\mathrm{T}_{\overline{B^*}\Delta^*})$.

Conversely, if $0\neq g\in \textup{ker}(\mathrm{T}_{\overline{B^*}\Delta^*})$, then there is an $h\in H^2(\mathbb{C}_+)$ such that $\overline{h^*(t)}=\overline{B^*(t)}\Delta^*(t)g^*(t)$. Thus $B^*(t)\overline{h^*(t)}=\Delta^*(t)g^*(t)$. This means the nonzero function $f(\lambda)=\Delta(\lambda)g(\lambda)$ lies in $\mathcal{H}_B$. Since $\overline{\lambda_j}$ is a zero of $f(\lambda)$ with  multiplicity at least $k_j$. Due to Lemma \ref{zeros}, $f$ is orthogonal to $\bigvee_{\lambda\in \sigma(T_m)}E_{\lambda}$ and as a consequence $T_\infty$ is not complete.
\end{proof}
\begin{proof}Proof of Thm.~\ref{mean}. Now this is clear because $\textup{ker}(\mathrm{T}_{e^{-\mathrm{i}bt}})$ is precisely the de Branges-Rovnyak space $\mathcal{H}_B$ with $B(\lambda)=e^{\mathrm{i}b\lambda}$. See \cite[\S~3]{makarov2005meromorphic}.
\end{proof}
\begin{proposition}If for $T\in \mathfrak{E}_1(H)$ the Weyl order $\rho_T<1$, then each strictly dissipative or accumulative boundary condition is complete.
\end{proposition}
\begin{proof}This is simply because due to Prop.~\ref{order} the mean type of each strictly dissipative boundary condition is zero when $\rho_T<1$.
\end{proof}
Roughly speaking, the non-vanishing mean type $b$ prevents the extension $T_\infty$ from having sufficiently many eigenvalues. Taking the Crofton formula \cite[Lemma 10.24]{wang2024complex} on the average spectral behavior of boundary conditions into account, one may reasonably conjecture that \emph{almost all} boundary conditions are complete. Note that $\mathcal{B}_T$ is a (strong) symplectic Hilbert space and thus $\mathcal{M}$ has its induced Fubini-Study metric. Let $\omega$ be the corresponding normalized volume form and $\mathcal{C}\subset \mathcal{M}$ the subset of strictly dissipative boundary conditions with non-vanishing mean type.
\begin{theorem}$\mathcal{C}$ is Borel measurable and the $\omega$-measure of $\mathcal{C}$ is zero.
\end{theorem}
\begin{proof}Choose a boundary triplet $(\mathbb{C}, \Gamma_+, \Gamma_-)$. Then $\mathcal{M}$ is holomorphically isomorphic to $\mathbb{CP}^1$ and strictly dissipative boundary conditions are parameterized by the unit disc $\mathbb{D}$. For $c\in \mathbb{D}$, the mean type of $T_c$ is by definition given by
\[b(c)=-\limsup_{y\rightarrow +\infty}\frac{\ln|\frac{B(\mathrm{i}y)-c}{1-B(iy)\bar{c}}|}{y}=-\limsup_{y\rightarrow +\infty}\frac{\ln|B(\mathrm{i}y)-c|}{y}.\]
Note that the function $F(y,c):=\frac{\ln|B(\mathrm{i}y)-c|}{y}$ is a continuous function from $(0, +\infty)\times \mathbb{D}$ to $[-\infty, +\infty)$. $\mathcal{C}$ can be identified with the set
 \[\{c\in \mathbb{D}|b(c)>0\}\] and thus the Borel measurability of $\mathcal{C}$ can be guaranteed.

 Since \[\mathcal{C}=\bigcup_{n=1}^\infty \{c\in \mathbb{D}|b(c)>\frac{1}{n}\},\]
 it suffices to prove that each $\mathcal{C}_n=\{c\in \mathbb{D}|b(c)>\frac{1}{n}\}$ has measure zero. Note that the identification $\mathcal{M}\cong \mathbb{CP}^1$ may not preserve the Fubini-Study metrics, but we can identify the Weyl map $W_T(\lambda)$ with $B(\lambda)$ and define $h_T(r)$ only in terms of $\mathbb{CP}^1$ and $B(\lambda)$. This only differs from the \emph{canonical} definition by a bounded term. To avoid the possible singularity at $r=0$, let us define $\tilde{h}_T(r):=h_T(r)-h_T(1)$ and $\tilde{N}_T(r,c):=N_T(r,c)-N_T(1,c)$. Then according to \cite[Lemma 10.25]{wang2024complex}, there is a positive constant $K_1>0$ independent of $c$ and $r$ such that for $r>1$
 \[\tilde{N}_T(r,c)<\tilde{h}_T(r)+K_1.\]

  On the other side, the proof of \cite[Thm.~10.19]{wang2024complex} shows here
 \[h_T(r)=\frac{1}{2\pi \mathrm{i}}\int_0^r\frac{dt}{t}\int_{-t}^td\ln B(u)+e_1(r)\]
 where $|e_1(r)|<\ln 2$. Thus
 \[\tilde{h}_T(r)=\frac{1}{2\pi \mathrm{i}}\int_1^r\frac{dt}{t}\int_{-t}^td\ln B(u)+e_2(r)\]
 where $|e_2(r)|<2\ln 2$. Therefore
 \[\tilde{N}_T(r,c)<\frac{1}{2\pi \mathrm{i}}\int_1^r\frac{dt}{t}\int_{-t}^td\ln B(u)+2\ln 2+K_1.\]
 Now if $c\in \mathcal{C}_n$, then
 \[\tilde{B}(\lambda):=e^{-\mathrm{i}b(c)\lambda}\times\frac{B(\lambda)-c}{1-B(\lambda)\bar{c}}\]
  is also a meromorphic inner function whose zeros are just $\sigma(T_c)$. We can view $\tilde{B}(\lambda)$ as one contractive Weyl function of \emph{another} entire operator $\tilde{T}$. Then the above argument applied to $\tilde{T}$ implies there is a positive constant $K_2$ independent of $c$ and $r$ such that
  \begin{eqnarray*}\tilde{N}_T(r,c)&<&\frac{1}{2\pi \mathrm{i}}\int_1^r\frac{dt}{t}\int_{-t}^td\ln \tilde{B}(u)+K_2\\
  &=&
  \frac{1}{2\pi \mathrm{i}}\int_1^r\frac{dt}{t}\int_{-t}^td\ln (\frac{B(u)-c}{1-B(u)\bar{c}})-\frac{b(c)(r-1)}{\pi}+K_2.
    \end{eqnarray*}
 Due to \cite[Thm.~10.43]{wang2024complex}, we know that
 \[|\frac{1}{2\pi \mathrm{i}}\int_{-t}^t d\ln B(u)-n_T(r, 1)|\leq 1\]
 and
 \[|\frac{1}{2\pi \mathrm{i}}\int_{-t}^t d\ln(\frac{B(u)-c}{1-B(u)\bar{c}}) -n_T(r, \frac{1+c}{1+\bar{c}})|\leq 1.\]
 It is also a basic fact that
 \[|n_T(r,1)-n_T(r, \frac{1+c}{1+\bar{c}})|\leq 1.\]
 Combining all these together, for $c\in \mathcal{C}_n$ we have
 \[\tilde{N}_T(r,c)<\tilde{h}_T(r)-\frac{b(c)(r-1)}{\pi}+3\ln r+K_3<\tilde{h}_T(r)-\frac{r-1}{n\pi}+3\ln r+K_3,\]
 where $K_3=K_2+2\ln 2$.

 We know the (refined) Crofton formula \cite[Lemma 10.24]{wang2024complex}
 \[\tilde{h}_T(r)=\int_{\mathbb{CP}^1}\tilde{N}_T(r,c)\omega.\]
 Since strictly dissipative boundary conditions and strictly accumulative ones appear in pairs and $\omega$ is $\mathbb{U}(2)$-invariant, we should have $\tilde{h}_T(r)=2\int_{\mathbb{D}}\tilde{N}_T(r,c)\omega$. Therefore, we have
 \begin{eqnarray*}\tilde{h}_T(r)&=&2\int_{\mathbb{D}\backslash \mathcal{C}_n}\tilde{N}_T(r,c)\omega+2\int_{ \mathcal{C}_n}\tilde{N}_T(r,c)\omega\\
 &\leq&2\int_{\mathbb{D}\backslash \mathcal{C}_n}(\tilde{h}_T(r)+K_1)\omega+2\int_{ \mathcal{C}_n}(\tilde{h}_T(r)-\frac{r-1}{n\pi}+3\ln r+K_3)\omega\\
 &\leq&\tilde{h}_T(r)+K_1+2\times(-\frac{r-1}{n\pi}+3\ln r+K_3)|\mathcal{C}_n|,
  \end{eqnarray*}
 where $|\mathcal{C}_n|$ is the $\omega$-measure of $\mathcal{C}_n$. Clearly, $|\mathcal{C}_n|$ should be zero, otherwise there is a contradiction when $r$ is sufficiently large.
\end{proof}

Due to the above theorem and Prop.~\ref{mul}, we see that for $T\in \mathfrak{E}_1(T)$ almost all strictly accumulative boundary conditions are complete and only have simple eigenvalues. For later simplicity, we assume this is the case for $T_\infty$. Now let $f_j(\lambda):=\frac{1}{\lambda-\lambda_j}$ for $\lambda_j\in \sigma(T_\infty)$. We find that
\[(f_j, f_l)_{\mathcal{H}_B}=\int_{\mathbb{R}}\frac{1}{(u-\lambda_j)(u-\overline{\lambda_l})}\frac{du}{2\pi}=\frac{\mathrm{i}}{\overline{\lambda_l}-\lambda_j}.\]
In particular, $\|f_j\|_{\mathcal{H}_B}=\frac{1}{\sqrt{-2\Im \lambda_j}}$. We have the normalized vectors $\mathfrak{f}_j(\lambda):=\frac{\sqrt{-2\Im \lambda_j}}{\lambda-\lambda_j}$. By a similar argument to the proof of Thm.~\ref{mean}, one can easily see that $\{\mathfrak{f}_j\}$ is \emph{minimal} in the sense that if one element is deleted, the residual vectors won't span $\mathcal{H}_B$ any more.
Recall that a Riesz basis of a Hilbert space is the image of an orthonormal basis under a bounded invertible operator.
\begin{definition}We say $T_\infty$ is Riesz complete, if the above $\{\mathfrak{f}_j(\lambda)\}_{j=1}^\infty$ is a Riesz basis of $\mathcal{H}_B$.
\end{definition}
\emph{Remark}. A Riesz basis is something close to an orthonormal basis and useful in the context of non-self-adjoint operators. For more details on this notion, we refer the reader to \cite[Chap.~12]{garcia2016introduction}. Obviously, Riesz completeness is only a property of $T_\infty$ and independent of the chosen boundary triplet.  Recall that a biorthogonal sequence associated to a sequence $\{x_j\}_{j=1}^\infty$ in a Hilbert space $\mathcal{H}$ is another sequence $\{\widetilde{x_j}\}_{j=1}^\infty$ in $\mathcal{H}$ such that $(\widetilde{x_k}, x_j)_{\mathcal{H}}=\delta_{jk}$, where $\delta_{jk}$ is the Kronecker delta. If $\{x_j\}_{j=1}^\infty$ is complete in $\mathcal{H}$, then its biorthogonal sequence is uniquely determined.
\begin{lemma}Let $B_l(\lambda)=\frac{\lambda_l}{\overline{\lambda_l}}\prod_{j\neq l}\frac{\lambda_j}{\overline{\lambda_j}}\frac{\lambda-\overline{\lambda_j}}{\lambda-\lambda_j}$. Then for the sequence
$\{\mathfrak{f}(\lambda)\}_{j=1}^\infty$, the corresponding biorthogonal sequence is given by
\[\widetilde{\mathfrak{f}}_j(\lambda)=\frac{\sqrt{-2\Im \lambda_j}}{B_j(\overline{\lambda_j})}\times \frac{B(\lambda)}{\lambda-\overline{\lambda_j}}.\]
In particular,
\[\|\widetilde{\mathfrak{f}}_j\|_{\mathcal{H}_B}=\frac{1}{|B_j(\overline{\lambda_j})|}=\prod_{k\neq j}|\frac{\overline{\lambda_j}-\overline{\lambda_k}}{\overline{\lambda_j}-\lambda_k}|.\]

\end{lemma}
\begin{proof}We see that
\begin{eqnarray*}
(\frac{B(\cdot)}{\cdot-\overline{\lambda_k}}, \mathfrak{f}_j(\cdot))_{\mathcal{H}_B}&=&\sqrt{-2\Im \lambda_j}\times\int_{\mathbb{R}}\frac{B(u)}{(u-\overline{\lambda_k})(u-\overline{\lambda_j})}\frac{du}{2\pi}\\
&=&\sqrt{-2\Im \lambda_j}\times \int_{\mathbb{R}}\frac{B_k(u)}{(u-\lambda_k)(u-\overline{\lambda_j})}\frac{du}{2\pi}\\
&=&\mathrm{i} \sqrt{-2\Im \lambda_j}\frac{B_k(\overline{\lambda_j})}{(\overline{\lambda_j}-\lambda_k)}\\
&=&\frac{B_j(\overline{\lambda_j})}{\sqrt{-2\Im \lambda_j}}\delta_{jk}.
\end{eqnarray*}
On the other side,
\[(\frac{B(\cdot)}{\cdot-\overline{\lambda_j}},\frac{B(\cdot)}{\cdot-\overline{\lambda_j}})_{\mathcal{H}_B}=\int_\mathbb{R}\frac{B(u)}{u-\overline{\lambda_j}}\cdot \frac{\overline{B(u)}}{u-\lambda_j}\frac{du}{2\pi}=\frac{1}{-2\Im\lambda_j}.\]
The lemma then follows immediately.
\end{proof}
\begin{theorem}\label{Riesz}$T_\infty$ is Riesz complete if and only if $\inf_{j\geq1}|B_j(\overline{\lambda_j})|>0$.
\begin{proof}We only cite relevant results in \cite[Chap.~12]{garcia2016introduction} and give an outline of an indirect proof. Given a Blaschke sequence $\{z_n\}$ in the unit disc $\mathbb{D}$, let $\mathrm{B}(z)$ be the associated Blaschke product, i.e.,
 \[\mathrm{B}(z)=\prod_{j=1}^\infty\frac{\overline{z_j}}{|z_j|}\frac{z_j-z}{1-\overline{z_j}z}.\]
$\mathrm{B}(z)$ is an inner function and thus generates a model space $H_\mathrm{B}=(\mathrm{B}H^2(\mathbb{D}))^\bot$, where $H^2(\mathbb{D})$ is the Hardy space on $\mathbb{D}$. \cite[Thm.~12.22]{garcia2016introduction} deals with the following question: \emph{when does the associated sequence of normalized reproducing kernels $k_{z_j}(z)$ form a Riesz basis in $H_B$?} This is completely parallel to the problem we are facing. More importantly, these two contexts can be connected by the following unitary map from $H^2(\mathbb{D})$ to $H^2(\mathbb{C}_+)$:
\[\mathcal{U}:H^2(\mathbb{D})\rightarrow H^2(\mathbb{C}_+),\quad g(z)\mapsto \frac{\sqrt{2}}{\lambda+\mathrm{i}}\times g(\frac{\lambda-\mathrm{i}}{\lambda+\mathrm{i}}).\]
Via the Cayley transform our $B(\lambda)$ on $\mathbb{C}_+$ is turned into a Blaschke product $\mathrm{B}(z)$ on $\mathbb{D}$; via $\mathcal{U}$ our $\{\mathfrak{f}_j\}$ can be transformed into a sequence of normalized reproducing kernels in $H_\mathrm{B}$. Clearly, the notions of Riesz basis and \emph{constant of uniform minimality} (\cite[\S~12.2]{garcia2016introduction}) are both unitarily invariant. Then Thm.~12.9 and Thm.~12.22 of \cite{garcia2016introduction} lead to our result.
\end{proof}
\end{theorem}
\emph{Remark}. The theorem means the distribution of $\{\lambda_j\}$ must be sufficiently sparse for $T_\infty$ to be Riesz complete. The same condition also appears in the study of interpolating sequences in the upper half-plane $\mathbb{C}_+$. By \cite[Thm.~1.1, Chap.~VII]{garnett2006bounded}, $T_\infty$ is Riesz complete, if and only if $\{\lambda_i\}$ is an interpolating sequence. Clearly, a necessary condition for $T_\infty$ to be Riesz complete is that there should be a positive constant $\delta$ such that $|\frac{\overline{\lambda_j}-\overline{\lambda_k}}{\overline{\lambda_j}-\lambda_k}|>\delta$ as soon as $k\neq j$.

The following two examples show that both Riesz complete and Riesz incomplete boundary conditions can occur.
\begin{example}Let $\lambda_j=-\mathrm{i}j^2$. Note that
\[|\frac{\overline{\lambda_{j+1}}-\overline{\lambda_j}}{\overline{\lambda_{j+1}}-\lambda_j}|=\frac{2j+1}{2j^2+2j+1}\rightarrow 0, \quad (j\rightarrow +\infty).\]
Then $T_\infty$ associated to the corresponding Blaschke product $B(\lambda)$ cannot be Riesz complete. It is easy to see the Weyl order $\rho_T=1/2$.
\end{example}
\begin{example}Let $\lambda_j=j-\mathrm{i}$. Then according to the discussion after \cite[Thm.~1.1, Chap.~VII]{garnett2006bounded}, $\{j+\mathrm{i}\}$ is an interpolating sequence and $T_\infty$ associated to the corresponding Blaschke product $B(\lambda)$ must be Riesz complete.
\end{example}
Note that in the above example $\rho_T=1$. More generally, we have
\begin{proposition}Let $T_\infty$ be Riesz complete. If $\{\lambda_j\}_{j=1}^\infty$ are the eigenvalues of $T_\infty$ and $\inf_j \Im \overline{\lambda_j}>0$, then $\rho_T\leq 1$.
\end{proposition}
\begin{proof}According to \cite[Thm.~1.1, Chap.~VII]{garnett2006bounded}, $\sum_j \Im{\overline{\lambda_j}}\times\delta_{\lambda_j}$ has to be a Carleson measure on $\mathbb{C}_+$. This implies that there is a positive constant $A$ such that
\[\sum_{\overline{\lambda_j}\in Q(r)}\Im{\overline{\lambda_j}}\leq 2Ar,\]
where $Q(r)$ is the square $\{\lambda\in \mathbb{C}_+| \Re \lambda\in [-r, r],\quad \Im \lambda\in (0, 2r]\}$. If $N$ is the number of eigenvalues of $T_\infty$ lying in $Q(r)$, then the above inequality leads to
 \[N\leq  \frac{2Ar}{\inf_j \Im \overline{\lambda_j}}, \]
which surely implies the claim.
\end{proof}
It is interesting to know whether the existence of a Riesz complete non-self-adjoint extension of $T$ restricts the Weyl order of $T$.
\section{The growth aspect of the model space}\label{growth}
 Fix $T\in \mathfrak{E}_1(H)$ and let $\mathfrak{H}$ be its canonical model space in \S~\ref{model}. It is natural to ask how elements in $\mathfrak{H}$ are controlled by the behavior of the Weyl curve $W_T(\lambda)$. If the de Branges model in \S\S~\ref{deBranges} is in use, $\mathcal{H}_e$ is a space of entire functions and thus the maximum modulus can be used to measure the growth of an element in $\mathcal{H}_e$. This was the viewpoint adopted in \cite{kaltenbackbranges}. For us, since $\mathcal{H}_e$ is not an intrinsic object and depends on how the characteristic line bundle $F$ is trivialized, the maximum modulus is not that suitable for our purpose. However, zeros of a section in $\mathfrak{H}$ are of intrinsic meaning, and our concern in this section is how the distribution of zeros of an element in $\mathfrak{H}$ is controlled by the geometry of $F$.
 \begin{definition}Let $c_1(F)$ be the first Chern form of $F$, i.e., $c_1(F)=\frac{\mathrm{i}}{2\pi}R$ where $R$ is the curvature of the canonical Chern connection in $F$. Define the characteristic function $\mathrm{T}_F(r)$ to be
 \[\mathrm{T}_F(r):=\int_0^r\frac{dt}{t}\int_{\mathbb{D}_t}c_1(F).\]
 Here $\mathbb{D}_t$ is the open disc on $\mathbb{C}$, with center 0 and radius $t>0$.
 \end{definition}
 \emph{Remark}. $\mathrm{T}_F(r)$ only depends on the geometry of $F$ and is a unitary invariant of $T$. Since $R$ is positive-definite, $\mathrm{T}_F(r)$ is a strictly increasing function in $r$. According to the philosophy in \cite{cornalba1975analytic}, $\mathrm{T}_F(r)$ measures the growth of the line bundle $F$.

 Let $s$ be a nontrivial holomorphic section of $F$ and $n_s(r)$ be the number of zeros (counted with multiplicity) of $s$ in the open disc $\mathbb{D}_r$. Then two functions can be defined.
 \begin{definition}The zero counting function $N_s(r)$ for $r>0$ is defined to be
 \[N_s(r):=\int_0^r\frac{n_s(t)-n_s(0)}{t}dt+n_s(0)\ln r,\]
 where $n_s(0)$ is the multiplicity of 0 as a zero of $s$.  The proximity function $m_s(r)$ is defined to be
 \[m_s(r):=-\frac{1}{2\pi}\int_0^{2\pi}\ln |s(re^{\mathrm{i}\phi})|d\phi,\]
 where $|s(\lambda)|$ is the length of $s(\lambda)$ in $F_\lambda$.
 \end{definition}
 \begin{proposition}\label{first}With the above definitions, we have
 \[\mathrm{T}_F(r)=m_s(r)+N_s(r)+O(1).\]
 Here $O(1)$ is a bounded term in $r$ as $r$ goes to $+\infty$.
 \end{proposition}
 \begin{proof}Recall the famous Poincare-Lelong formula
 \[dd^c[\ln |s|^2]=-c_1(F)+[(s)].\]
 Here $d^c=\frac{\mathrm{i}}{4\pi}(\overline{\partial}-\partial)$ and $(s)$ is the divisor on $\mathbb{C}$ determined by the zeros of $s$. This formula should be interpreted as an identity of currents. If $s(0)\neq 0$, by the Green-Jensen's formula (\cite[Thm.~A2.2.3]{ru2021nevanlinna}) we have
 \[\int_0^r\frac{dt}{t}\int_{\mathbb{D}_t}dd^c[\ln |s|^2]=\frac{1}{2\pi}\int_0^{2\pi}\ln|s(re^{\mathrm{i}\phi})|d\phi-\ln|s(0)|.\]
 The claim then follows. If 0 is a zero of $s(\lambda)$ with multiplicity $k$, it suffices to replace $s(\lambda)$ with $s(\lambda)/\lambda^k$ in the above argument.
 \end{proof}
In principle, by choosing a global holomorphic frame, the space of holomorphic sections of $F$ can be identified with the space of entire functions. Thus a general holomorphic section $s$ can be very wild and $N_s(r)$ won't be controlled by $\mathrm{T}_F(r)$ at all. However, this could not happen for $s\in \mathfrak{H}$. W.l.g., we assume that $\|s\|=1$, i.e., it is a unit vector in $\mathfrak{H}$.
\begin{lemma}\label{lem}If $s\in \mathfrak{H}$ and $\|s\|=1$, then $|s(\lambda)|\leq 1$ pointwise.
\end{lemma}
\begin{proof}Let $0\neq\omega\in F^\dag_\lambda$. Then by definition
\[|\langle s(\lambda), \omega\rangle|=|(s(\cdot), K(\cdot,\lambda)\omega)_{\mathfrak{H}}|\leq \|s\|\cdot \|K(\cdot, \lambda)\omega\|=\|K(\cdot, \lambda)\omega\|.\]
Note that
\[\|K(\cdot, \lambda)\omega\|^2=(K(\cdot, \lambda)\omega,K(\cdot, \lambda)\omega)_\mathfrak{H}=\langle K(\lambda,\lambda)\omega, \omega\rangle=(\iota_\lambda \omega, \iota_\lambda \omega)_H=|\omega|_\lambda^2.
\]
The claim then follows.
\end{proof}
We thus have
\begin{proposition}\label{esti}For a nonzero $s\in \mathfrak{H}$,
$N_s(r)\leq \mathrm{T}_F(r)+O(1)$.
\end{proposition}
\begin{proof}This is simply because $s$ and $\frac{s}{\|s\|}$ have the same zeros and $-\ln |\frac{s(\lambda)}{\|s\|}|\geq 0$ due to the above lemma.
\end{proof}

On one side, since $F$ is completely determined by $T$, $\mathrm{T}_F(r)$ surely measures the growth of $T$. On the other side, we also have the height function $h_T(r)$ to measure the growth of $T$, which can even be defined for entire operators with higher deficiency index. The following theorem implies that these two functions are essentially the same.
\begin{theorem}\label{equal}Let $F$ be the characteristic line bundle of $T\in \mathfrak{E}_1(H)$. Then
\[\mathrm{T}_F(r)=h_T(r)+O(\ln r).\]
\end{theorem}
\begin{proof}We note that $R=\omega(\lambda)d\lambda\wedge d\bar{\lambda}$ and $\omega(\bar{\lambda})=\omega(\lambda)$. Thus
\[\int_{\mathbb{D}_t}c_1(F)=2\int_{\mathbb{D}_t^+}c_1(F),\]
where $\mathbb{D}_t^+$ is the upper half of $\mathbb{D}_t$. In terms of a contractive Weyl function $B(\lambda)$ (see Lemma~\ref{metric}),
\[c_1(F)=dd^c\ln \frac{1-|B(\lambda)|^2}{2\Im \lambda}.\]
It should be mentioned that $\chi(\lambda)=\frac{1-|B(\lambda)|^2}{2\Im \lambda}$ is well-defined around the real line $\mathbb{R}$. Then by the Stokes Theorem, we have
\[\int_{\mathbb{D}_t^+}c_1(F)=\int_{\partial \mathbb{D}_t^+}d^c\ln \chi(\lambda)=\int_{C_t^+}d^c\ln \chi(\lambda)+\int_{I_t}d^c\ln \chi(\lambda),\]
where $C_t^+$ is the upper half of the circle $\partial \mathbb{D}_t$ and $I_t$ is the interval $[-t,t]$ on the real line. Note that in terms of polar coordinates,
\[d^c=\frac{1}{4\pi}(\rho d\phi\otimes\frac{\partial }{\partial \rho}-\frac{1}{\rho}d\rho\otimes\frac{\partial }{\partial \phi}).\]
Consequently,
\[\int_{C_t^+}d^c\ln \chi(\lambda)=\frac{t}{4\pi}\int_{0}^{\pi}\frac{\partial}{\partial t}\ln \chi(te^{\mathrm{i}\phi})d\phi\]
and hence
\[\int_0^r\frac{dt}{t}\int_{C_t^+}d^c\ln \chi(\lambda)=\frac{1}{4\pi}\int_0^rd\int_0^{\pi}\ln \chi(te^{\mathrm{i}\phi})d\phi=\frac{1}{4\pi}\int_0^{\pi}\ln \chi(re^{\mathrm{i}\phi})d\phi-\frac{1}{4}\ln \chi(0).\]
Since $\chi(\lambda)\leq \frac{1}{2\Im \lambda}$ on $\mathbb{C}_+$, we have
\[\int_0^{\pi}\ln \chi(re^{\mathrm{i}\phi})d\phi\leq -\int_0^{\pi}(\ln 2+\ln r+\ln \sin \phi)d\phi.\]
It's easy to see the integral $-\int_0^{\pi}\ln \sin \phi d\phi$ is convergent. We denote its value by $k_0$. Therefore,
\[\int_0^{\pi}\ln \chi(re^{\mathrm{i}\phi})d\phi\leq-\pi \ln 2-\pi \ln r+k_0.\]
Let $\lambda=u+\mathrm{i}v$ where $u,v\in \mathbb{R}$. Then in this rectangular coordinate system, we have
\[d^c=\frac{1}{4\pi}(dv\otimes \frac{\partial}{\partial u}-du\otimes \frac{\partial}{\partial v}).\]
Thus
\[\int_{I_t}d^c\ln \chi(\lambda)=-\frac{1}{4\pi}\int_{-t}^t(\frac{\partial \ln \chi}{\partial v})|_{v=0}du.\]
For $v>0$, we can find easily that
\[-\frac{\partial \ln \chi}{\partial v}=\frac{1}{v}-\frac{2\Im(B'(\lambda)\overline{B(\lambda)})}{1-|B(\lambda)|^2}.\]
In terms of notations in Prop.~\ref{cur}, we see that as $\varepsilon\rightarrow 0+$
\[1-|B(u+\mathrm{i}\varepsilon)|^2=2c_1\varepsilon-2c_1^2\varepsilon^2+o(\varepsilon^2)\]
and
\[2\Im(B'(u+\mathrm{i}\varepsilon)\overline{B(u+\mathrm{i}\varepsilon)})=2(c_1-2c_1^2\varepsilon+o(\varepsilon)).\]
From these we can deduce that
\[-(\frac{\partial \ln \chi}{\partial v})|_{v=0}=c_1=\theta'(u)\]
where $\theta(u)$ is the phase function for $B(\lambda)$.
Therefore,
\[\int_0^r\frac{dt}{t}\int_{I_t}d^c\ln \chi(\lambda)=\frac{1}{4\pi}\int_0^r\frac{dt}{t}\int_{-t}^td\theta=\frac{1}{4\pi i}\int_0^r\frac{dt}{t}\int_{-t}^td\ln B(u)=\frac{1}{2}h_T(r)+O(1).\]
Note that the last equality is due to the formula (\ref{height}).
From these we see that
\begin{equation}\mathrm{T}_F(r)\leq h_T(r)-\ln r/2+O(1).\label{RI}\end{equation}

On the other side, due to Prop.~\ref{lem1} for $\phi\in [0,\pi]$ we have $\chi(re^{\mathrm{i}\phi})\geq \chi(r\cos \phi+\mathrm{i}r)$. Let $B(\lambda)$ be of the form (\ref{cf}). Note that on $\mathbb{C}_+$ the modulus of each Blaschke factor is less than 1.

\emph{Case} 1. If the mean type $b>0$, then $|B(r\cos \phi+\mathrm{i}r)|\leq e^{-br}$ and
\[\chi(r\cos \phi+\mathrm{i}r)\geq \frac{1-e^{-2br}}{2r}.\]
Therefore,
\[\int_0^{\pi}\ln\chi(re^{\mathrm{i}\phi})d\phi\geq \int_0^{\pi}\ln\chi(r\cos\phi+\mathrm{i}r)d\phi\geq \pi \ln (1-e^{-2br})-\pi\ln 2-\pi\ln r.\]
Obviously, this implies that for $r$ large enough we have
\[\mathrm{T}_F(r)\geq h_T(r)-\ln r/2+O(1).\]
Combining this with (\ref{RI}), we thus have $\mathrm{T}_F(r)=h_T(r)-\ln r/2+O(1)$.

\emph{Case} 2. If the mean type $b=0$, we can single out a Blaschke factor, say $\frac{\overline{\lambda_1}}{\lambda_1}\frac{\lambda-\lambda_1}{\lambda-\overline{\lambda_1}}$ and have
\[|B(\lambda)|< |\frac{\lambda-\lambda_1}{\lambda-\overline{\lambda_1}}|.\]
Let $\lambda_1=u_0+iv_0$ with $u_0,v_0\in \mathbb{R}$. Then
\[|B(r\cos\phi+\mathrm{i}r)|^2< \frac{(r\cos\phi-u_0)^2+(r-v_0)^2}{(r\cos \phi-u_0)^2+(r+v_0)^2}.\]
From this we can obtain
\begin{eqnarray*}\chi(re^{\mathrm{i}\phi})&\geq& \chi(r\cos \phi+\mathrm{i}r)=\frac{1-|B(r\cos \phi)+\mathrm{i}r|^2}{2r}\\
&>&\frac{2v_0}{(r\cos \phi-u_0)^2+(r+v_0)^2}.
\end{eqnarray*}
W.l.g., we can assume $u_0>0$. Then for $r$ large enough,
\[\int_0^{\pi}\ln \chi(re^{\mathrm{i}\phi})d\phi\geq -2\pi\ln r+O(1).\]
Thus in this case
\[-\ln r+O(1)\leq \mathrm{T}_F(r)-h_T(r)\leq -\ln r/2+O(1).\]
The claimed result thus holds again.
\end{proof}
\emph{Remark}. (1) Obviously, more is proved in the above proof: For $r$ large enough, we must have $\mathrm{T}_F(r)< h_T(r)$, but the difference is of order $O(\ln r)$. (2) Generally speaking, the estimation of $N_s(r)$ in Prop.~\ref{esti} is almost optimal: If $\tilde{T}$ is a self-adjoint extension of $T$ that is an operator and $\{\lambda_i\}$ are eigenvalues of $\tilde{T}$, then in the de Branges model, $\{\mathrm{K}_1(\cdot, \lambda_i)\}$ are the corresponding eigenfunctions in $\mathcal{H}_e$. Due to the reproducing property, $\mathrm{K}_1(\cdot, \lambda_i)$ vanishes at $\lambda_j$ whenever $j\neq i$. Then by Eq.~(\ref{se}) and the above theorem, $N_s(r)=\mathrm{T}_F(r)+O(\ln r)$ for $s=\mathrm{K}_1(\cdot, \lambda_i)$.

Note that if $s\in \mathfrak{H}$ has no zeros, then Prop.~\ref{first} implies
\[\mathrm{T}_F(r)=-\frac{1}{2\pi}\int_0^{2\pi}\ln |s(re^{\mathrm{i}\phi})|d\phi+O(1).\]
Thus it's possible to estimate the growth of $T$ from $s$. In this respect, we have
\begin{proposition}\label{0-entire}If in the model space $\mathfrak{H}$ for $T\in \mathfrak{E}_1(H)$ there is a section $s$ of $F$ vanishing nowhere on $\mathbb{C}$, then the Weyl order $\rho_T\leq 1$.
\end{proposition}
\begin{proof}Let us use the de Branges model for $T$. Then the assumption means there is an $f\in \mathcal{H}_e$ vanishing nowhere on $\mathbb{C}$. We can further require $f$ to be real, i.e., $f=f^\sharp$. See Remark 2.8 in \cite{silva2012class}. Then by replacing the frame $\xi(\lambda)$ with the new frame $f(\lambda)\xi(\lambda)$, we can set $f(\lambda)\equiv 1$. In this case, the de Branges function $e(\lambda)$ has to be of bounded type on $\mathbb{C}_+$. Since $|e^\sharp(\lambda)|< |e(\lambda)|$ on $\mathbb{C}_+$, $e^\sharp(\lambda)$ also has to be of bounded type on $\mathbb{C}_+$. A well-known theorem of M. G. Krein says that for an entire function $f$, the following two conditions are equivalent:
  (i) $f$ is of exponential type and
\[\int_\mathbb{R}\frac{\ln^+|f(u)|}{1+u^2}du<+\infty;\]
(ii) both $f$ and $f^\sharp$ are of bounded type on $\mathbb{C}_+$. See \cite[Lecture 16]{levin1996lectures} or \cite[Appendix 5]{rosenblum1997hardy} for more details on this theorem.
Thus in our situation $e(\lambda)$ is at most of finite exponential type  i.e., there is a constant $K>0$ such that
\[\ln|e(\lambda)|\leq K |\lambda|,\quad \forall \lambda\in \mathbb{C}.\]

Note that in the present setting, $N_s(r)\equiv 0$ and
\[-\ln |s(\lambda)|=\frac{1}{2}\ln (\frac{|e(\lambda)|^2-|e^\sharp(\lambda)|^2}{2\Im \lambda}).\]
Obviously $|s(\lambda)|=|s(\bar{\lambda})|$ and for $\lambda\in \mathbb{C}_+$ we have
\[\ln (\frac{|e(\lambda)|^2-|e^\sharp(\lambda)|^2}{2\Im \lambda})=2\ln|e(\lambda)|+\ln \frac{1-|B(\lambda)|^2}{2\Im \lambda}=2\ln|e(\lambda)|+\ln \chi(\lambda).\]
Additionally, in the proof of Thm.~\ref{equal} we have already proved that \[\int_0^{\pi}\ln \chi(re^{\mathrm{i}\phi})d\phi\leq -\pi \ln 2-\pi \ln r+k_0.\]
Combining all these observations together, we obtain
\[-\frac{1}{2\pi}\int_0^{2\pi}\ln |s(re^{\mathrm{i}\phi})|d\phi\leq Kr-\ln r/2+C_0\]
for $r>0$ and a certain real constant $C_0$. Then Prop.~\ref{first} implies that $\mathrm{T}_F(r)\leq Kr-\ln r/2+O(1)$ and the claimed result follows from Thm.~\ref{equal} and the definition of $\rho_T$.
\end{proof}

Note that $T(b)$ in Example \ref{exp} has Weyl order 1, but a direct conclusion of \cite[Example 3.8]{silva2012class} is that there is no section $s$ in the canonical model space $\mathfrak{H}$ vanishing nowhere on $\mathbb{C}$. However, we do have a partial converse of the above proposition.
\begin{proposition}If $T\in \mathfrak{E}_1(H)$ has Weyl order $\rho_T<1$ and there is a strictly accumulative extension $\tilde{T}$ such that $\sum_{j=1}^\infty(\frac{\Re \lambda_j}{\Im \lambda_j})^2<+\infty$ where $\{\lambda_j\}$ are eigenvalues of $\tilde{T}$ counted with multiplicity, then there is an $s\in \mathfrak{H}$ such that $s$ vanishes nowhere on $\mathbb{C}$.
\end{proposition}
\begin{proof}Since $\rho_T< 1$, in the de Branges model of $T$ we can choose
\[e(\lambda)=\prod_{j=1}^\infty(1-\frac{\lambda}{\lambda_j}).\]
The order $\rho$ of $e(\lambda)$ as an entire function coincides with the exponent of convergence of the sequence $\{\lambda_i\}$ \cite[Prop.~9.10.6]{simon2015basic}. Due to \cite[Prop.~10.14]{wang2024complex} we have $\rho\leq \rho_T$. On the other side, $\rho_T$ is actually the order of the contractive Weyl function $e^\sharp/e$, which is not greater than $\rho$ ($e$ and $e^\sharp$ have the same order $\rho$). Thus $\rho=\rho_T$ and as a consequence we can choose $\varepsilon>0$ such that $\rho_T+\varepsilon<1$ and
$\ln|e(\lambda)|< |\lambda|^{\rho_T+\varepsilon}$ for $|\lambda|$ large enough. This implies particularly that $\int_\mathbb{R}\frac{\ln^+|e(u)|}{1+u^2}<+\infty$. Again by the abovementioned Krein's Theorem, we see both $e$ and $e^\sharp$ are of bounded type on $\mathbb{C}_+$ and so are $1/e$ and $1/e^\sharp$. It is easy to see both $1/e$ and $1/e^\sharp$ have vanishing mean type. We can also prove $1/e$ is in $L^2(\mathbb{R})$ because
\[\frac{1}{|e(u)|^2}\leq \frac{|\lambda_1|^2}{|u-\lambda_1|^2}\times \prod_{j=2}^\infty(1+(\frac{\Re \lambda_j}{\Im \lambda_j})^2),\]
where the convergence of the infinite product is guaranteed by assumption. Combining all these together leads to the conclusion that the constant function $1\in \mathcal{H}_e$.
\end{proof}
\emph{Remark}. By \cite[Thm.~2.1]{kaltenback2005hermite}, not only the non-vanishing section $s$ exists, but also all polynomial multiples of $s$ lie in $\mathfrak{H}$ and even such sections are dense there.
\begin{example}We can take $\lambda_j=-\mathrm{i}j^{1+\delta}$ where $\delta>0$ is a constant. Then the corresponding $e$ has order $1/(1+\delta)$ and consequently $T\in \mathfrak{E}_1(H)$ with $e^\sharp/e$ as its contractive Weyl function ought to have a section $s\in \mathfrak{H}$ that is zero-free.
\end{example}
The following example shows that the single condition $\rho_T<1$ cannot guarantee the existence of a zero-free section. It is taken from \S\S~3.5 of \cite{havin2003admissible}.
\begin{example}Consider $\lambda_j=\sqrt{2}2^j e^{-\mathrm{i}\pi/4}$ with multiplicity $\lfloor a^j\rfloor$, where $a<2$ is a constant. Then for $a$ close to $2$ we can set
\[e(\lambda)=\prod_{j=1}^\infty(1-\frac{\lambda}{\lambda_j})^{\lfloor a^j\rfloor}.\]
This does make sense for one can easily check the exponent of convergence of the zeros is less than 1. However, it was checked in \cite{havin2003admissible} that $1/e\notin L^2(\mathbb{R})$. As a consequence, $T\in \mathfrak{E}_1(H)$ with $e^\sharp/e$ as its contractive Weyl function cannot have any zero-free section in $\mathfrak{H}$.
\end{example}

In the de Branges model, though the de Branges function $e$ is not canonically determined by $T$, the growth of elements in $\mathcal{H}_e$ is really controlled by the growth of $e$. This was one of the main topics in \cite{kaltenbackbranges} and maximal modulus was used to explore this \cite[Thm.~3.4]{kaltenbackbranges}. We feel the Nevanlinna characteristic function is more suitable here. Recall that the Nevanlinna characteristic function of an entire function $f(\lambda)$ is defined to be
\[T_f(r)=\frac{1}{2\pi}\int_0^{2\pi}\ln^+|f(re^{\mathrm{i}\phi})|d\phi,\]
where for $x>0$, $\ln^+ x=\max\{0,\ln x\}$. The order and type of $f$ can be defined in a similar way as we define the Weyl order and type for $T$. This definition is equivalent to that expressed in terms of maximum modulus.

\begin{proposition}\label{mini}In the de Branges model of $T\in \mathfrak{E}_1(H)$, each element $f\in \mathcal{H}_e$ satisfies
\[T_f(r)\leq 2T_e(r)+O(\ln r).\]
\end{proposition}
\begin{proof}W.l.g., we can assume that $f\in \mathcal{H}_e$ is a unit vector in $\mathcal{H}_e$. Let $s$ be the section of $F$ associated to the constant function 1. Then $f(\lambda)s(\lambda)\in \mathfrak{H}$ and
\[|f(\lambda)||s(\lambda)|\leq 1\]
by Lemma \ref{lem}. From the proof of Prop.~\ref{0-entire}, we know this is precisely
\[|f(\lambda)|\leq (\frac{|e(\lambda)|^2-|e^\sharp(\lambda)|^2}{2\Im \lambda})^{1/2}.\]
Thus for $\lambda\in \mathbb{C}_+$
\[\ln^+|f(\lambda)|\leq \ln^+|e(\lambda)|+\frac{1}{2}\ln^+(\frac{1-|B(\lambda)|^2}{2\Im \lambda})\leq\ln^+|e(\lambda)|+\frac{1}{2}\ln^+(\frac{1}{2\Im \lambda})\]
while for $\lambda\in \mathbb{C}_-$
\[\ln^+|f(\lambda)|\leq \ln^+|e^\sharp(\lambda)|+\frac{1}{2}\ln^+(\frac{1-\frac{1}{|B(\lambda)|^2}}{-2\Im \lambda})\leq \ln^+|e^\sharp(\lambda)|+\frac{1}{2}\ln^+(\frac{1}{-2\Im \lambda}).\]
Therefore,
\begin{eqnarray*}T_f(r)&=&\int_0^\pi \ln^+|f(re^{\mathrm{i}\phi})|\frac{d\phi}{2\pi}+\int_\pi^{2\pi} \ln^+|f(re^{\mathrm{i}\phi})|\frac{d\phi}{2\pi}\\
&\leq&\int_0^\pi \ln^+|e(re^{\mathrm{i}\phi})|\frac{d\phi}{2\pi}+\int_\pi^{2\pi} \ln^+|e^\sharp(re^{\mathrm{i}\phi})|\frac{d\phi}{2\pi}+O(\ln r)\\
&=&2\int_0^\pi \ln^+|e(re^{\mathrm{i}\phi})|\frac{d\phi}{2\pi}+O(\ln r).
\end{eqnarray*}
The result then follows.
\end{proof}
\emph{Remark}. In the case that the growth is not very fast, \cite[Thm.~3.4]{kaltenbackbranges} contains finer information on the growth aspect of $\mathfrak{H}$.

According to \cite{silva2012class}, $T\in \mathfrak{E}_1(H)$ whose associated de Branges space $\mathcal{H}_e$ has an element vanishing nowhere on $\mathbb{C}$ is called $0$-entire. This kind of entire operators was first introduced by M. G. Krein. This notion can be extended to cover more operators in $\mathfrak{E}_1(H)$. Now let $T\in \mathfrak{E}_1(H)$ be given and $\mathrm{P}_n$ the space of polynomials in $\lambda$ with degree $\leq n$. Then for each $p\in \mathrm{P}_n$ and $f\in \mathfrak{H}$, $p\cdot f$ is again a holomorphic section of the characteristic line bundle $F$, but may not belong to $\mathfrak{H}$ any more. We use $\mathrm{P}_n(\mathfrak{H})$ to denote the linear space of all sections of $F$ that are a finite linear combination of these $p\cdot f$'s. Elements in $\mathrm{P}_n(\mathfrak{H})$ are called $n$-associated sections of $F$. The following definition is a geometric reformulation of Def.~6 in \cite{silva2013branges}.
\begin{definition}$T\in \mathfrak{E}_1(H)$ is called $n$-entire if there is an $s\in \mathrm{P}_n(\mathfrak{H})$, non-vanishing everywhere on $\mathbb{C}$.
\end{definition}
\emph{Remark}. In \cite{silva2013branges}, the notion of $n$-entire operators was introduced in terms of the de Branges model. Our reformulation is equivalent to that, but has the advantage that it is intrinsic. It is worth mentioning that in the definition the section can be replaced by one with finite zeros \cite[Prop.~3.11]{silva2012class}.  For some examples of $n$-entire operators, see \cite{silva2012class, silva2013branges}.

If $T$ is $n$-entire, it is $m$-entire as well if $m>n$. In \cite{silva2012class} it was observed these $n$-entire operators have not exhausted all possibilities in $\mathfrak{E}_1(H)$. Actually these only comprise a small part of $\mathfrak{E}_1(H)$.  \begin{proposition}If $T\in \mathfrak{E}_1(H)$ is $n$-entire, then the Weyl order $\rho_T\leq 1$.
\end{proposition}
\begin{proof}Let $B(\lambda)$ be a contractive Weyl function of $T$ and $\mathcal{H}_e$ the de Branges space associated to a de Branges function $e(\lambda)$ for $B(\lambda)$. Then $\mathrm{P}_n(\mathfrak{H})$ can be identified with $\mathcal{H}_{\tilde{e}}$ where $\tilde{e}(\lambda)=(\lambda+\mathrm{i})^ne(\lambda)$ \cite[Remark 2.9]{silva2012class} \cite[Coro.~3.4]{langer2002characterization}. Thus the assumption means there is an $f\in \mathcal{H}_{\tilde{e}}$ such that $f$ is non-vanishing everywhere. As in the proof of Prop.~\ref{0-entire}, by choosing $e(\lambda)$ properly we can assume that $f(\lambda)\equiv 1$. This leads to the conclusion that $\tilde{e}(\lambda)$ is at most of finite exponential type. Let
\[\tilde{B}(\lambda):=(\frac{\lambda-\mathrm{i}}{\lambda+\mathrm{i}})^n \times B(\lambda).\]
Then by Thm.~\ref{equal}, Prop.~\ref{first} and the formula (\ref{height}) we see that there is a constant $K>0$ such that
\[\frac{1}{2\pi \mathrm{i}}\int_0^r\frac{dt}{t}\int_{-t}^td\ln \tilde{B}(u)< Kr.\]
On the other side,
\begin{eqnarray*}-\mathrm{i}\int_{-t}^td\ln \tilde{B}(u)&=&-n\times \mathrm{i}\int_{-t}^td\ln(\frac{u-\mathrm{i}}{u+\mathrm{i}})-\mathrm{i}\int_{-t}^td\ln B(u)\\
&=&4n\arctan t-\mathrm{i}\int_{-t}^td\ln B(u),\end{eqnarray*}
implying that
\[\frac{1}{2\pi \mathrm{i}}\int_0^r\frac{dt}{t}\int_{-t}^td\ln B(u)< Kr.\]
The claim then follows from the formula (\ref{height}).
\end{proof}
\emph{Remark}. There is a characterization of $n$-entire operators in terms of the spectra of their self-adjoint extensions \cite[Thm.~2.7]{silva2012class}. (C2) in \cite[Thm.~2.7]{silva2012class} implies that an $n$-entire operator $T$ has a self-adjoint boundary condition $y$ whose Weyl exponent $\alpha_T(y)\leq 1$ (see \cite[Def.~10.13]{wang2024complex}). Then by \cite[Coro.~10.45]{wang2024complex}, $\rho_T\leq 1$. Our above proof of this fact is just to show directly that the existence of a non-vanishing section $s\in \mathrm{P}_n(\mathfrak{H})$ is actually a severe restriction on the growth of $T$.
\section{Indeterminate Hamburger moment problems revisited}\label{Hmoment}
\subsection{Basics on Hamburger moment problems}Given a sequence of real numbers $\{s_i\}_{i=0}^\infty$, the classical Hamburger moment problem is to look for a positive Borel measure $\mu$ on the real line such that all its $i$-th moments exist and
\[s_i=\int_{-\infty}^\infty x^id\mu(x),\quad i\in \mathbb{N}_0=\mathbb{N}\cup \{0\}.\]
We only consider the case that $\mu$ has an \emph{infinite} support.
Then the problem is solvable if and only if the quadratic form
$\sum_{i,j=0}^ns_{i+j}\xi_i\overline{\xi_j}$ ($\xi_i\in \mathbb{C}$) is positive-definite for each $n\in \mathbb{N}$ \cite[Thm.~16.1]{schmudgen2012unbounded}. If this is the case, we say $\{s_i\}_{i=0}^\infty$ is positive-definite. There are several approaches towards the Hamburger moment problem and the standard reference is \cite{akhiezer2020classical}. However, our focus here is the operator-theoretic content of this problem. For this respect, \cite{schmudgen2012unbounded} also contains a concise introduction.

Given a positive-definite moment sequence $s=\{s_i\}_{i=0}^\infty$, if the corresponding moment problem has a unique solution, it is called \emph{determinate}, otherwise \emph{indeterminate}. We only consider the indeterminate case and adopt the convention that $s_0=1$. $s$ determines an inner product on the linear space $P$ of polynomials with complex coefficients, i.e.,
\[(p(\cdot), q(\cdot))=\sum_{i,j=0}^ns_{i+j}\xi_i\overline{\eta_j}\]
if $p(\lambda)=\sum_{i=0}^n \xi_i \lambda^i$ and $q(\lambda)=\sum_{i=0}^n\eta_i\lambda^i$. We denote the completion of $P$ w.r.t. this inner product by $H$. We can apply the Gram-Schmidt procedure to the basis $\{1, \lambda, \lambda^2,\cdots\}$ to obtain the so-called orthonormal polynomials $\{P_n\}$ of the first kind. It is assumed that the coefficient of the highest order term of $P_n$ is positive. In particular, $P_0(\lambda)\equiv 1$ and each coefficient of $P_n$ is real. Obviously, multiplication by the independent variable $\lambda$ maps $P$ to $P$ and thus by taking closure we get a symmetric operator $M_\lambda$ in $H$. Then we have the \emph{three term recurrence relation}
\begin{equation}M_\lambda P_n=\lambda P_n(\lambda)=a_nP_{n+1}(\lambda)+b_nP_n(\lambda)+a_{n-1}P_{n-1}(\lambda),\label{recurrence}\end{equation}
where $a_n>0, b_n$ are real numbers and it is assumed that $a_{-1}=1$, $P_{-1}(\lambda)\equiv 0$. It is known that $M_\lambda$ is densely defined and $M_\lambda\in \mathfrak{E}_1(H)$.

Let $\{e_n\}_{n=0}^\infty$ be the standard orthonormal basis in the Hilbert space $l^2(\mathbb{N}_0)$ given by $e_n=(\delta_{in})$. By setting $UP_n=e_n$ we get a unitary map from $H$ to $l^2(\mathbb{N}_0)$ and $M_\lambda$ is transformed into the \emph{Jacobi operator} $T$:
\[Te_n=a_n e_{n+1}+b_ne_n+a_{n-1}e_{n-1},\]
where $e_{-1}:=0$. We can define a linear map $\mathcal{T}$ on the linear space of complex sequences $\{\gamma_n\}_{n=0}^\infty$ by
\[(\mathcal{T}\gamma)_n=a_n\gamma_{n+1}+b_n\gamma_n+a_{n-1}\gamma_{n-1},\quad n\in \mathbb{N}_0,\]
where $\gamma_{-1}:=0$. It can be shown that
\[D(T^*)=\{\gamma\in l^2(\mathbb{N}_0)|\mathcal{T}\gamma\in l^2(\mathbb{N}_0)\}\]
and $T^*\gamma=\mathcal{T}\gamma$ for $\gamma\in D(T^*)$ \cite[Prop.~16.7]{schmudgen2012unbounded}.

For a fixed $\lambda\in \mathbb{C}$, consider the equation $T^*\gamma=\lambda \gamma$ (a difference equation) with the initial conditions $\gamma_{-1}=0$ and $\gamma_0=1$. It is easy to find that there is a \emph{unique} solution $\mathfrak{p}(\lambda)=\{\gamma_n(\lambda)\}_{n=0}^\infty$, where $\gamma_n(\lambda)$ is precisely $P_n(\lambda)$. This is just another way to interpret the occurrence of polynomials $P_n(\lambda)$. The difference equation $T^*\gamma =e_0+\lambda \gamma$ with the initial conditions $\gamma_{-1}=-1$ and $\gamma_0=0$ is also often used and the unique solution $\mathfrak{q}(\lambda)=\{Q_n(\lambda)\}_{n=0}^\infty$ is actually a sequence of polynomials with real coefficients. In particular $Q_0(\lambda)\equiv 0$. These are called polynomials of the second kind associated to the moment sequence $s$.

What interests us most here is the following four entire functions:
\[a(\lambda):=\lambda\sum_{n=0}^\infty Q_n(0)Q_n(\lambda),\quad b(\lambda):=-1+\lambda\sum_{n=0}^\infty Q_n(0)P_n(\lambda),\]
\[c(\lambda):=1+\lambda\sum_{n=0}^\infty P_n(0)Q_n(\lambda),\quad d(\lambda):=\lambda\sum_{n=0}^\infty P_n(0)P_n(\lambda).\]
 The matrix $N(\lambda):=\left(
               \begin{array}{cc}
                 a(\lambda) & b(\lambda) \\
                 c(\lambda) & d(\lambda) \\
               \end{array}
             \right)
 $ satisfies $\det N(\lambda)\equiv 1$, and is called the \emph{Nevanlinna matrix}, which was used by Nevanlinna to parameterize all solutions of the moment problem associated to $s$.
 It was noted very early by M. Riesz that all these entire functions $a(\lambda), b(\lambda), c(\lambda), d(\lambda)$ are of at most minimal exponential type. Later in \cite{berg1994order}, it was proved that all these functions share the same order $\rho$ and type $\tau$. The numbers $\rho$ and $\tau$ now are called the order and type of the moment problem. If the growth is much slower, finer notions like logarithmic order and type are also considered \cite{berg2007logarithmic, berg2014order}.
\subsection{On the growth of indeterminate Hamburger moment problems}

 Our focus in this subsection is also the growth aspect of indeterminate Hamburger moment problems. However, compared with investigations in the literature on this topic, rather than just the rough quantities order and type, we have the much finer object $h_T(r)$ or $\mathrm{T}_F(r)$ associated to the Jacobi operator $T$ to measure the growth of the moment problem \emph{canonically}. Thus our basic task here is to interpret the growth of the moment problem in the more general formalism developed in previous sections.

  \begin{lemma}\label{bt}$D(T^*)=D(T)\oplus \textup{span}\{\mathfrak{q}(0), \mathfrak{p}(0)\}$. For $\gamma=\beta+c_0\mathfrak{q}(0)+c_1\mathfrak{p}(0)$ where $\beta\in D(T)$ and $c_0, c_1\in \mathbb{C}$, set $\Gamma_0\gamma=c_0$ and $\Gamma_1\gamma=-c_1$. Then $(\mathbb{C}, \Gamma_0, \Gamma_1)$ is a boundary triplet whose corresponding Weyl function is $M(\lambda)=b(\lambda)/d(\lambda)$.
 \end{lemma}
 \begin{proof}That $(\mathbb{C}, \Gamma_0, \Gamma_1)$ is a boundary triplet is actually Prop.~16.26 in \cite{schmudgen2012unbounded}. We only prove the claim on the Weyl function. Let $\mathfrak{p}(\lambda)=r+c_0 \mathfrak{q}(0)+c_1 \mathfrak{p}(0)$ where $r\in D(T)$. Then by definition,
 \[\lambda \mathfrak{p}(\lambda)=T r+c_0 e_0. \]
 Thus
\[\lambda (\mathfrak{p}(\lambda), \mathfrak{p}(0))_{l^2}=(Tr, \mathfrak{p}(0))_{l^2}+c_0=(r, T^*\mathfrak{p}(0))_{l^2}+c_0=c_0.\]
This is precisely $c_0=d(\lambda)$. On the other side,
\[1=(\mathfrak{p}(\lambda),e_0)_{l^2}=(r,e_0)_{l^2}+c_1 \]
and
\[\lambda (\mathfrak{p}(\lambda), \mathfrak{q}(0))_{l^2}=(Tr, \mathfrak{q}(0))_{l^2}=(r, T^*\mathfrak{q}(0))_{l^2}=(r, e_0)_{l^2}.\]
Hence,
\[c_1=1-(r,e_0)_{l^2}=1-\lambda (\mathfrak{p}(\lambda), \mathfrak{q}(0))_{l^2},\]
which is precisely $c_1=-b(\lambda)$. Then $M(\lambda)=b(\lambda)/d(\lambda)$ by definition.
\end{proof}
\emph{Remark}. The corresponding contractive Weyl function is thus
\[B(\lambda)=\frac{M(\lambda)-\mathrm{i}}{M(\lambda)+\mathrm{i}}=\frac{b(\lambda)-\mathrm{i}d(\lambda)}{b(\lambda)+\mathrm{i}d(\lambda)}.\]
Due to the fact $\det N(\lambda)\equiv 1$, $b(\lambda)$ and $d(\lambda)$ have no common zeros. Then $\sigma(T_1)$ (resp. $\sigma(T_{-1})$) is precisely the zero set of $d(\lambda)$ (resp. $b(\lambda)$). Additionally by definition $b(\lambda)=b^\sharp(\lambda)$ and $d(\lambda)=d^\sharp(\lambda)$. Thus $e(\lambda):=[b(\lambda)+\mathrm{i}d(\lambda)]/\sqrt{2}$ is a de Branges function for $B(\lambda)$ and $\mathfrak{p}(\lambda)$ in this context is actually $\zeta(\lambda)$ in \S\S~\ref{deBranges}. Note also that
$(e_n, \mathfrak{p}(\bar{\lambda}))_{l^2}=P_n(\lambda)$. Thus we can see $H=\mathcal{H}_e$. In particular, $\widehat{e_0}\in \mathfrak{H}$ is a holomorphic non-vanishing section of $F$ over $\mathbb{C}$.

An easy computation shows
\[|\widehat{e_0}(\lambda)|^2=\frac{1}{(\mathfrak{p}(\bar{\lambda}), \mathfrak{p}(\bar{\lambda}))_{l^2}}=\frac{1}{\sum_{n=0}^\infty |P_n(\lambda)|^2}.\]
Denote $(\sum_{n=0}^\infty |P_n(\lambda)|^2)^{1/2}$ by $P(\lambda)$. Then we immediately have
\begin{theorem}\label{moment1}
The first Chern form $c_1(F)$ for the characteristic line bundle $F$ of the Jacobi operator $T$ is $c_1(F)=2dd^c\ln P(\lambda)$. In particular,
\[\mathrm{T}_F(r)=\frac{1}{\pi}\int_0^\pi\ln P(re^{\mathrm{i}\phi})d\phi-\ln P(0).\]
\end{theorem}
\begin{proof}The first claim is clear from our general theory, and the second follows from the proof of Prop.~\ref{first} and the fact that $P(\lambda)=P(\bar{\lambda})$.
\end{proof}
Even without our geometric formalism, the importance of $P(\lambda)$ in measuring the growth of an indeterminate Hamburger moment problem was already known: M. Riesz proved that $P(\lambda)$ is of at most minimal exponential type, i.e., for arbitrary $\varepsilon>0$, there is a constant $C_\varepsilon>0$ such that
\[P(\lambda)\leq C_\varepsilon e^{\varepsilon|\lambda|}, \quad \forall \lambda\in \mathbb{C}.\] See \cite[Thm.~2.4.3]{akhiezer2020classical}. The following corollary is immediate from the above theorem.
\begin{corollary}The Jacobi operator $T$ associated to an indeterminate Hamburger moment problem is of at most minimal exponential type, i.e., $\rho_T\leq 1$ and if $\rho_T=1$, then $\tau_T=0$.
\end{corollary}

$P(\lambda)$ also played an essential part in \cite{berg1994order, berg2014order}. In \cite{berg1994order} it was also proved that $\ln P(\lambda)$ is a subharmonic function. For us this actually comes from the general result that $F$ is positive-definite, i.e., $c_1(F)$ is positive-definite. Our main new observation here is that the precise quantity measuring the growth is actually $\frac{1}{\pi}\int_0^\pi\ln P(re^{\mathrm{i}\phi})d\phi$ and $\mathrm{T}_F(r)$ is essentially its geometric interpretation. It's thus reasonable to call $\mathrm{T}_F(r)$ the \emph{characteristic function of the indeterminate Hamburger moment problem}. It should be noted that $\mathrm{T}_F(r)$ only depends on the unitary equivalence class of $T$.

 Now it is natural to compare $\mathrm{T}_F(r)$ for the characteristic line bundle $F$ with the Nevanlinna characteristic functions of the four entire functions $a(\lambda), b(\lambda), c(\lambda)$ and $d(\lambda)$. Recall that an entire function version of First Main Theorem states that
\begin{equation}T_f(r)=\int_0^{2\pi}\ln^+|\frac{1}{f(re^{\mathrm{i}\phi})}|\frac{d\phi}{2\pi}+N_f(r,0)+O(1),\label{fun}\end{equation}
where $N_f(r,0)$ is the zero counting function of $f$. See for example \cite[Chap.~VI]{lang2013introduction}.
\begin{theorem}\label{moment2}Let $f(\lambda)$ be any one of $a(\lambda), b(\lambda), c(\lambda)$ and $d(\lambda)$. Then
\[\mathrm{T}_F(r)=T_f(r)+O(\ln r).\]
\end{theorem}
\begin{proof}Let $f(\lambda)$ be $d(\lambda)$. Since the zero set of $d(\lambda)$ is actually the spectrum of a self-adjoint extension of $T$, by the formula (\ref{se}) we see that due to Thm.~\ref{equal},
\[\mathrm{T}_F(r)=h_T(r)+O(\ln r)=N_d(r,0)+O(\ln r)\leq T_d(r)+O(\ln r),\]
where Eq.~(\ref{fun}) is used. On the other side, it is clear that by the Cauchy-Schwarz inequality
\[|d(\lambda)|\leq |\lambda|\times P(0)\times P(\lambda).\]
Thus
\[\ln^+|d(\lambda)|\leq \ln^+(|\lambda|\times P(0)\times P(\lambda))\leq \ln^+|\lambda|+\ln P(0)+\ln P(\lambda).\]
Note that in the last inequality we have used the fact $P_0(\lambda)\equiv 1$ and thus $P(\lambda)\geq 1$ for any $\lambda$.
Since $d(\lambda)$ is real, we see that
\[T_d(r)=\frac{1}{\pi}\int_0^\pi \ln^+|d(re^{\mathrm{i}\phi})|d\phi\leq \frac{1}{\pi}\int_0^\pi \ln P(re^{\mathrm{i}\phi})d\phi+O(\ln r).\]
Combining the above argument with Prop.~\ref{moment1} leads to the claim for $d(\lambda)$.

 Let $f(\lambda)$ be $b(\lambda)$. On one side, since the zeros of $d(\lambda)$ and $b(\lambda)$ are pairwise interlaced, we must have $N_b(r,0)=N_d(r,0)+O(\ln r)$. On the other side, we also have
\[|b(\lambda)|\leq 1+k\times |\lambda| \times P(\lambda)\]
where $k=(\Sigma_{n=0}^\infty |Q_n(0)|^2)^{1/2}$ is a constant. Then we can prove the claim for $b(\lambda)$ similarly.

The proof for $a(\lambda)$ and $c(\lambda)$ is slightly more subtle. From \cite[Prop.~2.2, Prop.~2.3]{berg1994order}, we know that for any $\lambda=u+\mathrm{i}v$ where $u, v\in \mathbb{R}$ and $v\neq 0$,
\[|a(\lambda)|\leq \frac{1}{|v|}|b(\lambda)|,\quad |c(\lambda)|\leq \frac{1}{|v|}|d(\lambda)|, \]
and
\[|b(\lambda)|\leq(\frac{a_0^2}{|v|}+|\lambda-b_0|)|a(\lambda)|,\quad |d(\lambda)|\leq(\frac{a_0^2}{|v|}+|\lambda-b_0|)|c(\lambda)|.\]
From those inequalities concerning $c(\lambda)$ we have
\[\ln ^+|c(\lambda)|\leq \ln^+ |d(\lambda)|+\ln^+\frac{1}{|v|}\]
and
\[\ln^+(\frac{a_0^2}{|v|}+|\lambda-b_0|)+\ln ^+|c(\lambda)|\geq \ln^+|d(\lambda)|.\]
Note also that
\[\ln^+(\frac{a_0^2}{|v|}+|\lambda-b_0|)\leq \ln^+\frac{1}{|v|}+\ln^+|\lambda-b_0|+\ln^+a_0^2+\ln 2\]
and for $\lambda=re^{\mathrm{i}\phi}$ where $\phi\in (0,\pi)$,
\[\ln^+\frac{1}{|v|}\leq \ln^+\frac{1}{r}-\ln\sin \phi.\]
 From these we can conclude that $T_c(r)=T_d(r)+O(\ln r)$ and our claim on $c(\lambda)$ follows from that on $d(\lambda)$. The result on $a(\lambda)$ follows in a similar way.
 \end{proof}
 Thus via this theorem, we have provided an explanation for why $a(\lambda), b(\lambda), c(\lambda)$ and $d(\lambda)$ all share the same growth property: Their Nevanlinna characteristic functions are equal to the same $\mathrm{T}_F(r)$ up to the term $O(\ln r)$ and the latter provides a canonical quantity measuring the growth of indeterminate moment problems. In particular, this demonstrates clearly how the distribution of eigenvalues of a self-adjoint extension is related to the growth of the moment problem.

There is another way to interpret polynomials of the second kind \cite[\S~2]{berg1994order}. Replacing the sequences $\{a_n\}$ and $\{b_n\}$ by the shifted sequences
\[\tilde{a}_n=a_{n+1},\quad \tilde{b}_n=b_{n+1},\]
one can obtain the new polynomials $\tilde{P}_n(\lambda)$ and $\tilde{Q}_n(\lambda)$. In particular, $\tilde{P}_n(\lambda)=a_0Q_{n+1}(\lambda)$. There is a new underlying indeterminate Hamburger moment problem. Let $\tilde{a}(\lambda), \tilde{b}(\lambda), \tilde{c}(\lambda), \tilde{d}(\lambda)$ be the corresponding Nevanlinna functions. Then a result of Pedersen \cite[Prop.~2.1]{pedersen1994nevanlinna} \cite{berg1994order} shows
\[a(\lambda)=\frac{1}{a_0^2}\tilde{d}(\lambda),\quad c(\lambda)=-\frac{b_0}{a_0^2}\tilde{d}(\lambda)-\tilde{b}(\lambda).\]
Then our result certainly implies that the shifted Hamburger moment problem has essentially the same characteristic function $\mathrm{T}_F(r)$.

Another question we are interested in is whether the Jacobi operator $T$ is determined by its Weyl class, in other words, whether $T$ is the only Jacobi operator in the unitary equivalence class of $T$. In this respect, we have

\begin{theorem}\label{unique}Let $T$ be the Jacobi operator associated to a given indeterminate Hamburger moment sequence $s$. Then in the unitary equivalence class of $T$, $T$ is the only Jacobi operator associated to a certain moment sequence.
\end{theorem}
\begin{proof}If two Jacobi operators $T_1$ and $T_2$ are unitarily equivalent, then they share the same Weyl class. Let $B_1(\lambda)$ and $B_2(\lambda)$ be the contractive Weyl functions w.r.t. the boundary triplets in Lemma \ref{bt} respectively. Then
\[B_2(\lambda)=e^{\mathrm{i}\gamma}\times \frac{B_1(\lambda)-c}{1-\bar{c}B(\lambda)}\]
for some constants $e^{\mathrm{i}\gamma}$ and $c$ ($|c|<1$). Since $B_1(0)=B_2(0)=1$, we must have $e^{\mathrm{i}\gamma}=\frac{1-\bar{c}}{1-c}$. Let $e_1(\lambda)=(b(\lambda)+\mathrm{i}d(\lambda))/\sqrt{2}$ be the de Branges function of $B_1(\lambda)$. Then it is easy to see that the de Branges function of $B_2(\lambda)$ is\footnote{Note that by construction we should have $e_1(0)=e_2(0)=-\frac{1}{\sqrt{2}}$.}
\[e_2(\lambda)=\frac{e_1(\lambda)-\bar{c}e_1^\sharp(\lambda)}{1-\bar{c}}.\]
As a result, the corresponding reproducing kernels $\mathrm{K}_1(\lambda, \mu)$ of $\mathcal{H}_{e_1}$ and $\mathrm{K}_2(\lambda, \mu)$ of $\mathcal{H}_{e_2}$ are related by (see \S\S~\ref{deBranges})
\[\mathrm{K}_2(\lambda, \mu)=\frac{1-|c|^2}{|1-c|^2}\mathrm{K}_1(\lambda, \mu),\]
implying that $\mathcal{H}_{e_1}=\mathcal{H}_{e_2}$ as function spaces and that the map $U:\mathcal{H}_{e_1}\rightarrow \mathcal{H}_{e_2}$, $f\mapsto \frac{\sqrt{1-|c|^2}}{|1-c|}f$ is unitary. In particular,
\[1=(1,1)_{\mathcal{H}_{e_1}}=\frac{1-|c|^2}{|1-c|^2}\times(1,1)_{\mathcal{H}_{e_2}}=\frac{1-|c|^2}{|1-c|^2}.\]
Since $U$ commutes with the multiplication operator $M_\lambda$ and $s_i=((M_\lambda)^i1,1)_{\mathcal{H}_{e_1}}$, we see immediately that $T_1$ and $T_2$ share the same moment sequence $s$ and consequently $T_1=T_2$ as required.
\end{proof}

In the rest of this subsection, we demonstrate how the curvature function $\omega(u)$ of the Jacobi operator $T$ is related to the polynomials $P_n(\lambda)$ and construct two linearly independent  solutions to the equation $y''+3\omega(u)y=0$.
\begin{lemma}The function $|M'(\lambda)|/\Im{M(\lambda)}$ on $\mathbb{C}_+$ is a complete unitary invariant of the Jacobi operator $T$. In terms of polynomials of the first kind
\[\frac{|M'(\lambda)|}{\Im M(\lambda)/\Im \lambda}=\frac{|\sum_{n=0}^\infty P_n^2(\lambda)|}{P^2(\lambda)}.\]
\end{lemma}
\begin{proof}The first claim is clear from Eq.~(\ref{cur2}). From Eq.~(\ref{herg}) we find
\[M'(\lambda)=(\gamma(\lambda),\gamma(\bar{\lambda}))_{l^2},\quad \frac{\Im{M(\lambda)}}{\Im \lambda}=(\gamma(\lambda), \gamma(\lambda))_{l^2}\]
where $\gamma(\lambda)=\mathfrak{p}(\lambda)/d(\lambda)$ by definition. The second claim then follows.
\end{proof}
\begin{proposition}In terms of polynomials of the first kind, the curvature function
\begin{eqnarray*}\omega(u)&=&\frac{\sum_{m,n=0}^\infty[P_n(u)P_m'(u)-P_n'(u)P_m(u)]^2}{2P^4(u)}\\
&=&\frac{P^2(u)\cdot \sum_{n=0}^\infty P_n'^2(u)-(\sum_{n=0}^\infty P_n(u)P_n'(u))^2}{P^4(u)},\quad \forall u\in \mathbb{R}.\end{eqnarray*}
In particular, \[\omega(u)\leq \frac{\sum_{n=0}^\infty P_n'^2(u)}{P^2(u)}.\]
\end{proposition}
\begin{proof}Let
\begin{eqnarray*}R(\lambda):&=&(\sum_{n=0}^\infty P_n(\lambda)P_n(\bar{\lambda}))^2-(\sum_{n=0}^\infty P_n^2(\lambda))\cdot (\sum_{n=0}^\infty P_n^2(\bar{\lambda}))\\
&=&\sum_{m,n=0}^\infty P_n(\lambda)P_m(\bar{\lambda})[P_n(\bar{\lambda})P_m(\lambda)-P_n(\lambda)P_m(\bar{\lambda})]\\
&=&2\mathrm{i}\sum_{m,n=0}^\infty P_n(\lambda)P_m(\bar{\lambda})\Im[P_n(\bar{\lambda})P_m(\lambda)].\end{eqnarray*}
Then clearly
\[\omega(u)=\lim_{\varepsilon\rightarrow 0+}\frac{R(u+\mathrm{i}\varepsilon)}{4\varepsilon^2 P^4(u+\mathrm{i}\varepsilon)}.\]
The first equality follows when one takes the second-order term of $\varepsilon$ out of $R(u+\mathrm{i}\varepsilon)$. As for the second, we note that in terms of the reproducing kernel $\mathrm{K}_1$ in the de Branges model
\[\mathrm{K}_1(\lambda, \bar{\mu})=\sum_{n=0}^\infty P_n(\lambda)P_n(\mu),\]
which converges absolutely and uniformly on any compact subset of $\mathbb{C}^2$. Thus
\[\frac{\partial \mathrm{K}_1(\lambda, \bar{\mu})}{\partial \mu}=\sum_{n=0}^\infty P_n(\lambda)P_n'(\mu),\quad \frac{\partial^2 \mathrm{K}_1(\lambda, \bar{\mu})}{\partial \lambda\partial \mu}=\sum_{n=0}^\infty P_n'(\lambda)P_n'(\mu),\]
from which the second equality follows. The last inequality is obvious.
\end{proof}
\emph{Remark}. Let $\mathfrak{r}(u):=\{P_n'(u)\}_{n=0}^\infty\in l^2(\mathbb{N}_0)$ and $\alpha(u)$ be the angle between $\mathfrak{p}(u)$ and $\mathfrak{r}(u)$. Then
\[\omega(u)=(\frac{\|\mathfrak{r}(u)\|_{l^2}}{\|\mathfrak{p}(u)\|_{l^2}})^2(1-\cos^2 \alpha(u))=(\frac{\|\mathfrak{r}(u)\|_{l^2}}{\|\mathfrak{p}(u)\|_{l^2}}\sin \alpha(u))^2.\]
Thus $\omega(u)$ measures how different the two vectors $\mathfrak{p}(u)$ and $\mathfrak{r}(u)$ are. Since $\omega(u)$ determines the unitary equivalence class of $T$, it's interesting to know how the moment sequence $s$ can be recovered from $\omega(u)$. We have no answer to this question up to now.

\begin{proposition}The equation $y''+3\omega(u)y=0$ has the following two real-valued linearly independent solutions:
\[y_1(u)=\frac{b(u)}{[\det\left(
                           \begin{array}{cc}
                             d(u) & b(u) \\
                             d'(u) & b'(u) \\
                           \end{array}
                         \right)]^{1/2}
},\quad y_2(u)=\frac{d(u)}{[\det\left(
                           \begin{array}{cc}
                             d(u) & b(u) \\
                             d'(u) & b'(u) \\
                           \end{array}
                         \right)]^{1/2}
}.\]
\begin{proof}We just look for $y_1$ and $y_2$ such that $M(u)=y_1(u)/y_2(u)$. Since the zeros of $y_1(u)$ (resp.~$y_2(u)$) are just zeros of $b(u)$ (resp.~$d(u)$), we can assume that $y_1(u)=f(u)b(u)$ (resp.~$y_2(u)=f(u)d(u)$), where $f(u)$ is a real-valued function without zeros on $\mathbb{R}$. Due to the famous Liouville's formula, we know that
\[\det \left(
    \begin{array}{cc}
      f(u)d(u) & f(u)b(u) \\
       f'(u)d(u)+f(u)d'(u)& f'(u)b(u)+f(u)b'(u) \\
    \end{array}
  \right)
\]
must be a constant. This means
\[f^2(u)\det\left(
                           \begin{array}{cc}
                             d(u) & b(u) \\
                             d'(u) & b'(u) \\
                           \end{array}
                         \right)\]
should be a constant. By setting $u=0$, we can see this constant must be positive and by rescaling we can assume it to be 1. The claim then follows.
\end{proof}
\end{proposition}
\subsection{Non-self-adjoint extensions}
In the operator-theoretic approach to an indeterminate Hamburger moment problem, non-self-adjoint extensions of the underlying Jacobi operator are seldom considered and little is known about them. Here we just mention a few facts on this topic that can be derived from our investigation.

Among all solutions to an indeterminate Hamburger moment problem, those provided by self-adjoint extensions of the Jacobi operator $T$ are called the von Neumann solutions: For each $t\in \mathbb{R}\cup \{\infty\}$, there is a discrete measure $\mu_t$ on $\mathbb{R}$ such that
\begin{equation}
I_t(\lambda):=\int_\mathbb{R}\frac{1}{u-\lambda}d\mu_t(u)=-\frac{a(\lambda)+tc(\lambda)}{b(\lambda)+td(\lambda)}.\label{Borel}
\end{equation}
The measure $\mu_t$ corresponds to a self-adjoint extension $T_{(t)}$ such that $I_t(\lambda)=((T_{(t)}-\lambda)^{-1}e_0,e_0)_{l^2}$.
\begin{lemma}In terms of the boundary triplet in Lemma \ref{bt}, the above self-adjoint extension $T_{(t)}$ is $T_c$ where $c=\frac{t+\mathrm{i}}{t-\mathrm{i}}$.
\end{lemma}
\begin{proof}It suffices to prove that $\sigma(T_{(t)})=\sigma(T_c)$. $\sigma(T_{(t)})$ is precisely the zero set of $b(\lambda)+td(\lambda)$ while $\sigma(T_c)$ is given by the roots of
\[B(\lambda)-c=\frac{b(\lambda)-\mathrm{i}d(\lambda)}{b(\lambda)+\mathrm{i}d(\lambda)}-\frac{t+\mathrm{i}}{t-\mathrm{i}}=0.\]
It is easy to see the latter equation is equivalent to $b(\lambda)+td(\lambda)=0$.
\end{proof}
We note that the RHS of Eq.~(\ref{Borel}) still makes sense for $t\in \mathbb{C}\backslash\mathbb{R}$, but we cannot expect a new measure $\mu_t$ such that Eq.~(\ref{Borel}) continues to hold in this extended case. However, as one may hope from the above lemma, for $t\in \mathbb{C}\backslash\mathbb{R}$ the RHS of Eq.~(\ref{Borel}) really comes from a non-self-adjoint extension.
\begin{proposition}\label{I}In terms of the boundary triplet in Lemma \ref{bt}, for $c=\frac{t+\mathrm{i}}{t-\mathrm{i}}$ where $t\in \mathbb{C}\backslash\mathbb{R}$
\[I_t(\lambda):=((T_{c}-\lambda)^{-1}e_0,e_0)_{l^2}=-\frac{a(\lambda)+tc(\lambda)}{b(\lambda)+td(\lambda)}.\]
\end{proposition}
\begin{proof}We can use the Krein type formula (8.3) in \cite{wang2024complex}. In the present setting, it leads to
\[((T_c-\lambda)^{-1}e_0,e_0)_{l^2}=((T_\infty-\lambda)^{-1}e_0,e_0)_{l^2}+\mathrm{i}(c-B(\lambda))^{-1}\times(e_0, \gamma_-(\bar{\lambda}))_{l^2}\times (\gamma_+(\lambda),e_0)_{l^2}.\]
Therefore, by setting $c=1$ and taking difference we obtain
\[((T_c-\lambda)^{-1}e_0,e_0)_{l^2}=((T_1-\lambda)^{-1}e_0,e_0)_{l^2}+\mathrm{i}\frac{1-c}{(c-B(\lambda))(1-B(\lambda))}\times(e_0, \gamma_-(\bar{\lambda}))_{l^2}\times (\gamma_+(\lambda),e_0)_{l^2}.\]
It can be obtained easily that
\[\gamma_+(\lambda)=\frac{\mathfrak{p}(\lambda)}{e(\lambda)},\quad \gamma_-(\bar{\lambda})=\frac{\mathfrak{p}(\bar{\lambda})}{\overline{e(\lambda)}},\]
where $e(\lambda)=[b(\lambda)+\mathrm{i}d(\lambda)]/\sqrt{2}$ and consequently that
\[(e_0, \gamma_-(\bar{\lambda}))_{l^2}\times (\gamma_+(\lambda),e_0)_{l^2}=\frac{1}{(e(\lambda))^2}.\]
We also know that ($t=\infty$)
\[((T_1-\lambda)^{-1}e_0,e_0)_{l^2}=-\frac{c(\lambda)}{d(\lambda)}.\]
Combining all these together will result in the conclusion.
\end{proof}
\emph{Remark}. The above proposition actually says $((T_c-\lambda)^{-1}e_0,e_0)_{l^2}$ is an anlytic family of meromorphic functions parameterized by the projective line $\mathcal{M}$. When $|c|=1$, from the spectral resolution of $T_c$ we know that the RHS of Eq.~(\ref{Borel}) allows a partial fraction expansion. We don't know if this still holds for the case $|c|\neq 1$.

The following proposition is a generalization of a classical result relating $I_t(\lambda)$ to the moment sequence $s$ \cite[Prop.~16.33 (i)]{schmudgen2012unbounded}.
\begin{proposition}In terms of the boundary triplet in Lemma \ref{bt}, for each $c\in \mathbb{C}$ such that $|c|>1$ ($c=\infty$ is allowed) and each $n\in \mathbb{N}_0$,
\[\lim_{y\rightarrow +\infty}y^{n+1}[((T_c-\mathrm{i}y)^{-1}e_0,e_0)_{l^2}+\sum_{j=0}^n\frac{s_j}{(\mathrm{i}y)^{j+1}}]=0.\]
Similarly, for each $c\in \mathbb{C}$ such that $|c|<1$ and each $n\in \mathbb{N}_0$,
\[\lim_{y\rightarrow -\infty}y^{n+1}[((T_c-\mathrm{i}y)^{-1}e_0,e_0)_{l^2}+\sum_{j=0}^n\frac{s_j}{(\mathrm{i}y)^{j+1}}]=0.\]
\end{proposition}
\begin{proof}We only prove the case $c=\infty$ and other cases can be proved along the same line. Due to \cite[Prop.~16.33 (i)]{schmudgen2012unbounded}, we surely have
\[\lim_{y\rightarrow +\infty}y^{n+1}[((T_1-\mathrm{i}y)^{-1}e_0,e_0)_{l^2}+\sum_{j=0}^n\frac{s_j}{(\mathrm{i}y)^{j+1}}]=0.\]
From the proof of Prop.~\ref{I}, we know
\[((T_1-\lambda)^{-1}e_0,e_0)_{l^2}-((T_\infty-\lambda)^{-1}e_0,e_0)_{l^2}=\frac{\mathrm{i}}{1-B(\lambda)}\times \frac{1}{(e(\lambda))^2}=\frac{\mathrm{1}}{\sqrt{2}d(\lambda)}\times \frac{1}{e(\lambda)}.\]
Thus we only have to prove
\[\lim_{y\rightarrow +\infty}\frac{y^{n+1}}{|d(\mathrm{i}y)||e(\mathrm{i}y)|}=0.\]
Since both $d(\lambda)$ and $e(\lambda)$ are of at most minimal exponential type, they are determined by their zeros up to a constant factor, i.e.,
\[d(\lambda)=c_1\lambda \prod_{j=1}^\infty(1-\frac{\lambda}{\lambda_j}),\quad e(\lambda)=c_2 \prod_{j=1}^\infty(1-\frac{\lambda}{\mu_j}),\]
where $\lambda_j\in \mathbb{R}$, $\mu_j\in \mathbb{C}_-$ and $c_1$, $c_2$ are constants. We can see both $|d(\mathrm{i}y)|$ and $|e(\mathrm{i}y)|$ grow faster than any polynomial in $y$ as $y\rightarrow+\infty$. The claim then follows.
\end{proof}

\begin{proposition}The non-self-adjoint extension $T_c$ with $|c|\neq 1$ is complete.
\end{proposition}
\begin{proof}Since $T$ is of at most minimal exponential type, for each $|c|<1$ the mean type of $T_c$ has to vanish. The result follows from Thm.~\ref{mean} immediately.
\end{proof}
Since $\rho_T\leq 1$, the distribution of eigenvalues of $T_c$ is sparse, but we don't know if $T_c$ is Riesz complete.

\end{document}